\PassOptionsToPackage{reqno}{amsmath}
\documentclass[11pt]{amsart}
\numberwithin{equation}{section}

\allowdisplaybreaks[4]

\usepackage{amsmath}
\usepackage{mathrsfs}
\usepackage{amssymb,graphicx}
\usepackage{enumerate}
\usepackage{tikz}
\usepackage{comment}
\usepackage{bm}
\usepackage{mathtools}
\usepackage{microtype}
\usepackage[margin=1.15in]{geometry}
\usepackage{xcolor}
\usepackage[colorlinks=true, linkcolor=blue, citecolor=red, urlcolor=cyan]{hyperref}
\usepackage{fancyhdr}
\usepackage{amsmath}             
\usepackage{amssymb}           
\usepackage{amsfonts}           
\usepackage{latexsym}           
\usepackage{amsthm}              
\usepackage{tikz}
\usetikzlibrary{decorations.pathreplacing}
\usetikzlibrary{angles}
\usetikzlibrary{calc}
\usepackage{caption}
\usepackage{type1cm}
\usepackage{todonotes}
\usepackage{standalone}
\usepackage{graphicx}
\usepackage{tabularx}
\usepackage{array}

\reversemarginpar

\makeatletter
\newcommand{\labelsymbol}[2]{%
	\phantomsection
	\edef\@currentlabel{#2}
	\label{#1}
	#2
}
\makeatother

\newcolumntype{Y}{>{\centering\arraybackslash}X}

\makeatletter
\newcommand\bigDiamond{\mathop{\mathpalette\bigDi@mond\relax}}
\newcommand\bigDi@mond[2]{%
	\vcenter{\hbox{\m@th
			\scalebox{\ifx#1\displaystyle 2\else1.2\fi}{$#1\Diamond$}%
	}}%
}
\newcommand\bigLozenge{\mathop{\mathpalette\bigL@zenge\relax}}
\newcommand\bigL@zenge[2]{%
	\vcenter{\hbox{\m@th
			\scalebox{\ifx#1\displaystyle 2\else1.2\fi}{$#1\blacklozenge$}%
	}}%
}
\makeatother

\makeatletter

\newcommand{\leftsymbollabel}[2]{%
	\def\@currentlabel{$#2$}
	\makebox[0pt][r]{\ensuremath{#2}\, }
	\label{#1}%
}
\makeatother

\newcommand{\paddedmath}[1]{%
	\raisebox{0pt}%
	[\dimexpr\height+5pt\relax]%
	[\dimexpr\depth+5pt\relax]%
	{$\displaystyle #1$}%
}

\theoremstyle{plain}
\newtheorem{theorem}{Theorem}[section]
\newtheorem{lemma}[theorem]{Lemma}
\newtheorem{corollary}[theorem]{Corollary}
\newtheorem{proposition}[theorem]{Proposition}

\theoremstyle{definition}

\newtheorem{remark}[theorem]{Remark}
\newtheorem{example}[theorem]{Example}

\newcommand{\La}{\Lambda}

\newcommand{\asy}{\asymp}

\newcommand{\R}{\mathbb{R}}

\newcommand{\Q}{\mathbb{Q}}
\newcommand{\Z}{\mathbb{Z}}
\newcommand{\N}{\mathbb{N}}

\newcommand{\bi}{\mathbf{i}}
\newcommand{\bj}{\mathbf{j}}

\newcommand{\fola}{\mathcal{O}^+_{\Lambda}}

\newcommand{\eps}{\varepsilon}

\newcommand{\F}{\mathcal{F}}

\newcommand{\de}{\delta}
\newcommand{\f}{\frac}

\newcommand{\lam}{\lambda}
\newcommand{\od}{\overline{D}}
\newcommand{\ud}{\underline{D}}
\newcommand{\Nm}{N_{\text{max}}}
\newcommand{\yle}{\lesssim}
\newcommand{\yge}{\gtrsim}
\newcommand{\mI}{\mathcal{I}}
\newcommand{\mB}{\mathcal{B}}
\newcommand{\mbB}{\mathbf{B}}
\newcommand{\mE}{\mathcal{E}}

\newcommand{\bs}{\boldsymbol{\sigma}}
\newcommand{\col}{\Theta}
\newcommand{\mQ}{\mathscr{Q}}
\newcommand{\fo}{\mathcal{O}^+}

\DeclareMathOperator{\dimLM}{\dim_{\mathrm{LM}}}
\DeclareMathOperator{\dimUM}{\dim_{\mathrm{UM}}}
\DeclareMathOperator{\dimL}{\dim_{\mathrm{LH}}}

\DeclareMathOperator{\dimbe}{\dim_{\mathrm{Be}}}
\DeclareMathOperator{\dimH}{\dim_{\mathrm{H}}}
\DeclareMathOperator{\dimdH}{\dim_{\mathrm{DH}}}
\DeclareMathOperator{\dimdP}{\dim_{\mathrm{DP}}}
\DeclareMathOperator{\dimB}{\dim_{\mathrm{B}}}
\DeclareMathOperator{\dimP}{\dim_{\mathrm{P}}}
\DeclareMathOperator{\dimM}{\dim_{\mathrm{M}}}

\DeclareMathOperator{\dima}{dim_A}

\allowdisplaybreaks[4]

\title{Dimensions of Bedford-McMullen type sets in $\Z^2$}
\author{Jun Jie Miao}
\address{School of Mathematical Sciences,  Key Laboratory of MEA(Ministry of Education) \& Shanghai Key Laboratory of PMMP,  East China Normal University, Shanghai 200241, China}

\email{jjmiao@math.ecnu.edu.cn}

\author{Minghui Xu}
\address{School of Mathematical Sciences, East China Normal University, No. 500, Dongchuan Road, Shanghai 200241, P. R. China}
\email{xmhhh@stu.ecnu.edu.cn}

\begin{document}
	
	\begin{abstract}
		We study forward orbits in $\mathbb Z^2$ generated by the expanding
		affine maps $(x,y)\mapsto(mx+i,ny+j)$, where $m\ge n\ge2$ are
		integers and $(i,j)$ ranges over a nonempty digit set
		$\Lambda\subseteq\{0,\ldots,m-1\}\times\{0,\ldots,n-1\}$.
		We obtain explicit formulae for the mass, Beurling, discrete packing,
		and Assouad dimensions, the discrete Hausdorff dimension and its
		lower variant, and the lower entropy index.
		When $m=n$, all these dimensions equal $\log_m\#\Lambda$.
		When $m>n$, they can differ, and several depend on the starting
		point through the sizes of the endpoint columns of $\Lambda$.
		The two Hausdorff dimensions coincide and admit a pressure
		variational formula. Our proofs use finite symbolic models and
		approximate squares to relate digit counts to coverings in centered
		windows, translated windows, and annuli.
		We also characterize invariant lattice sets and determine their
		dimensions. Finally, we compare the orbit dimensions with semigroup
		growth and the dimensions of the compact dual attractor. In
		particular, the mass and Beurling dimensions need not equal the
		semigroup growth exponent, in contrast to the corresponding
		one-dimensional theory.
	\end{abstract}
	
	\maketitle
	\tableofcontents
	\section{Introduction}
	
	Bedford--McMullen carpets are a basic class of planar self-affine
	sets for which the effect of unequal contraction rates on dimension
	can be described explicitly. Given integers $m\ge n\ge2$ and a
	nonempty digit set
	$\La\subseteq\{0,\ldots,m-1\}\times\{0,\ldots,n-1\}$, the associated
	carpet is
	\[
	F_\La=
	\left\{\left(\sum_{\ell=1}^{\infty}\frac{i_\ell}{m^\ell},
	\sum_{\ell=1}^{\infty}\frac{j_\ell}{n^\ell}\right):
	(i_\ell,j_\ell)\in\La\text{ for every }\ell\ge1\right\}.
	\]
	The work of Bedford~\cite{Bedford1984} and
	McMullen~\cite{McMullen1984} established that the Hausdorff and
	box dimensions of these carpets can differ when the coordinate
	contraction ratios are unequal. Their results show how dimension
	depends on both the contraction ratios and the distribution of
	digits among the coordinate fibers. We recall the explicit
	classical formulae in Section~\ref{semigroup growth section}.
	
	Subsequent work has developed this connection between
	anisotropic geometry and dimension in several directions.
	Kenyon and Peres~\cite{KP1996} extended the theory to self-affine
	sponges and established the existence of invariant measures of
	full dimension for a class of expanding toral maps.
	Ferguson, Fraser, and Sahlsten~\cite{FFS2015} applied ergodic
	methods to projection and distance-set problems for self-affine
	carpets. Mackay~\cite{Mackay2011} determined the Assouad
	dimension of Bedford--McMullen carpets, while Banaji and
	Kolossv\'ary~\cite{BK2024} computed their intermediate dimensions and
	obtained obstructions to Lipschitz equivalence. These results
	make the carpets a natural setting for comparing notions of
	dimension that capture different aspects of scaling.
	
	In this paper we study a discrete expanding counterpart of this
	construction. For $(i,j)\in\La$, define
	\begin{equation}\label{map}
		S_{(i,j)}(x,y)=(mx+i,ny+j),\qquad (x,y)\in\Z^2.
	\end{equation}
	For $z=(u,v)\in\Z^2$ and $k\ge1$, let
	\[
	\La_k(z)=
	\left\{\left(m^ku+\sum_{\ell=0}^{k-1}i_\ell m^\ell,
	n^kv+\sum_{\ell=0}^{k-1}j_\ell n^\ell\right):
	(i_\ell,j_\ell)\in\La\text{ for }0\le\ell<k\right\}.
	\]
	The forward orbit of $z$ is
	$\fola(z)=\bigcup_{k\ge1}\La_k(z)$. Equivalently, it consists of
	all images of $z$ under nonempty finite compositions of the maps
	in~\eqref{map}. For
	$(\bi,\bj)=(i_0\ldots i_{k-1},j_0\ldots j_{k-1})\in\La^k$, we write
	\[
	S_{(\bi,\bj)}=S_{(i_0,j_0)}\circ\cdots\circ S_{(i_{k-1},j_{k-1})}.
	\]
	We call $\fola(0,0)$ the \emph{discrete Bedford--McMullen carpet}
	associated with $\La$. The inverse maps, acting on $\R^2$, have
	the compact attractor $-F_\La$. Thus the forward orbit and the
	classical carpet arise from the expanding and contracting versions
	of the same affine system.
	
	Expanding constructions of this kind belong to the theory of
	\emph{reverse iterated function systems}, introduced by
	Strichartz~\cite{Strichartz1998} in the study of fractals in the
	large. Every orbit considered here is a subset of $\Z^2$, and
	hence its ordinary Hausdorff dimension is zero. Its geometry is
	instead reflected in the growth of point counts and covering
	numbers as the observation window tends to infinity. We study
	several dimensions that describe different aspects of this
	large-scale geometry.
	
	{Before introducing these dimensions, we illustrate how the
		digit set and the starting point affect the orbit geometry through
		the six examples in Figure~\ref{fig:intro-orbits}.
		Take $m=4$, $n=3$, and
		\[
		\begin{aligned}
			\Lambda_1&=\{0,3\}\times\{0,2\},\\
			\Lambda_2&=\{(0,0),(1,1),(2,1),(3,2)\},\\
			\Lambda_3&=\bigl(\{0\}\times\{0,1,2\}\bigr)
			\cup\{(1,1)\}
			\cup\bigl(\{3\}\times\{0,2\}\bigr).
		\end{aligned}
		\]
		The first three panels compare a Cartesian product of digit sets,
		a pattern with one digit pair in each occupied column, and a pattern
		with unequal column sizes. Panels (c)--(f) use the same digit set
		$\Lambda_3$, with starting points $(0,0)$, $(-1,0)$, $(1,0)$, and
		$(0,1)$, respectively. They show how changing either coordinate of
		the starting point can affect the location and distribution of the orbit. 	The variations in these windows motivate a comparison of
		point growth and covering behavior at large scales. We begin with
		dimensions defined by point counts.
		
		\begin{figure}[htbp]
			\centering
			\includegraphics[width=\linewidth]{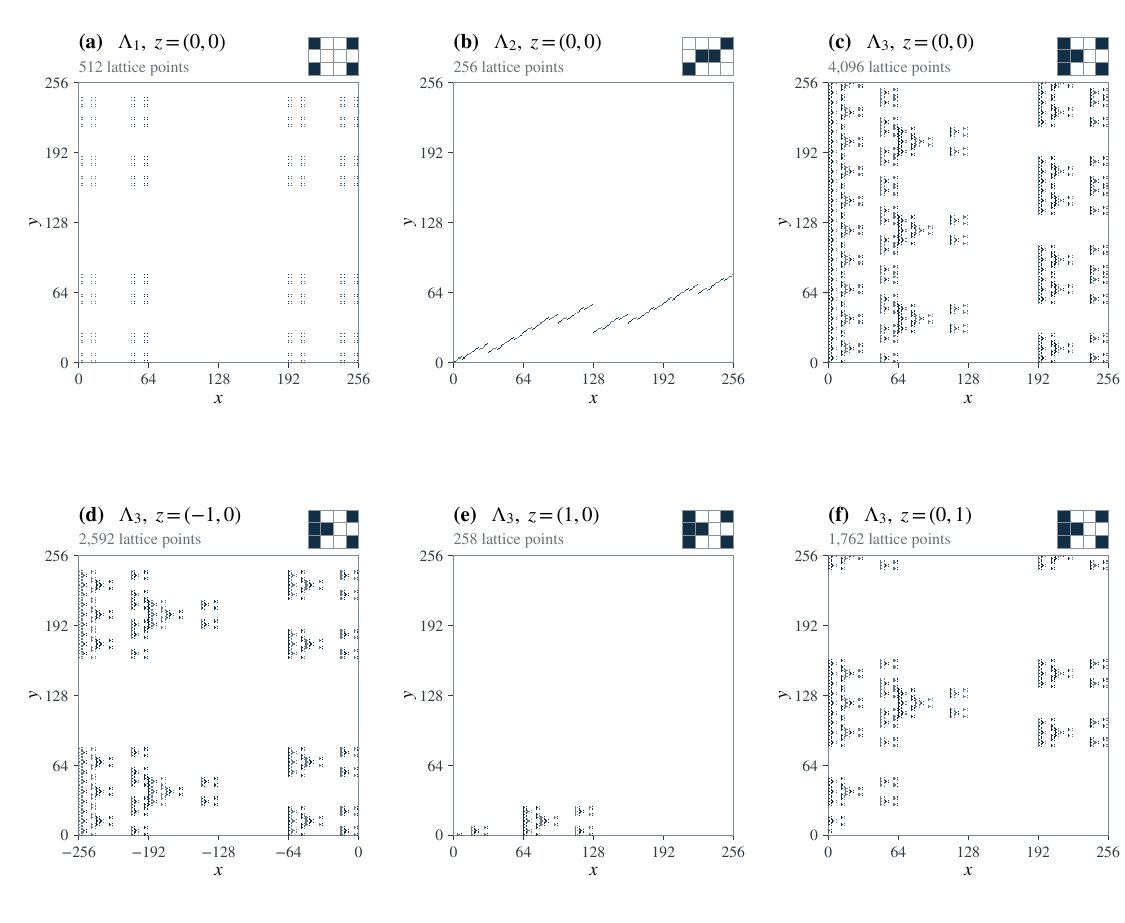}
			\caption{Forward orbits $\mathcal{O}_{\Lambda_j}^{+}(z)$
				for $m=4$ and $n=3$, shown in the window $[0,256)^2$,
				except for panel (d), which uses $[-256,0)\times[0,256)$.
				Shaded cells in the $4\times3$ insets indicate the digit sets.}
			\label{fig:intro-orbits}
		\end{figure}

		Mass dimensions measure the polynomial growth of the number of
		points in windows centered at a fixed point. They form part of
		the discrete dimension theory developed by Barlow and
		Taylor~\cite{BarlowTaylor1989,BarlowTaylor1992}. They have also been
		used in discrete analogues of geometric measure theory:
		Glasscock~\cite{Glasscock2016} proved Marstrand-type projection
		results for counting and upper mass dimensions, with applications
		to the dimensions of typical linear sumsets, and
		Pathak~\cite{Pathak2024} established a corresponding slicing
		theorem for mass dimension.
		Glasscock, Moreira, and Richter~\cite{GMR2024} used mass
		dimensions to study intersections and sumsets of integer sets
		invariant under multiplicatively independent bases, establishing
		discrete analogues of Furstenberg's transversality results.
		These results make mass dimension
		useful for studying projections, intersections with strips, and
		additive properties of sparse sets.
		
		Given $A\subseteq \Z^2$, recall that the {\em lower} and {\em upper mass dimensions} of $A$ are defined as
		\[
		\dimLM A=\liminf_{h\to+\infty}\f{\log \#(A\cap [-h,h]^2)}{\log h},\quad \dimUM A=\limsup_{h\to+\infty}\f{\log \#(A\cap [-h,h]^2)}{\log h}.
		\]
		If $\dimLM A=\dimUM A$, then the common value, the {\em mass dimension} of $A$, is denoted by $\dimM A$. 
		
		The Beurling dimension uses the largest point count among all
		translates of a window, and thus also detects dense regions far
		from the origin. An important application is to frequency sets
		of exponential frames for fractal measures. Dutkay, Han, Sun,
		and Weber~\cite{DHSW2011} related the Beurling dimension of such
		frequency sets to the Hausdorff dimension of the underlying
		fractal, proving equality under additional hypotheses.
		He, Kang, Tang, and Wu~\cite{HKTW2018} obtained sharp upper
		bounds for the Beurling dimensions of Bessel sets and frame
		spectra for classes of self-similar measures.
		This connects the spatial distribution of an unbounded set of
		frequencies with the geometry of a compactly supported measure.
		
		Fourier spectra of planar self-affine measures provide a
			closely related class of expanding orbits. Recall that a set
			$\Gamma\subset\mathbb R^2$ is a \emph{spectrum} of a probability
			measure $\mu$ on $\mathbb R^2$ if
			$\{x\mapsto e^{2\pi i\langle\gamma,x\rangle}:\gamma\in\Gamma\}$
			is an orthonormal basis for $L^2(\mu)$.
			Consider the Sierpi\'nski-type measure $\mu_{D,\La}$ with
			equal weights for the maps $x\mapsto D^{-1}(x+d),d\in\La$, where
			$D=\operatorname{diag}(m,n)$ and
			$\La=\{(0,0),(1,0),(0,1)\}$.
			Dutkay and Jorgensen~\cite{DJ2007} treated the integer case
			$m=n$, and Li~\cite{Li2010} the integer diagonal case:
			$\mu_{D,\La}$ is spectral if and only if
			$m,n\in3\mathbb Z$.
			Deng and Lau~\cite{DL2015} and Dai, Fu, and Yan~\cite{DFY2021}
			extended this criterion to real expansion factors in the
			self-similar and diagonal self-affine settings, respectively. When $m,n\in3\mathbb Z$, the frequency digit set
			$\La'=\{(0,0),(m/3,-n/3),(-m/3,n/3)\}$ gives the spectrum
			\[
			\Gamma_{\La'}=\bigcup_{k\ge1}
			\left\{\sum_{\ell=0}^{k-1}D^\ell d_\ell:
			d_\ell\in \La'\right\};
			\]
			see \cite{Li2010,LW2025Spectra}.
			Thus $\Gamma_{\La'}$ is the orbit of the origin under the expanding
			maps $x\mapsto Dx+d$, $d\in \La'$, with signed frequency digits.
			Li and Wu~\cite{LW2023Beurling} proved that
			$\dimbe\Gamma_{\La'}=\log_m3$.
			They subsequently showed that $\log_m3$ is the optimal upper
			bound for the Beurling dimensions of spectra of
			$\mu_{D,\La}$, and that every value in $[0,\log_m3]$
			is attained by a spectrum~\cite{LW2025Spectra}.
		
		Given $\alpha \ge 0$, the  \emph{$\alpha$-Beurling density} (or the \emph{upper Beurling density of order $\alpha$}) of $A$ is defined by
		\begin{equation}\label{def_Bedensity}
			\mathcal{D}^+_\alpha(A) = \limsup_{h\to+\infty} \sup_{z\in\R^2} \frac{\#(A \cap (z+[-h,h]^2))}{h^\alpha}.
		\end{equation}
		The \emph{Beurling dimension} (or \emph{upper Beurling dimension}) of $A$ is given by
		\[
		\dimbe A = \sup\{\alpha\ge 0 : \mathcal{D}^+_\alpha(A) > 0\} = \inf\{\alpha>0 : \mathcal{D}^+_\alpha(A) < \infty\}.
		\]
		Equivalently,
		\[
		\dimbe A = \limsup_{h\to+\infty} \sup_{z\in\R^2} \frac{\log \#(A \cap (z+[-h,h]^2))}{\log h}.
		\]
		It follows immediately that $\dimbe A \ge \dimUM A$.
		
		Covering-based dimensions also take account of the arrangement
		of points within large windows. The discrete Hausdorff dimension
		of Barlow and Taylor~\cite{BarlowTaylor1989,BarlowTaylor1992} is
		widely called the \emph{macroscopic Hausdorff dimension},
		particularly in probability. It measures the normalized cost of
		covering a set in successive annuli, using cubes whose side
		lengths are bounded below by one. Applications include the
		geometry of random-walk ranges: Georgiou, Khoshnevisan, Kim,
		and Ramos~\cite{GKKR2018} determined the macroscopic Minkowski
		and Hausdorff dimensions of the range of an arbitrary transient
		random walk on $\Z^d$. In a different direction,
		Khoshnevisan, Kim, and Xiao~\cite{KKX2017,KKX2018} used macroscopic
		Hausdorff dimension to describe the multifractal structure of the
		tall peaks of solutions to parabolic stochastic partial
		differential equations. Their work also distinguishes this
		large-scale multifractality from intermittency.
		The same covering framework can therefore describe both
		deterministic lattice sets and random sets at large scales.
		
		We use the following lattice formulation. Given
		$x\in\mathbb Z^2$ and $r\ge1$, define the cubes
		\begin{gather*}
			C(x,r) = \{y \in \mathbb{Z}^2 : x_i \le y_i < x_i + r,\ i=1,2\},\\
			V(x,r) = \{y \in \mathbb{Z}^2 : x_i - r/2 \le y_i < x_i + r/2,\ i=1,2\},
		\end{gather*}
		and let $\mathscr{C} = \{C(x,r) : x\in\mathbb{Z}^2,\ r\in\mathbb{Z}_{>0}\}$. For a finite set $A \subset \mathbb{Z}^2$, set
		\[
		d(A) = \min\{ r \in \mathbb{Z}_{>0} : A \subseteq C(x,r) \text{ for some } x\in\mathbb{Z}^2\}.
		\]
		Let $V_k^{(2)}=V(0,2^k)$ for $k\ge0$, put
		$S_1^{(2)}=V_1^{(2)}$, and set
		$S_k^{(2)}=V_k^{(2)}\setminus V_{k-1}^{(2)}$ for $k\ge2$.
		Then $d(V_k^{(2)})=d(S_k^{(2)})=2^k$ for $k\ge1$.
		For $A,F \subset \mathbb{Z}^2$ with $F$ finite and nonempty and $\alpha>0$, define
		\[
		\nu_\alpha(A,F) = \min\left\{ \sum_i \left(\frac{d(U_i)}{d(F)}\right)^{\alpha} : U_i \in \mathscr{C},\ A \cap F \subseteq \bigcup_i U_i \right\}.
		\]
		The \emph{discrete Hausdorff dimension}, or \emph{macroscopic Hausdorff dimension}, of $A$ is
		\[
		\dimdH A = \inf\Bigl\{ \alpha>0 : \sum_{n=1}^\infty \nu_\alpha(A,S_n^{(2)}) < \infty \Bigr\}.
		\]
		The \emph{lower discrete Hausdorff dimension} is defined using convergence to zero in place of summability:
		\[
		\dimL A = \inf\Bigl\{ \alpha>0 : \lim_{n\to\infty} \nu_\alpha(A,S_n^{(2)}) = 0 \Bigr\}.
		\]
		These definitions give $\dimL A\le\dimdH A$, but the two
		dimensions need not coincide for general lattice sets. The
		lower discrete Hausdorff dimension also occurs in harmonic
		analysis. Li, Zeng, and Wu~\cite{LZW2024,LZW2025corr} used it
		to compare the size of Fourier spectra of Moran measures with
		the Hausdorff dimension of their supports. This provides another
		connection between discrete dimension and the geometry of
		fractal measures.
		
		To compare point counts and coverings across intermediate
		scales, we also consider the discrete packing dimension, the
		lower entropy index, and the large-scale Assouad dimension.
		Discrete packing dimension also has applications to random
		media: Xiao and Zheng~\cite{XZ2013} determined the discrete
		Hausdorff and packing dimensions of the ranges of random walks
		in an i.i.d. conductance environment on $\Z^d$, $d\ge3$, with
		conductances bounded below by a positive constant.
		We first fix our covering and packing conventions.
		
		For $A\subseteq\mathbb R^d$, let $N_r(A)$ be the minimum number
		of axis-parallel cubes of side length $r$ needed to cover $A$,
		and let $P_r(A)$ be the maximum cardinality of a family of
		pairwise disjoint axis-parallel cubes of side length $r$
		centered at points of $A$. Throughout this paper, the scale
		parameter $r$ in these quantities is always assumed to satisfy
		$r\ge1$. Standard covering and packing estimates give
		\[
		N_r(A)\asy_d P_r(A).
		\]
		Moreover, since $r/2\le\lfloor r\rfloor\le r$ for $r\ge1$,
		\[
		N_{\lfloor r\rfloor}(A)\asy_d N_r(A),
		\qquad
		P_{\lfloor r\rfloor}(A)\asy_d P_r(A),
		\]
		with constants independent of $A$ and $r$. We shall therefore
		use real and integer scales interchangeably, since this changes
		the relevant estimates only by multiplicative constants and
		does not affect the resulting dimensions. When real scales are
		used, maxima and minima over $r$ are understood as suprema and
		infima, respectively.
		
		For $A\subseteq\mathbb Z^2$, define the \emph{discrete packing
			dimension} (or the \emph{upper entropy index}) of $A$ by 
		\[
		\dimdP A=\inf\bigg\{\alpha\ge 0:\lim_{k\to\infty}\max_{1\le r\le 2^{k(1-\eps)}}\Big(\f{r}{2^k}\Big)^\alpha P_r(A\cap S_k^{(2)})=0\text{\ for each\ }\eps\in (0,1)\bigg\}
		\]
		and the \emph{lower entropy index} of $A$ by
		\[
		\de (A)=\inf\bigg\{\alpha\ge 0:\lim_{k\to\infty}\min_{1\le r\le 2^{k(1-\eps)}}\Big(\f{r}{2^k}\Big)^\alpha P_r(A\cap S_k^{(2)})=0\text{\ for some\ }\eps\in (0,1)\bigg\}.
		\]
		The classical Assouad dimension is closely related to the
		doubling property and to Euclidean embeddings of snowflaked
		metric spaces; see Naor and Neiman~\cite{NN2012}. Here we use
		its large-scale version. Writing $B(x,R)$ for the Euclidean
		ball of radius $R$ centered at $x$, we define
		\[
		\begin{split}
			\dima A=\inf\biggl\{\alpha\ge0:\ &\text{there exists }C>0\text{ such that}\\
			&N_r(A\cap B(x,R))\le C(R/r)^\alpha
			\text{ for every }x\in A\text{ and }1\le r<R\biggr\}.
		\end{split}
		\]
		
		For every nonempty $A\subseteq\Z^2$, the following inequalities hold:
		\[
		0 \,
		\mathrel{\vcenter{\hbox{%
					\setlength{\arraycolsep}{2pt}%
					\renewcommand{\arraystretch}{0.85}%
					$\begin{array}{@{}ccc@{}}
						\rotatebox[origin=c]{15}{$\le$}
						&
						\raisebox{1pt}{\,$ \dimL A\le \dimdH A\le\de(A)$\,}
						&
						\raisebox{0.5pt}{\rotatebox[origin=c]{-15}{$\le$}}
						\\[4pt]
						\rotatebox[origin=c]{-15}{$\le$}
						&
						\raisebox{-2pt}{$\dimLM A$}
						&
						\raisebox{-0.5pt}{\rotatebox[origin=c]{15}{$\le$}}
					\end{array}$}}}
		\,\dimUM A\,\mathrel{\vcenter{\hbox{%
					\setlength{\arraycolsep}{2pt}%
					\renewcommand{\arraystretch}{0.85}%
					$\begin{array}{@{}ccc@{}}
						\rotatebox[origin=c]{15}{$\le$}
						&
						\raisebox{1pt}{$\dimdP A$}
						&
						\raisebox{0.5pt}{\rotatebox[origin=c]{-15}{$\le$}}
						\\[4pt]
						\rotatebox[origin=c]{-15}{$\le$}
						&
						\raisebox{-2pt}{$\dimbe A$}
						&
						\raisebox{-0.5pt}{\rotatebox[origin=c]{15}{$\le$}}
					\end{array}$}}}\, \dima A\le 2.
		\]
		In general, the pairs
		\[
		(\dimL A,\dimLM A),\quad (\dimdH A,\dimLM A),\quad
		(\de(A),\dimLM A),\quad(\dimbe A,\dimdP A)
		\]
		are not comparable; see Example \ref{distinguish} and examples in \cite{BarlowTaylor1992}.
		\begin{remark}\label{finite stability}
			Except for the lower mass dimension and the lower entropy index,
			all the dimensions introduced above are finitely stable. Namely, if
			$A=\bigcup_{j=1}^s A_j$, then
			\[
			\textup{Dim} A=\max_{1\le j\le s}\textup{Dim} A_j
			\]
			for $
			\textup{Dim}\in\{\dimUM,\dimL,\dimdH,\dimdP,\dimbe,\dima\}.$
			The finite stability follows directly from the corresponding
			subadditivity estimates for counting, covering, and packing quantities.
			In contrast, $\dimLM$ and $\delta$ need not be finitely stable, since
			the relevant $\liminf$ and the minimization over intermediate scales,
			respectively, need not commute with finite unions.
		\end{remark}

		Related results for expanding systems were obtained by
		Miao and Xu~\cite{MX26}. For one-dimensional affine expanding iterated function systems, they proved that every locally finite forward orbit has mass dimension equal
		to the growth exponent of the semigroup of distinct maps. If the orbit
		is uniformly locally finite, its Beurling dimension is equal to the same
		exponent. They further related this
		exponent to the dual attractor and decomposed locally finite
		invariant sets into finitely many forward orbits. We use the
		decomposition theorem in Section~\ref{section invariant} and recall
		the precise dimension comparisons in Section~\ref{semigroup growth section}.
		
		The planar systems considered here exhibit a different
		phenomenon when $m>n$. Although every orbit is uniformly
		discrete and distinct symbolic words define distinct affine
		maps, the mass and Beurling dimensions can differ. Moreover,
		the mass dimension and the covering-based dimensions can
		depend on the starting point. Our main results determine all
		the dimensions introduced above and identify the precise
		roles of the two expansion factors, the column sizes, and
		the starting point. To state them, we introduce the following
		notation.
		
		For $i\in\{0,1,\ldots,m-1\}$, we denote
		\[
		\col_i=\{j:(i,j)\in\La\},\qquad N_i=\#\col_i,\qquad N_\La=\#\{i:N_i>0\},\qquad \Nm=\max_{0\le i\le m-1} N_i.
		\]
		Let
		\[
		\eta(u)=\begin{cases}
			\max\{1,N_0\}, & u=0,\\
			\max\{1,N_{m-1}\}, \quad & u=-1,\\
			1, & \text{otherwise}.
		\end{cases}
		\]
		The parameter $\eta(u)$ records the size of an endpoint column
		when the initial horizontal coordinate allows the leading
		horizontal digits to cancel its expansion. This occurs for
		$u=0$ in column $0$ and for $u=-1$ in column $m-1$. Otherwise,
		no such column contributes additional branching at the
		larger vertical digit levels, and $\eta(u)=1$.
		
		With this notation, we can state the dimension formulae.

		\begin{theorem}\label{main theorem m=n}
			If $m=n$, then for each $z\in\Z^2$, 
			\[
			\textup{Dim}\, \fola(z)=\log_m \#\La,
			\]
			where $\textup{Dim}\in\{\dimL,\dimdH,\de(\cdot),\dimM,\dimdP,\dimbe,\dima\} $.
		\end{theorem}
		
		\begin{theorem}\label{main theorem m>n}
			If $m>n$, then for each $z=(u,v)\in\Z^2$,
			\begin{align*}
				\dimM\fola(z)&=\log_m\Big(\f{\#\La}{\eta(u)}\Big)+\log_n\eta(u),\\
				\dimbe\fola(z)&=\log_m\Big(\f{\#\La}{\Nm}\Big)+\log_n\Nm,\\
				\de(\fola(z))&=\min\Big\{\log_m\Big(\f{\#\La}{\eta(u)}\Big),\log_m N_\La\Big\}+\log_n\eta(u),\\
				\dimdP\fola(z)&=\max\Big\{\log_m\Big(\f{\#\La}{\eta(u)}\Big),\log_m N_\La\Big\}+\log_n\eta(u),\\
				\dima \fola(z)&=\log_m N_\La+\log_n \Nm,\\
				\dimL\fola(z)&=\dimdH \fola(z)=\min_{0\le\lambda\le1}\log_m\bigg(\sum_{i:N_i>0}\Big(\f{N_i}{\eta(u)}\Big)^\lam\bigg)+\log_n \eta(u).
			\end{align*}
		\end{theorem}
		
		The equal-expansion case gives a common value for all the
		dimensions. For unequal expansion factors, the formulae
		separate several different features of the digit set. The
		Beurling and Assouad dimensions involve the largest column,
		whereas the mass dimension involves the endpoint parameter
		$\eta(u)$. The two Hausdorff dimensions coincide and are
		given by a minimization involving all nonempty columns.	In Section~\ref{section hausdorff}, we express this common value
		both as a minimum of topological pressures on the full shift over
		the nonempty columns and as an entropy maximization over probability
		vectors on these columns.
		The discrete packing dimension and the lower entropy index
		are, respectively, the larger and smaller of two explicit
		counting exponents. In particular, even these simple affine
		systems allow several distinct dimension values, as the
		following examples show.
		
		\begin{example}\label{distinguish}
			Let $m=7$, $n=6$, $z=(0,0)$, and
			\[
			\begin{aligned}
				\Lambda_1
				&=
				\bigl(\{0\}\times\{0,1,2\}\bigr)
				\cup\{(1,0)\}
				\cup
				\bigl(\{2\}\times\{0,1,\ldots,5\}\bigr),\\
				\Lambda_2
				&=
				\bigl(\{0\}\times\{0,1,2\}\bigr)
				\cup
				\bigl(\{1,2,3,4\}\times\{0\}\bigr)
				\cup
				\bigl(\{5,6\}\times\{0,1,\ldots,5\}\bigr).
			\end{aligned}
			\]
			Set $E_i=\mathcal O_{\Lambda_i}^{+}(z), i=1,2$.
			Figure~\ref{fig:two-orbits-m7-n6} compares these two orbits in
			the same window. The calculations below give their dimension values
			and show that the dimensions occur in different orders.
			
			Put
			\[
			\begin{aligned}
				\Delta_1&=\log_7\left(1+\log_36\cdot(\log_23)^{\log_62}\right)+\log_63,\\
				\Delta_2
				&=\log_7\left(1+2\log_36\cdot(2\log_23)^{\log_62}\right)+\log_63.
			\end{aligned}
			\]
			
			\noindent\begin{minipage}{\textwidth}
				For the first orbit, the dimension formulae give
				\[
				\begin{alignedat}{2}
					\dimL E_1
					=\dim_{\mathrm{DH}}E_1
					&=\Delta_1,
					&\qquad
					\Delta_1&\approx1.168912,\\
					\delta(E_1)
					&=\log_7 3+\log_6 3,
					&
					&\approx1.177722,\\
					\dim_{\mathrm M}E_1
					=\dim_{\mathrm{DP}}E_1
					&=\log_7\frac{10}{3}+\log_6 3,
					&
					&\approx1.231867,\\
					\dim_{\mathrm{Be}}E_1
					&=1+\log_7\frac{5}{3},
					&
					&\approx1.262512,\\
					\dim_{\mathrm A}E_1
					&=1+\log_7 3,
					&
					&\approx1.564575.
				\end{alignedat}
				\]
				Hence
				\[
				\dimL E_1=\dim_{\mathrm{DH}}E_1<
				\delta(E_1)<\dim_{\mathrm M}E_1=\dim_{\mathrm{DP}}E_1<
				\dim_{\mathrm{Be}}E_1<\dim_{\mathrm A}E_1.
				\]
			\end{minipage}
			
			\medskip
			
			\noindent\begin{minipage}{\textwidth}
				For the second orbit, one has
				\[
				\begin{alignedat}{2}
					\dimL E_2
					=\dim_{\mathrm{DH}}E_2
					&=\Delta_2,
					&\qquad
					\Delta_2&\approx1.542154,\\
					\dim_{\mathrm M}E_2
					=\delta(E_2)
					&=\log_7\frac{19}{3}+\log_6 3,
					&
					&\approx1.561714,\\
					\dim_{\mathrm{Be}}E_2
					&=1+\log_7\frac{19}{6},
					&
					&\approx1.592360,\\
					\dim_{\mathrm{DP}}E_2
					&=1+\log_6 3,
					&
					&\approx1.613147,\\
					\dim_{\mathrm A}E_2
					&=2.
					&
				\end{alignedat}
				\]
				Consequently,
				\[
				\dimL E_2=\dim_{\mathrm{DH}}E_2<\dim_{\mathrm M}E_2=\delta(E_2)<\dim_{\mathrm{Be}}E_2<\dim_{\mathrm{DP}}E_2<\dim_{\mathrm A}E_2.
				\]
			\end{minipage}
			\begin{figure}[htbp]
				\centering
				\includegraphics[width=\linewidth]{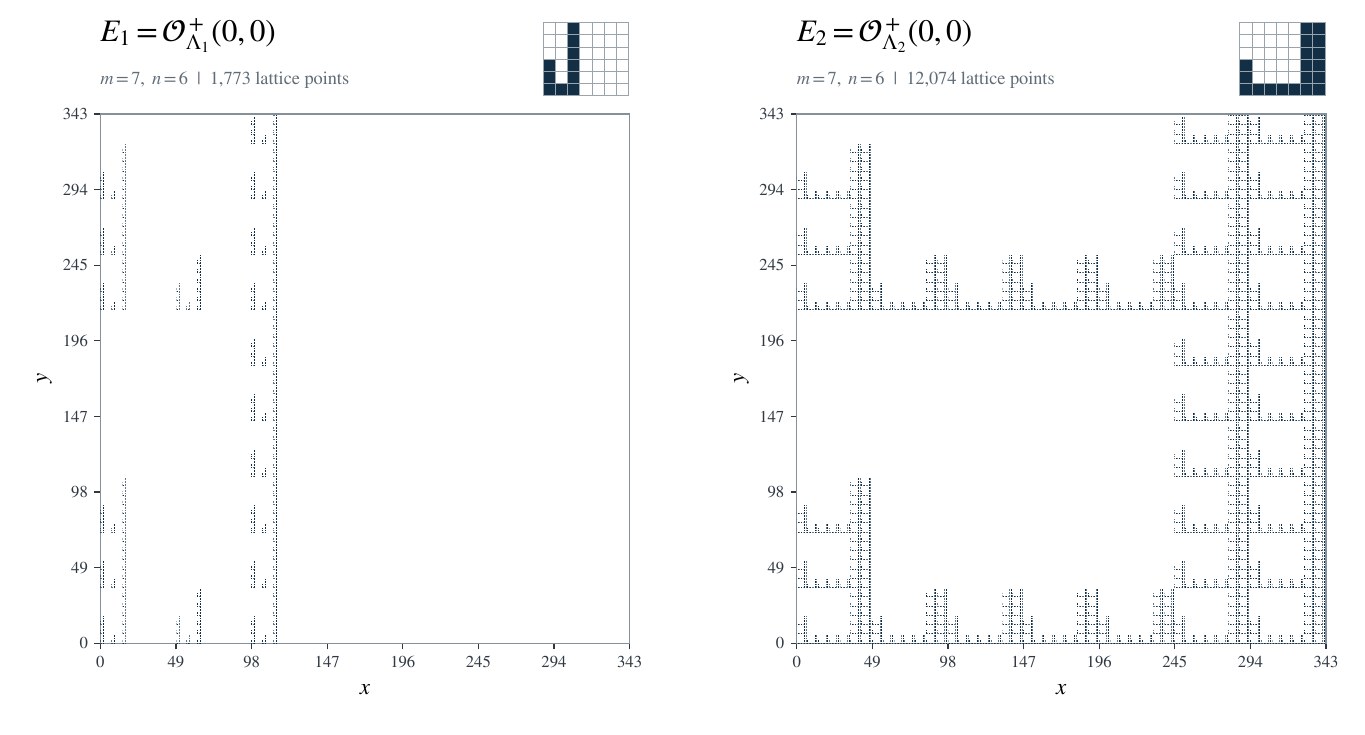}
				\caption{The forward orbits $E_i=\mathcal O_{\Lambda_i}^{+}(0,0)$,
					$i=1,2$.
					Each panel shows every orbit point in
					$[0,343)^2\cap\mathbb Z^2$.}
			\label{fig:two-orbits-m7-n6}
		\end{figure}
	\end{example}

	To prove the dimension formulae in Theorems \ref{main theorem m=n} and \ref{main theorem m>n} , we use a finite
	symbolic model that reflects the different coordinate scales. At scale $m^k$, there are $k$
	levels of paired horizontal and vertical digits and up to
	$\lfloor k\log m/\log n\rfloor-k$ additional vertical levels.
	Nested approximate squares record the digits visible at a
	given resolution. Their counts give the mass, Beurling,
	packing, and Assouad estimates, while an entropy--pressure
	argument gives the two Hausdorff dimensions. Allowing the
	vertical alphabets to vary from level to level lets us
	represent orbit pieces in translated windows by the same
	class of models. A separate localization argument places
	suitable model pieces in annuli, as required for the lower
	entropy index.
	
	In parallel with the one-dimensional study in \cite{MX26}, we
	determine the dimensions of invariant lattice sets and compare
	orbit dimensions with semigroup growth and the dimensions of the
	dual attractor. These comparisons reveal a distinction between
	the one- and two-dimensional settings: when $m>n$, the semigroup
	growth exponent, the mass and Beurling dimensions of the orbit,
	and the box and affinity dimensions of the dual attractor form
	a nondecreasing sequence in which all four inequalities can be
	strict. This reflects the unequal coordinate expansion rates:
	the orbit dimension formulae involve column data, whereas those
	of the dual carpet involve row data.

	The paper is organized as follows. Section \ref{section bases} discusses changes of base and discrete scales. In Section \ref{section model}, we develop a finite symbolic model for the forward orbit and establish the approximate-square estimates and transfer principles used throughout the paper. Section \ref{section counts} applies these estimates to the mass, Beurling, discrete packing, and Assouad dimensions, while Section \ref{section hausdorff} treats the discrete Hausdorff and lower discrete Hausdorff dimensions through an entropy--pressure approach. The case $m=n$ is handled separately in Section \ref{section equal}. Section \ref{section lower entropy} is devoted to the lower entropy index, whose annular nature requires a separate argument. In Section \ref{section invariant}, we characterize invariant subsets of $\Z^2$ and determine their dimensions from those of the constituent forward orbits. Finally, Section \ref{semigroup growth section} compares the orbit dimensions with semigroup growth and the dimensions of the dual Bedford--McMullen attractor.
	
	Throughout the paper, we write
	\[p(a)=\lfloor\log_m a\rfloor,\qquad
	q(a)=\lfloor\log_n a\rfloor\]
	for $a\ge1$, and
	$Q(\omega,h)=\omega+[-h,h]^2$ for
	$\omega\in\mathbb R^2$ and $h>0$.
	
	\section[Preliminaries on discrete dimensions]{Preliminaries on discrete dimensions}\label{section bases}
	
	For $t>1$, put $V_k^{(t)}=V(0,t^k)$, $S_1^{(t)}=V_1^{(t)}$, and
	$S_k^{(t)}=V_k^{(t)}\setminus V_{k-1}^{(t)}$ for $k\ge2$.
	Let $I_k^{(t)}=\Z\cap[-t^k/2,t^k/2).$
	Then $V_k^{(t)}=I_k^{(t)}\times I_k^{(t)}$, and the consecutive
	integer coordinates give the exact formula
	\begin{equation}\label{window lattice diameter}
		d(V_k^{(t)})=\#I_k^{(t)}
		=\left\lceil\frac{t^k}{2}\right\rceil
		+\left\lfloor\frac{t^k}{2}\right\rfloor=t^k+O(1).
	\end{equation}
	If $k\ge2$ and $S_k^{(t)}\ne\emptyset$, choose
	$y_k\in I_k^{(t)}\setminus I_{k-1}^{(t)}$. Then
	$I_k^{(t)}\times\{y_k\}\subseteq S_k^{(t)}\subseteq V_k^{(t)}$,
	so $d(S_k^{(t)})=d(V_k^{(t)})$. This equality is also immediate
	for $k=1$. Since $t^k-t^{k-1}\to\infty$, the sets
	$I_k^{(t)}\setminus I_{k-1}^{(t)}$ are nonempty for all sufficiently
	large $k$. Thus only finitely many annuli can be empty. We set
	$\nu_\alpha(A,\emptyset)=0$ for those indices. When $t\ge2$ is
	an integer, every annulus is nonempty and
	$d(S_k^{(t)})=d(V_k^{(t)})=t^k$ for every $k\ge1$.
	For a finite set $E\subset\Z^2$, define
	\[
	\mathscr H_\alpha(E)=\inf\left\{\sum_{j=1}^{J}d(U_j)^\alpha:
	E\subseteq\bigcup_{j=1}^{J}U_j,\ U_j\in\mathscr C\right\},
	\qquad \mathscr H_\alpha(\emptyset)=0.
	\]
	This cost is monotone, subadditive, and invariant under integer
	translations, and
	$\nu_\alpha(A,F)=d(F)^{-\alpha}\mathscr H_\alpha(A\cap F)$ when $F$ is
	finite and nonempty. Write $\delta_t(A)$ for the lower entropy index
	obtained by replacing $S_k^{(2)}$ and $2^k$ by $S_k^{(t)}$ and $t^k$.
	Thus the index in the introduction is $\delta(A)=\delta_2(A)$.

	\begin{proposition}\label{base independence}
		Let $A\subseteq\Z^2$ be nonempty and let $t>1$.\\
		(i) \begin{gather*}
			\dimdH A = \inf\Bigl\{ \alpha>0 : \sum_{k=1}^\infty \nu_\alpha(A,S_k^{(t)}) < \infty \Bigr\}
			= \inf\Bigl\{ \alpha>0 : \sum_{k=1}^\infty \nu_\alpha(A,V_k^{(t)}) < \infty \Bigr\}.\\
			\dimL A = \inf\Bigl\{ \alpha>0 : \lim_{k\to\infty} \nu_\alpha(A,S_k^{(t)}) = 0 \Bigr\}
			= \inf\Bigl\{ \alpha>0 : \lim_{k\to\infty} \nu_\alpha(A,V_k^{(t)}) = 0 \Bigr\}.
		\end{gather*}
		\noindent (ii) 	\begin{align*}
			\dimdP A&=\inf\bigg\{\alpha\ge 0:\lim_{k\to\infty}\max_{1\le r\le t^{k(1-\eps)}}\Big(\f{r}{t^k}\Big)^\alpha P_r(A\cap S_k^{(t)})=0\text{\ for each\ }\eps\in (0,1)\bigg\}\\
			&=\inf\bigg\{\alpha\ge 0:\lim_{k\to\infty}\max_{1\le r\le t^{k(1-\eps)}}\Big(\f{r}{t^k}\Big)^\alpha P_r(A\cap V_k^{(t)})=0\text{\ for each\ }\eps\in (0,1)\bigg\}.
		\end{align*}
		
		\noindent (iii)  We have
		\begin{equation}\label{general lower index comparison}
			\dimdH A\le\delta_t(A)\le\dimUM A.
		\end{equation}
	\end{proposition}
	\begin{proof}
		(i) First fix $\alpha>0$. Since $V_k^{(t)}=\bigcup_{i=1}^k S_i^{(t)}$, after absorbing the finitely many
		initial terms into the constants, we have
		\begin{gather*}
			\nu_\alpha(A,S_k^{(t)})\le \nu_\alpha(A,V_k^{(t)})\yle_{t,\alpha}
			\sum_{i=1}^k t^{-\alpha(k-i)}\nu_\alpha(A,S_i^{(t)})
		\end{gather*}
		for all sufficiently large $k$. Therefore, summing the second inequality and reversing the order of summation, together with the first inequality,
		gives
		\[
		\sum_{k=1}^\infty \nu_\alpha(A,S_k^{(t)}) < \infty\quad\Longleftrightarrow\quad \sum_{k=1}^\infty \nu_\alpha(A,V_k^{(t)}) < \infty.
		\]
		
		Note that $\nu_\alpha(A,S_k^{(t)})\le 1$. Suppose that $\lim_{k\to\infty} \nu_\alpha(A,S_k^{(t)}) = 0$. For $K\in\mathbb N$ and all sufficiently large $k>K$, we have
		\[
		\nu_\alpha(A,V_k^{(t)})\yle_{t,\alpha}\sum_{i=1}^{K}t^{-\alpha(k-i)}+\sup_{\ell>K}\nu_\alpha(A,S_\ell^{(t)})\sum_{i=K+1}^{k}t^{-\alpha(k-i)}.
		\]
		Letting first $k\to\infty$ and then $K\to\infty$ yields
		\[
		\lim_{k\to\infty} \nu_\alpha(A,S_k^{(t)}) =0\quad\Longleftrightarrow\quad \lim_{k\to\infty} \nu_\alpha(A,V_k^{(t)}) =0.
		\]
		For bases $t,b>1$, put $i(k)=\lceil k\log_b t\rceil$.
		Then $V_k^{(t)}\subseteq V_{i(k)}^{(b)}$ and
		$d(V_{i(k)}^{(b)})\asymp_{t,b}d(V_k^{(t)})$.
		Consequently, for every $\alpha>0$,
		\[
		\nu_\alpha(A,V_k^{(t)})\yle_{t,b,\alpha}
		\nu_\alpha(A,V_{i(k)}^{(b)}).
		\]
		Since $i(k)\to\infty$ and $i(k)$ has bounded multiplicity, convergence
		to zero and summability in base $b$ imply the corresponding properties
		in base $t$. Interchanging $t$ and $b$ proves the converse.

		(ii)  Define
		\[
		\mathcal P_{\alpha,\eps}^{S,t}(k)
		=\max_{1\le r\le t^{k(1-\varepsilon)}}
		\left(\frac r{t^k}\right)^\alpha P_r(A\cap S_k^{(t)}),\quad
		\mathcal P_{\alpha,\eps}^{V,t}(k)
		=\max_{1\le r\le t^{k(1-\varepsilon)}}
		\left(\frac r{t^k}\right)^\alpha P_r(A\cap V_k^{(t)}).
		\]
		We first compare these two quantities at a fixed base $t$.
		The inequality
		$\mathcal P_{\alpha,\eps}^{S,t}(k)
		\le\mathcal P_{\alpha,\eps}^{V,t}(k)$ is obvious.
		For the converse, suppose
		$\mathcal P_{\alpha,\theta}^{S,t}(\ell)\to0\,(\ell\to\infty)$ for every
		$\theta\in(0,1)$, where $\alpha>0$ is fixed.
		Given $\varepsilon\in(0,1)$, choose $0<\theta<\varepsilon$ such
		that
		\[
		\frac{2\theta}{1-\theta}<\frac{\alpha\varepsilon}{2}.
		\]
		For $1\le r\le t^{k(1-\varepsilon)}$, put
		\[
		J(r)=\max\left\{1,
		\left\lceil\frac{\log_t r}{1-\theta}\right\rceil\right\}.
		\]
		For all sufficiently large $k$, uniformly over these $r$,
		\[
		J(r)\le\frac{1-\varepsilon}{1-\theta}k+1\le k,
		\qquad t^{J(r)}\le t r^{\f1{1-\theta}}.
		\]
		All packing squares of side $r$ centered in $V_{J(r)}^{(t)}$ lie
		in a square of side $t^{J(r)}+r$. Comparing their areas gives
		\[
		P_r(A\cap V_{J(r)}^{(t)})
		\le\left(1+\frac{t^{J(r)}}r\right)^2\le\left(1+t r^{\f\theta{1-\theta}}\right)^2 \yle_t r^{\f{2\theta}{1-\theta}}.
		\]
		The contribution of this inner window consequently satisfies
		\begin{equation}\label{packing inner window}
			\max_{1\le r\le t^{k(1-\varepsilon)}}\left(\frac r{t^k}\right)^\alpha
			P_r(A\cap V_{J(r)}^{(t)})
			\yle_t t^{-\alpha\varepsilon k+\f{2\theta k}{1-\theta}}
			\le t^{-\f{\alpha\eps k}2}.
		\end{equation}
		For each $\ell>J(r)$,  we have
		$r\le t^{\ell(1-\theta)}$. Subadditivity of packing numbers and
		\[
		V_k^{(t)}=V_{J(r)}^{(t)}\cup
		\bigcup_{\ell=J(r)+1}^k S_\ell^{(t)}
		\]
		therefore give
		\[
		\mathcal P_{\alpha,\varepsilon}^{V,t}(k)
		\yle_t  t^{-\f{\alpha\eps k}2}+\sum_{\ell=1}^k t^{-\alpha(k-\ell)}
		\mathcal P_{\alpha,\theta}^{S,t}(\ell) .
		\]
		
		Since $
		\mathcal P_{\alpha,\theta}^{S,t}(\ell)\longrightarrow0,$
		the sequence $\bigl(\mathcal P_{\alpha,\theta}^{S,t}(\ell)\bigr)_{\ell\ge1}$
		is bounded. Similar to (i), for $K\in\mathbb N$ and all sufficiently large $k>K$, it follows that
		\[
		\mathcal P_{\alpha,\eps}^{V,t}(k)\yle_t t^{-\f{\alpha\eps k}2}+\sum_{\ell=1}^{K}t^{-\alpha(k-\ell)}
		\mathcal P_{\alpha,\theta}^{S,t}(\ell)+\sum_{\ell=K+1}^{k}t^{-\alpha(k-\ell)}
		\mathcal P_{\alpha,\theta}^{S,t}(\ell)
		\]
		Let first $k\to\infty$ and then $K\to\infty$.
		This proves 
		\[
		\lim_{k\to\infty}\mathcal P_{\alpha,\eps}^{S,t}(k)=0\text{\ for each\ }\eps\in (0,1)\quad\Longleftrightarrow\quad\lim_{k\to\infty}\mathcal P_{\alpha,\eps}^{V,t}(k)=0\text{\ for each\ }\eps\in (0,1).
		\]

		 We next compare bases $t,b>1$. Set $j(k)=\lceil k\log_b t\rceil$.
		Since
		$V_k^{(t)}\subseteq V_{j(k)}^{(b)}$ and
		$t^{k(1-\varepsilon)}\le b^{j(k)(1-\varepsilon)}$, we have
		\[
		\mathcal P_{\alpha,\eps}^{V,t}(k)
		\le b^\alpha\mathcal P_{\alpha,\varepsilon}^{V,b}(j(k)).
		\]
		Since $j(k)\to\infty$, interchanging $t$ and $b$ proves base
		independence.

		(iii)
		We first prove that $\delta_t(A)\le \dim_{\mathrm{UM}}A$.
		Fix $\alpha>\dim_{\mathrm{UM}}A$ and $\varepsilon\in(0,1)$.
		Since $r=1$ is admissible in the minimum defining $\delta_t(A)$,
		\[
		0\le
		\min_{1\le r\le t^{j(1-\varepsilon)}}
		\left(\frac{r}{t^j}\right)^\alpha
		P_r(A\cap S_j^{(t)})
		\le
		t^{-\alpha j}P_1(A\cap S_j^{(t)})
		\le
		t^{-\alpha j}\#(A\cap V_j^{(t)}).
		\]
		The last expression tends to zero by the definition of upper
		mass dimension. Thus $\delta_t(A)\le\alpha$, and letting
		$\alpha\downarrow\dim_{\mathrm{UM}}A$ proves the upper inequality.
		
		For the lower inequality, let $\alpha\ge0$ be such that, for
		some $\varepsilon\in(0,1)$,
		\[
		M_j:=
		\min_{1\le r\le t^{j(1-\varepsilon)}}
		\left(\frac{r}{t^j}\right)^\alpha
		P_r(A\cap S_j^{(t)})
		\longrightarrow0.
		\]
		As agreed, the minimum may be taken over integer scales.
		For each $j$ such that $A\cap S_j^{(t)}\ne\varnothing$, choose
		an integer $r_j$ attaining this minimum. Then
		\[
		\left(\frac{r_j}{t^j}\right)^\alpha
		P_{r_j}(A\cap S_j^{(t)})=M_j,
		\qquad
		\frac{r_j}{t^j}\le t^{-\varepsilon j}.
		\]
		
		Fix $\beta>\alpha$. By the covering--packing comparison,
		$A\cap S_j^{(t)}$ can be covered by at most
		$C P_{r_j}(A\cap S_j^{(t)})$ squares of side $r_j$, where
		$C$ is a constant. Each covering square can be
		enlarged to a lattice cube of side at most $r_j+2\le3r_j$.
		Since $d(S_j^{(t)})\asymp_t t^j$ for all sufficiently large
		$j$, this cover gives
		\[
		\nu_\beta(A,S_j^{(t)})\yle_{t,\beta}
		\left(\frac{r_j}{t^j}\right)^\beta
		P_{r_j}(A\cap S_j^{(t)})=
		\left(\frac{r_j}{t^j}\right)^{\beta-\alpha}M_j\le 
		t^{-\varepsilon j(\beta-\alpha)}M_j.
		\]
		The sequence $(M_j)_{j\ge 1}$ is bounded. Hence the right-hand side
		is summable in $j$. Annuli not meeting $A$ have zero cost,
		and the finitely many remaining initial terms have finite
		cost. Therefore
		\[
		\sum_{j=1}^{\infty}\nu_\beta(A,S_j^{(t)})<\infty.
		\]
		Part~\textup{(i)} now implies $\dim_{\mathrm{DH}}A\le\beta$.
		Letting $\beta\downarrow\alpha$ and then taking the infimum
		over all admissible $\alpha$ yields
		$\dim_{\mathrm{DH}}A\le\delta_t(A)$, completing the proof.
	\end{proof}
	
	For comparison, define the centered-window variant by
	\begin{equation}\label{window lower index definition}
		\delta_t^{V}(A)=\inf\biggl\{\alpha\ge0:\ 
		\lim_{k\to\infty}\min_{1\le r\le t^{k(1-\varepsilon)}}
		\left(\frac r{t^k}\right)^\alpha P_r(A\cap V_k^{(t)})=0
		\text{\, for some }\varepsilon\in(0,1)\biggr\}.
	\end{equation}
	Unlike the discrete Hausdorff, lower discrete Hausdorff, and discrete
	packing dimensions, the lower entropy index may change either when the
	base is changed or when annuli are replaced by windows, as the following
	two examples show.
	
	\begin{example}\label{lower entropy radix counterexample}
		For $k\ge2$, define
		\begin{gather*}
			F_k=\Big\{\Big(2^{4k-3}+\ell,0\Big):\ell\in\Z\cap[0,4^k)\Big\},\qquad
			G_k=\Big\{\Big(2^{4k-2}+ 2^{3k-3}\ell,0\Big):\ell\in\Z\cap[0,2^k)\Big\},\\
			A=\bigcup_{k\ge2}(F_k\cup G_k).
		\end{gather*}
		Then we have 
		\[
		\delta_2(A)=\frac14<\frac13=\delta_4(A).
		\]
	\end{example}
	\begin{proof}
		These sets satisfy
		\begin{gather*}
			F_k\subset\Big[2^{4k-3},2^{4k-2}\Big)\times\{0\}\subset S_{4k-1}^{(2)},\qquad
			G_k\subset\Big[2^{4k-2},2^{4k-1}\Big)\times\{0\}\subset S_{4k}^{(2)},\\
			F_k\cup G_k\subset\Big[2^{4k-3},2^{4k-1}\Big)\times\{0\}\subset S_{2k}^{(4)}.
		\end{gather*}
		Therefore, for all sufficiently large \(\ell\),
		\[
		A\cap S_\ell^{(2)}=\begin{cases}
			G_k,\quad & \ell=4k,\\
			F_k, & \ell=4k-1,\\
			\emptyset, &\ell=1,2\mod 4,
		\end{cases}\qquad A\cap S_\ell^{(4)}=\begin{cases}
			F_k\cup G_k,\quad & \ell=2k,\\
			\emptyset, &\ell\text{\, is odd}.
		\end{cases}
		\]
		For $1\le r\le 16^k$, interval packing gives, with absolute constants,
		\begin{equation}\label{counterexample packing counts}
			P_r(F_k)\asy\f{4^k}r+1,\qquad
			P_r(G_k)\asy\min\Big\{2^k,\f{2^{4k}}{8r}\Big\}.
		\end{equation}
		Also $P_r(F_k\cup G_k)$ is comparable to the maximum of these two
		quantities, by monotonicity and subadditivity.
		
		We first calculate $\delta_2(A)$. Fix $\varepsilon=1/4$. Since $4^k<2^{(4k-1)(1-\eps)}$ for sufficiently large $k$, it follows that for every $\alpha>0$,
		\begin{gather*}
			\left(\frac{4^k}{2^{4k-1}}\right)^\alpha P_{4^k}(A\cap S_{4k-1}^{(2)})=\left(\frac{4^k}{2^{4k-1}}\right)^\alpha P_{4^k}(F_k)
			\yle_\alpha 2\cdot 4^{-\alpha k},\\
			\left(\frac1{2^{4k}}\right)^\alpha P_{1}(A\cap S_{4k}^{(2)})=\left(\frac1{2^{4k}}\right)^\alpha P_{1}(G_k)\yle2^{k(1-4\alpha)}.
		\end{gather*}
		This gives $\delta_2(A)\le1/4$.
		Conversely, if $0\le\alpha<1/4$ and $\varepsilon\in(0,1)$ is arbitrary,
		\eqref{counterexample packing counts} gives, for
		$1\le r\le 2^{4k(1-\varepsilon)}$,
		\begin{align*}
			\left(\frac r{2^{4k}}\right)^\alpha P_r(A\cap S_{4k}^{(2)})&=\left(\frac r{2^{4k}}\right)^\alpha P_r(G_k)\\
			&\asy\begin{cases}
				r^\alpha2^{k(1-4\alpha)},\quad &1\le r\le 2^{3k-3},\\
				\f18 \big(\f{2^{4k}}r\big)^{1-\alpha},& 2^{3k-3}<r\le 2^{4k(1-\varepsilon)},
			\end{cases}\\
			&\yge\min\Big\{2^{k(1-4\alpha)},2^{4k\eps(1-\alpha)}\Big\}
			\longrightarrow\infty.
		\end{align*}
		Thus $\delta_2(A)=1/4$.
		
		For base $4$, take $\varepsilon=1/2$. Since $2^k< 4^{2k(1-\eps)}$,
		\eqref{counterexample packing counts} yields that 
		\begin{align*}
			\left(\frac {2^k}{4^{2k}}\right)^\alpha P_{2^k}(A\cap S_{2k}^{(4)})&=\left(\frac {2^k}{4^{2k}}\right)^\alpha P_{2^k}(F_k\cup G_k)\\
			&\le \left(\frac {2^k}{4^{2k}}\right)^\alpha \Big(P_{2^k}(F_k)+P_{2^k} (G_k)\Big)
			\yle_\alpha 2\cdot 2^{k(1-3\alpha)}.
		\end{align*}
		Hence $\de_4(A)\le \f 13$.
		For the reverse inequality, fix $0\le\alpha<1/3$ and an arbitrary
		$\varepsilon\in(0,1)$. We obtain
		\begin{align*}
			\left(\frac {r}{4^{2k}}\right)^\alpha P_r(A\cap S_{2k}^{(4)})&=\left(\frac {r}{4^{2k}}\right)^\alpha P_r(F_k\cup G_k)\\
			&\ge\begin{cases}
				(\frac {r}{4^{2k}})^\alpha P_r(F_k)\yge 2^{k(1-3\alpha)},\quad &1\le r\le 2^k,\\
				(\frac {r}{4^{2k}})^\alpha P_r(G_k)\yge 2^{k(1-3\alpha)}, &2^k< r\le 2^{3k-3},\\
				(\frac {r}{4^{2k}})^\alpha P_r(G_k)\yge 2^{4k\eps(1-\alpha)}, &2^{3k-3}< r\le 2^{4k(1-\eps)},
			\end{cases}\\
			&\longrightarrow\infty.
		\end{align*}
		Therefore $\de_4(A)=\f13$.
	\end{proof}
	
	\begin{example}\label{lower entropy window counterexample}
		For the set $A$ constructed in
		Example \ref{lower entropy radix counterexample},
		\[
		\delta_2(A)=\frac14<\frac13=\delta_2^V(A).
		\]
	\end{example}
	\begin{proof}
		The first equality has already been proved. To prove
		$\delta_2^V(A)\ge1/3$, note that
		$F_k\cup G_k\subset A\cap V_{4k}^{(2)}$.
		For $0\le\alpha<1/3$ and every $\varepsilon\in(0,1)$, the
		three scale estimates in the preceding example give
		\[
		\begin{split}
			\min_{1\le r\le 2^{4k(1-\varepsilon)}}
			\left(\frac r{4^{2k}}\right)^\alpha P_r(A\cap V_{4k}^{(2)})
			\ge
			\min\big\{2^{k(1-3\alpha)},2^{4k\eps(1-\alpha)}\big\}
			\longrightarrow\infty.
		\end{split}
		\]
		Hence no $\varepsilon$ gives the convergence in
		\eqref{window lower index definition} at such an $\alpha$.
		
		For the upper bound, let $\alpha>1/3$. Recall that $F_k\cup G_k\subset\big[2^{4k-3},2^{4k-1}\big)\times\{0\}$.
		If $F_k$ or $G_k$ intersects
		$V_\ell^{(2)}$, then $2^{4k-3}<2^{\ell-1},$ i.e., $4k<\ell+2$.
		Subadditivity and \eqref{counterexample packing counts} then yield
		\begin{align*}
			P_r(A\cap V_\ell^{(2)})&\le\sum_{\substack{k\ge2\\4k<\ell+2}}P_r(F_k\cup G_k)\\
			&\yle\sum_{\substack{k\ge2\\4k<\ell+2}}
			\left(1+\frac{4^k}r+2^k\right)
			\yle \ell+\frac{2^{\f\ell2}}r+2^{\f\ell4}.
		\end{align*}
		Choose $\varepsilon=1/2$. Since $2^{\f\ell4}<2^{\ell(1-\eps)}$,
		\[
		\min_{1\le r\le2^{\ell(1-\eps)}}
		\left(\frac r{2^\ell}\right)^\alpha P_r(A\cap V_\ell^{(2)})
		\le\left(\frac{2^{\f\ell4}}{2^\ell}\right)^\alpha
		P_{2^{\f\ell4}}(A\cap V_\ell^{(2)})
		\yle2^{-\f{3\alpha}4\ell}\big(\ell+2^{\f\ell4+1}\big)\longrightarrow0.
		\]
		This proves $\delta_2^V(A)\le1/3$ and completes the calculation.
	\end{proof}

	\section[Finite symbolic models and orbit localization]{Finite symbolic models and orbit localization}\label{section model}
	
	Throughout this section, $m>n$, $z=(u,v)\in\Z^2$. Put $\mI=\{i:N_i>0\}$,
	$\tau=\log m/\log n>1$, and $q_k=q(m^k)=\lfloor\tau k\rfloor$.
	For $k\geq1$, we identify each element of
	\[
	\La^k\times\{0,\ldots,n-1\}^{q_k-k}
	\]
	with a pair of words $
	(i_0\ldots i_{k-1},j_0\ldots j_{q_k-1}),$ 
	where $(i_\ell,j_\ell)\in\La$ for $0\leq\ell\leq k-1$ and
	$j_\ell\in\{0,\ldots,n-1\}$ for $k\leq\ell\leq q_k-1$. Define $
	\Psi_k:\La^k\times\{0,\ldots,n-1\}^{q_k-k}\longrightarrow\Z^2$
	by
	\[
	\Psi_k(i_0\ldots i_{k-1},j_0\ldots j_{q_k-1})=\left(\sum_{\ell=0}^{k-1}i_\ell m^\ell,\sum_{\ell=0}^{q_k-1}j_\ell n^\ell\right).
	\]
	Finite positional expansions make $\Psi_k$ injective.
	Let
	$\mbB=(\mB_\ell)_{\ell=k}^{q_k-1}$ be a sequence of nonempty
	subsets of $\{0,\ldots,n-1\}$, and define
	$\Omega_k[\mbB]=\La^k\times\prod_{\ell=k}^{q_k-1}\mB_\ell.$
	The corresponding finite
	model is the planar set
	\begin{equation}\label{varying tail model}
		\mE_k[\mbB]=\Psi_k(\Omega_k[\mbB]).
	\end{equation}
	Since
	$n^{q_k}\le m^k$, we have
	\begin{equation}\label{varying model cardinality}
		\mE_k[\mbB]\subset ([0,m^k)\times[0,n^{q_k}))\cap\Z^2\subseteq C(0,m^k),\qquad
		\#\mE_k[\mbB]=(\#\La)^k
		\prod_{\ell=k}^{q_k-1}\#\mB_\ell.
	\end{equation}
	Here and throughout, we call the pairs $(i_\ell,j_\ell)$,
	$0\leq\ell<k$, the \emph{mixed digits}, with $i_\ell$ and $j_\ell$
	being their horizontal and vertical components, respectively. The digits
	$j_\ell$, $k\leq\ell<q_k$, are called the \emph{tail digits}.
	Furthermore, we define the constant tail alphabet
	\[
	\mB_u=\begin{cases}
		\col_0,&u=0\text{ and }N_0\ge2,\\
		\col_{m-1},&u=-1\text{ and }N_{m-1}\ge2,\\
		\{0\},&\eta(u)=1.
	\end{cases}
	\]
	and denote $(\{0\})_{\ell=k}^{q_k-1}$ by $\mathbf{0}$.
	In particular, let $\mbB^u=(\mB_u)_{\ell=k}^{q_k-1}$ and retain the notation
	\[
	\Omega_k(u)=\Omega_k[\mbB^u]
	=\La^k\times\mB_u^{q_k-k},\qquad
	\mE_{k,q_k}(u)=\mE_k[\mbB^u].
	\]
	For a measure $\mu$ on $\Omega_k[\mbB]$, its pushforward to
	$\Z^2$ is denoted by $(\Psi_k)_*\mu$ and is supported on
	$\mE_k[\mbB]$. Taking $\mbB=\mbB^u$ in
	\eqref{varying model cardinality} gives
	\[
	\mE_{k,q_k}(u)\subset C(0,m^k),\qquad
	\#\mE_{k,q_k}(u)=(\#\La)^k\eta(u)^{q_k-k}.
	\]
	
	\subsection{Approximate squares and their counts}
	Fix $k\ge 1$.
	For integers $0\le s\le k$, put
	\[
	r_s=m^{k-s},\qquad t_s=\max\{0,k-q(r_s)\},\qquad
	e_s=q_k-\max\{k,q(r_s)\}.
	\]
	These numbers satisfy
	\begin{equation}\label{levelsforapproximatesquare}
		0\le t_s\le s,\qquad e_s\ge0,\qquad
		t_s+e_s=q_k-q(r_s)=\tau s+O(1),
	\end{equation}
	where the last error has absolute value less than $1$.
	For a fixed model $\mE_k[\mbB]$, a depth-$s$ approximate
	square is the image under $\Psi_k$ of a nonempty subset of
	$\Omega_k[\mbB]$ obtained by fixing the mixed horizontal
	digits $i_\ell$ for $k-s\le\ell<k$, the mixed vertical digits
	$j_\ell$ for $k-t_s\le\ell<k$, and the tail digits $j_\ell$ for
	$q_k-e_s\le\ell<q_k$. More precisely, for $\omega\in\Omega_k[\mbB]$, define
	\[
	Q_s(\omega)
	=
	\Psi_k\left(
	\left\{
	\omega'\in\Omega_k[\mbB]:
	\begin{array}{ll}
		i'_\ell=i_\ell, & k-s\le\ell<k,\\
		j'_\ell=j_\ell, & q(r_s)\le\ell<q_k
	\end{array}
	\right\}
	\right),
	\]
	where $i_\ell,j_\ell$ and $i'_\ell,j'_\ell$ denote the digits
	of $\omega$ and $\omega'$, respectively.
	The distinct sets $Q_s(\omega)$ form a partition
	\[
	\mQ_s(k;\mbB)
	=
	\{Q_s(\omega):\omega\in\Omega_k[\mbB]\}
	\]
	of $\mE_k[\mbB]$. In particular, put
	$\mQ_s(k,u)=\mQ_s(k;\mbB^u)$.
	Both prescribed digit ranges increase with depth, so
	$Q_t(\omega)\subseteq Q_s(\omega)$ for $0\le s\le t\le k$;
	in particular, these partitions are nested.
	The depth-$0$ partition consists of
	$\mE_k[\mbB]$ alone, and the depth-$k$ partition consists of
	singletons. 
	
	\begin{lemma}\label{approximate square geometry}
		For every choice of $\mbB$, every depth-$s$ approximate
		square is contained in a lattice cube of
		side $r_s$. Every cube of side at most $r_s$ intersects at most $2(n+2)$
		depth-$s$ approximate squares. Consequently,
		\begin{equation}\label{varying model partition count}
			\begin{split}
				N_{r_s}(\mE_k[\mbB])
				\asy P_{r_s}(\mE_k[\mbB])
				\asy\#\mQ_s(k;\mbB)=N_\La^{s-t_s}(\#\La)^{t_s}
				\prod_{\ell=q_k-e_s}^{q_k-1}\#\mB_\ell.
			\end{split}
		\end{equation}
		In particular,
		\begin{equation}\label{model partition count}
			N_{r_s}(\mE_{k,q_k}(u))\asy P_{r_s}(\mE_{k,q_k}(u))
			\asy_{n}\#\mQ_s(k,u)
			=N_\La^{s-t_s}(\#\La)^{t_s}\eta(u)^{e_s}.
		\end{equation}
	\end{lemma}
	
	\begin{proof}
		Fix $s\le k$ and $\mbB$. For $(u,v)\in\Z^2$, let
		\[
		R_{u,v}=[ur_s,(u+1)r_s)\times[vn^{q(r_s)},(v+1)n^{q(r_s)}).
		\]
		If $(x,y)=\Psi_k(i_0\ldots i_{k-1},j_0\ldots j_{q_k-1})$, then
		\[
		\left\lfloor\frac{x}{r_s}\right\rfloor=\sum_{\ell=k-s}^{k-1}i_\ell m^{\ell-k+s},
		\qquad
		\left\lfloor\frac{y}{n^{q(r_s)}}\right\rfloor=
		\sum_{\ell=q(r_s)}^{q_k-1}j_\ell n^{\ell-q(r_s)}.
		\]
		By uniqueness of finite positional expansions, two model
		points lie in the same rectangle $R_{u,v}$ if and only if
		their symbolic representatives agree on all the digits
		prescribed in the definition of $Q_s$. Consequently,
		\[
		\mQ_s(k;\mbB)=\big\{\mE_k[\mbB]\cap R_{u,v}:(u,v)\in\Z^2,\ \mE_k[\mbB]\cap R_{u,v}\ne\emptyset
		\big\}.
		\]
		Since $n^{q(r_s)}\le r_s$, each member of this partition
		satisfies
		\[
		\mE_k[\mbB]\cap R_{u,v}\subset C\bigl((ur_s,vn^{q(r_s)}),r_s\bigr).
		\]
		This proves the first geometric assertion and gives $
		N_{r_s}(\mE_k[\mbB])\le\#\mQ_s(k;\mbB).$
		
		Let $U$ be an axis-parallel square of side at most $r_s$.
		An interval of length $L$ intersects at most $\lceil L/h\rceil+1$
		members of the partition
		$\{[uh,(u+1)h):u\in\Z\}$.
		Applying this to the two coordinate projections of $U$
		shows that $U$ intersects at most
		\[
		2\left(\left\lceil\frac{r_s}{n^{q(r_s)}}\right\rceil+1\right)\le 2(n+2)
		\]
		rectangles $R_{u,v}$.
		It therefore intersects at most $2(n+2)$ members of
		$\mQ_s(k;\mbB)$. Every cover of $\mE_k[\mbB]$ must intersect every nonempty
		member of this partition. Thus
		\[
		\frac{\#\mQ_s(k;\mbB)}{2(n+2)}
		\le
		N_{r_s}(\mE_k[\mbB])
		\le
		\#\mQ_s(k;\mbB).
		\]
		Together with $N_r\asy P_r$, this proves the required
		comparisons with covering and packing numbers.
		
		It remains to count the members of $\mQ_s(k;\mbB)$.
		The prescribed data consist of
		\[
		\begin{array}{lll}
			i_\ell,
			& k-s\le\ell<k-t_s,\quad
			& N_\La \text{ choices at each level},\\
			(i_\ell,j_\ell),\quad
			& k-t_s\le\ell<k,
			& \#\La \text{ choices at each level},\\
			j_\ell,
			& q_k-e_s\le\ell<q_k,
			& \#\mB_\ell \text{ choices at each level}.
		\end{array}
		\]
		These choices are independent between levels.
		Distinct choices give distinct members of the partition,
		by the injectivity of $\Psi_k$. Hence
		\[
		\#\mQ_s(k;\mbB)
		=
		N_\La^{s-t_s}(\#\La)^{t_s}
		\prod_{\ell=q_k-e_s}^{q_k-1}\#\mB_\ell.
		\]
		This proves \eqref{varying model partition count}.
		Taking $\mbB=\mbB^u$, so that
		$\#\mB_\ell=\eta(u)$ at every tail level, gives
		\eqref{model partition count}.
	\end{proof}

	\begin{lemma}\label{varying model counts}
		Uniformly over $k\ge1$, $1\le r\le m^k$, and all sequences
		$\mbB$ of nonempty tail alphabets,
		\begin{equation}\label{varying model all scales}
			\begin{split}
				N_r(\mE_k[\mbB])
				\asy P_r(\mE_k[\mbB])
				\asy_{m,n}
				N_\La^{\min\{q(r),k\}-p(r)}
				(\#\La)^{\max\{k-q(r),0\}}
				\prod_{\ell=\max\{k,q(r)\}}^{q_k-1}\#\mB_\ell.
			\end{split}
		\end{equation}
		If $\#\mB_a\le\Nm$ for every $a$, then
		\begin{equation}\label{model maximal fiber bound}
			N_r(\mE_k[\mbB])
			\yle_{m,n}
			N_\La^{k-p(r)}\Nm^{q_k-q(r)}.
		\end{equation}
		All constants are independent of the tail alphabets.
	\end{lemma}
	\begin{proof}
		Put $s=k-p(r)$.
		Since $r_s=m^{p(r)}$, we have
		\[
		r_s\le r<mr_s,\qquad 0\le q(r)-q(r_s)\le\lceil\tau\rceil.
		\]
		Every square of side $r$ can be covered by at most $m^2$
		squares of side $r_s$. Consequently,
		\[
		N_r(\mE_k[\mbB])\le N_{r_s}(\mE_k[\mbB])
		\le m^2N_r(\mE_k[\mbB]).
		\]
		
		For integers $p(r)\le d\le q_k$, define
		\[
		M_d=N_\La^{\min\{d,k\}-p(r)}
		(\#\La)^{\max\{k-d,0\}}
		\prod_{\ell=\max\{k,d\}}^{q_k-1}\#\mB_\ell.
		\]
		The definitions of $s$, $t_s$, and $e_s$ give
		\[
		s-t_s=\min\{q(r_s),k\}-p(r),\qquad
		t_s=\max\{k-q(r_s),0\},\qquad
		e_s=q_k-\max\{k,q(r_s)\}.
		\]
		Thus \eqref{varying model partition count} yields $
		\#\mQ_s(k;\mbB)=M_{q(r_s)},$
		and the preceding scale comparison gives $N_r(\mE_k[\mbB])\asy_{m,n}M_{q(r_s)}.$
		We now compare $M_{q(r_s)}$ with $M_{q(r)}$. Direct cancellation
		shows that, for $p(r)\le d<q_k$,
		\[
		\frac{M_d}{M_{d+1}}=\begin{cases}
			\displaystyle\frac{\#\La}{N_\La},&d<k,\\[6pt]
			\#\mB_d,&d\ge k.
		\end{cases}
		\]
		Both expressions lie in $[1,n]$.
		Multiplying these ratios  gives
		\[
		M_{q(r)}\le M_{q(r_s)}\le n^{q(r)-q(r_s)}M_{q(r)}\le n^{\lceil\tau\rceil}M_{q(r)}.
		\]
		Therefore $N_r(\mE_k[\mbB])\asy_{m,n}M_{q(r)}.$
		Together with $N_r\asy P_r$, this proves
		\eqref{varying model all scales}.
		
		Finally, suppose that $\#\mB_\ell\le\Nm$ at every tail
		level. Since $\#\La\le N_\La\Nm$,
		we obtain
		\[
		M_{q(r)}\le N_\La^{\min\{q(r),k\}-p(r)}
		(N_\La\Nm)^{\max\{k-q(r),0\}}
		\Nm^{q_k-\max\{k,q(r)\}}=N_\La^{k-p(r)}\Nm^{q_k-q(r)}.
		\]
		Combining this with the estimate for $N_r(\mE_k[\mbB])$
		proves \eqref{model maximal fiber bound}.
		All comparison constants depend only on $m$ and $n$.
	\end{proof}

	For $1\le r\le h$, define the single expression
	\begin{align}
		\Phi(u,r,h)&=N_\La^{\min\{q(r),p(h)\}-p(r)}
		(\#\La)^{\max\{p(h)-q(r),0\}}
		\eta(u)^{q(h)-\max\{p(h),q(r)\}}\notag\\
		&=\begin{cases}
			N_\La^{q(r)-p(r)}(\#\La)^{p(h)-q(r)}\eta(u)^{q(h)-p(h)},\ & \text{if\ \, }p(r)\le q(r)\le p(h)\le q(h),\\
			N_\La^{p(h)-p(r)}\eta(u)^{q(h)-q(r)}, &\text{if\ \, }p(r)\le p(h)\le q(r)\le q(h).\label{Phi definition}
		\end{cases}
	\end{align}
	All exponents are nonnegative.
	Taking $\mbB=\mbB^u$ in
	Lemma \ref{varying model counts} gives
	\begin{equation}\label{model all scales}
		N_r(\mE_{k,q_k}(u))\asy_{m,n}
		P_r(\mE_{k,q_k}(u))\asy_{m,n}\Phi(u,r,m^k).
	\end{equation}
	A fixed multiplicative change in $h$ changes
	each exponent by a bounded amount, including when $q(r)=p(h)$ is
	crossed. In particular, whenever both sides are defined,
	\begin{equation}\label{Phi comparable scales}
		\Phi(u,r,ch)\asy_{m,n,\La,c}\Phi(u,r,h)
		\qquad(c>0\text{ fixed}).
	\end{equation}
	
	\begin{lemma}\label{model descendants}
		Fix an integer $S\ge 0$. There is $C_S>0$, independent of $k$
		and of $Q\in\mQ_S(k,u)$, such that
		\[
		P_r(Q)\ge C_S\Phi(u,r,m^k)
		\]
		for all $k\ge S$ and $1\le r\le m^{k-S}$.
	\end{lemma}
	\begin{proof}
		Fix $Q\in\mQ_S(k,u)$ and $1\le r\le m^{k-S}$, and set $
		s=k-p(r).$
		Then $S\le s\le k$ and $r_s\le r<mr_s.$
		Let
		\[
		D_s(Q)
		=
		\#\{Q'\in\mQ_s(k,u):Q'\subseteq Q\}.
		\]
		By nestedness, the depth-$s$ approximate squares counted by
		$D_s(Q)$ form a partition of $Q$.
		Let $i_\ell$, $k-S\le\ell<k$, be the horizontal digits
		prescribed by $Q$, and set $
		d=\min\{S,t_s\}.$
		Since $s\ge S$, we have $t_s\ge t_S$ and $e_s\ge e_S$.
		In particular, $d\ge t_S$.
		For $Q'\in\mQ_s(k,u)$ with $Q'\subseteq Q$, choose
		$\omega'\in\Psi_k^{-1}(Q')$, and write $i'_\ell,j'_\ell$
		for its digits. The depth-$s$ digit data not already
		prescribed by $Q$ are exactly the following:
		\[
		\begin{array}{c|c|c}
			\text{levels }\ell& \text{remaining data}& \text{admissible values}\\ \hline
			k-s\le\ell<k-\max\{S,t_s\}
			& i'_\ell
			& \mI\\
			k-t_s\le\ell<k-d
			& (i'_\ell,j'_\ell)
			& \La\\
			k-d\le\ell<k-t_S
			& j'_\ell
			& \col_{i_\ell}\\
			q_k-e_s\le\ell<q_k-e_S
			& j'_\ell
			& \mB_u
		\end{array}
		\]
		Consequently,
		\[
		\begin{aligned}
			D_s(Q)=
			N_\La^{s-\max\{S,t_s\}}
			(\#\La)^{t_s-d}
			\eta(u)^{e_s-e_S}
			\prod_{\ell=k-d}^{k-t_S-1}N_{i_\ell}.
		\end{aligned}
		\]
		Using $\#\mQ_s(k,u)=N_\La^{s-t_s}(\#\La)^{t_s}\eta(u)^{e_s}$
		and
		$\max\{S,t_s\}=S+t_s-d$, we obtain
		\[
		\frac{\#\mQ_s(k,u)}{D_s(Q)}=N_\La^{S-d}\f{(\#\Lambda)^d}{\prod_{\ell=k-d}^{k-t_S-1}N_{i_\ell}}\eta(u)^{e_S}\le (\#\Lambda)^S\eta(u)^{e_S}\le(\#\Lambda)^S\eta(u)^{\tau S+1} ,
		\]
		where we used $N_{i_\ell}\ge1$ and $N_\La\le\#\La$.

		By Lemma \ref{approximate square geometry}, every square
		of side $r_s$ intersects at most $2(n+2)$ depth-$s$
		approximate squares. Since the $D_s(Q)$ squares above are nonempty and
		partition $Q$,
		\[
		N_{r_s}(Q)\ge\frac{D_s(Q)}{2(n+2)}.
		\]
		Moreover, an $r$-square can be covered by at most $m^2$
		squares of side $r_s$, so $N_r(Q)\ge m^{-2}N_{r_s}(Q).$
		Using $N_r\asy P_r$, we conclude that $P_r(Q)\yge_{m,n}D_s(Q).$
		
		Finally, Lemmas \ref{approximate square geometry}
		and \ref{varying model counts} give
		\[
		\#\mQ_s(k,u)\ge N_{r_s}(\mE_{k,q_k}(u))\ge N_r(\mE_{k,q_k}(u))
		\yge_{m,n}\Phi(u,r,m^k).
		\]
		Combining the preceding estimates yields
		\[
		P_r(Q)\yge_{m,n}D_s(Q)\ge (\#\La)^{-S}\eta(u)^{-\tau S-1}\#\mQ_s(k,u)
		\yge_{m,n}
		(\#\La)^{-S}\eta(u)^{-\tau S-1}\Phi(u,r,m^k).\qedhere
		\]
	\end{proof}

	\subsection{Orbit blocks and translated models}
	
	For $U,V\in\Z$, set
	\[
	X_U^{(k)}=\Big\{(x,y)\in\Z^2:\Big\lfloor \f x{m^k}\Big\rfloor=U\Big\},
	\qquad
	G_{U,V}^{(k)}
	=X_U^{(k)}\cap
	\Big\{(x,y)\in\Z^2:\Big\lfloor \f y{n^{q_k}}\Big\rfloor=V\Big\}.
	\]
	Thus $X_U^{(k)}$ is a vertical strip, and $G_{U,V}^{(k)}$
	is a rectangular grid cell.
	\begin{lemma}\label{orbit model blocks}
		Let $z=(u,v)\in\Z^2$. For every $k\ge1$, the following statements hold.
		\begin{enumerate}[(i)]
			
			\item
			For each $1\le\ell\le k$, there is $T_{\ell,k}\in\Z^2$
			such that
			\begin{equation}\label{short word model inclusion}
				\La_\ell(z)\subseteq T_{\ell,k}+\mE_k[\mathbf0].
			\end{equation}
			
			\item
			Let $k\le\ell\le q_k$ and suppose that
			$\La_\ell(z)\cap X_U^{(k)}\ne\emptyset$.
			There are unique digits $i_d(\ell,U)\in\mI$,
			$k\le d<\ell$, satisfying
			\[
			U-m^{\ell-k}u=
			\sum_{d=k}^{\ell-1}i_d(\ell,U)m^{d-k}.
			\]
			Define $\mbB^{\ell,U}=(\mB_d^{\ell,U})_{d=k}^{q_k-1}$ by
			\[
			\mB_d^{\ell,U}=
			\begin{cases}
				\col_{i_d(\ell,U)},&k\le d<\ell,\\
				\{0\},&\ell\le d<q_k.
			\end{cases}
			\]
			Then
			\begin{equation}\label{fixed length model block}
				\La_\ell(z)\cap X_U^{(k)}
				=
				(m^kU,n^\ell v)+\mE_k[\mbB^{\ell,U}].
			\end{equation}
			
			\item
			Suppose
			that $\Big(\bigcup_{\ell>q_k}\La_\ell(z)\Big)\cap G_{U,V}^{(k)}\ne\varnothing$.
			There are unique digits $i_d(U)\in\mI$,
			$k\le d<q_k$, satisfying
			\[
			U-m^{q_k-k}
			\left\lfloor\frac{U}{m^{q_k-k}}\right\rfloor
			=
			\sum_{d=k}^{q_k-1}i_d(U)m^{d-k}.
			\]
			Set $\mbB^U=(\col_{i_d(U)})_{d=k}^{q_k-1}$.
			Then
			\begin{equation}\label{long word model block}
				\Big(\bigcup_{\ell>q_k}\La_\ell(z)\Big)\cap G_{U,V}^{(k)}=(m^kU,n^{q_k}V)+\mE_k[\mbB^U].
			\end{equation}
			In particular, the tail alphabets depend only on $k$ and $U$.
			
			\item
			For any $\mathbf i=(i_d)_{d=k}^{q_k-1}\in\mI^{q_k-k}$,
			set $\mbB^{\mathbf i}=(\col_{i_d})_{d=k}^{q_k-1}$.
			Then
			\begin{equation}\label{model orbit embedding}
				\left(
				m^{q_k}u+\sum_{d=k}^{q_k-1}i_dm^d,\,
				n^{q_k}v
				\right)
				+\mE_k[\mbB^{\mathbf i}]
				\subseteq\La_{q_k}(z)\subseteq \fola (z).
			\end{equation}
		\end{enumerate}
		
		Consequently, for every $\omega\in\Z^2$ and
		$k\le\ell\le q_k$, the set
		$\La_\ell(z)\cap C(\omega,m^k)$ is contained in at most
		two translated models, while
		$\Big(\bigcup_{\ell>q_k}\La_\ell(z)\Big)\cap C(\omega,m^k)$ is contained in at most
		$2(n+1)$ translated models.
		All these models have tail alphabets of cardinality at
		most $\Nm$.
		The corresponding intersections with $Q(w,m^k)$ are
		contained in $O_{m,n}(1)$ translated models, uniformly
		in $w\in\R^2$.
	\end{lemma}
	
	\begin{proof}
		Recall that every point $(x,y)\in\La_\ell((u,v))$ has a
		representation
		\[
		x=m^\ell u+\sum_{d=0}^{\ell-1}i_dm^d,
		\qquad
		y=n^\ell v+\sum_{d=0}^{\ell-1}j_dn^d,
		\]
		where $(i_d,j_d)\in\La$ for $0\le d<\ell$.
		
		For (i), fix $(i_*,j_*)\in\La$ and set
		\[
		T_{\ell,k}=\left(
		m^\ell u-\sum_{d=\ell}^{k-1}i_*m^d,\,
		n^\ell v-\sum_{d=\ell}^{k-1}j_*n^d\right).
		\]
		Then
		\begin{align*}
			(x,y)-T_{\ell,k}=\left(
			\sum_{d=0}^{\ell-1}i_dm^d+\sum_{d=\ell}^{k-1}i_*m^d,
			\sum_{d=0}^{\ell-1}j_dn^d+\sum_{d=\ell}^{k-1}j_*n^d
			\right).
		\end{align*}
		This point belongs to $\mE_k[\mathbf0]$ since its mixed
		word consists of the original $\ell$ symbols followed
		by $k-\ell$ copies of $(i_*,j_*)$, and its tail digits
		are zero. This proves \eqref{short word model inclusion}.
		
		For (ii), note that
		\[
		\left\lfloor\frac{x}{m^k}\right\rfloor=
		m^{\ell-k}u+\sum_{d=k}^{\ell-1}i_dm^{d-k}.
		\]
		Thus membership in $X_U^{(k)}$ is equivalent to
		\[
		\sum_{d=k}^{\ell-1}i_dm^{d-k}=
		U-m^{\ell-k}u.
		\]
		Uniqueness of the finite base-$m$ expansion determines
		the digits $i_d=i_d(\ell,U)$ for $k\le d<\ell$.
		The assumed nonemptiness ensures that these digits
		belong to $\mI$.
		
		The remaining admissible choices are precisely
		\[
		(i_d,j_d)\in\La\quad(0\le d<k),\qquad j_d\in\col_{i_d(\ell,U)}
		\quad(k\le d<\ell).
		\]
		For these choices,
		\[
		(x,y)=(m^kU,n^\ell v)+\left(
		\sum_{d=0}^{k-1}i_dm^d,\,
		\sum_{d=0}^{k-1}j_dn^d
		+\sum_{d=k}^{\ell-1}j_dn^d\right).
		\]
		The second vector ranges over exactly
		$\mE_k[\mbB^{\ell,U}]$, with zero tail digits at
		levels $\ell\le d<q_k$. This proves
		\eqref{fixed length model block}.
		
		For (iii), consider a point of
		$\Big(\bigcup_{\ell>q_k}\La_\ell(z)\Big)\cap G_{U,V}^{(k)}$ represented by a word
		of length $\ell>q_k$. Its quotients satisfy
		\[
		U=m^{\ell-k}u+\sum_{d=k}^{\ell-1}i_dm^{d-k},
		\qquad
		V=n^{\ell-q_k}v+\sum_{d=q_k}^{\ell-1}j_dn^{d-q_k}.
		\]
		Reducing the first identity modulo $m^{q_k-k}$ gives
		\[
		U\equiv\sum_{d=k}^{q_k-1}i_dm^{d-k}\pmod{m^{q_k-k}}.
		\]
		Its finite base-$m$ expansion therefore determines
		$i_d=i_d(U)$ for $k\le d<q_k$, independently of $\ell$.
		
		Furthermore,
		\[
		(x,y)-(m^kU,n^{q_k}V)=\left(
		\sum_{d=0}^{k-1}i_dm^d,\,
		\sum_{d=0}^{q_k-1}j_dn^d\right)
		\in\mE_k[\mbB^U].
		\]
		This proves one inclusion in
		\eqref{long word model block}.
		
		For the reverse inclusion, choose a word
		$((i_d',j_d'))_{d=0}^{\ell'-1}$ of length
		$\ell'>q_k$ whose orbit point belongs to
		$G_{U,V}^{(k)}$.
		For any $\omega\in\Omega_k[\mbB^U]$, with mixed
		digits $(i_d(\omega),j_d(\omega))$ and tail digits $j_d(\omega)$, define
		a word of length $\ell'$ by
		\[
		(\widehat i_d,\widehat j_d)=
		\begin{cases}
			(i_d(\omega),j_d(\omega)),&0\le d<k,\\
			(i_d(U),j_d(\omega)),&k\le d<q_k,\\
			(i_d',j_d'),&q_k\le d<\ell'.
		\end{cases}
		\]
		Every symbol in this word belongs to $\La$.

		The preceding quotient identities now give
		\[
		m^{\ell'}u+\sum_{d=0}^{\ell'-1}\widehat i_dm^d=
		m^kU+\sum_{d=0}^{k-1}i_d(\omega)m^d,\qquad
		n^{\ell'}v
		+\sum_{d=0}^{\ell'-1}\widehat j_dn^d=
		n^{q_k}V+\sum_{d=0}^{q_k-1}j_d(\omega)n^d.
		\]
		Hence
		\[
		(m^kU,n^{q_k}V)+\Psi_k(\omega)
		\in\La_{\ell'}(z)\cap G_{U,V}^{(k)}.
		\]
		Since $\omega$ was arbitrary, the reverse inclusion
		follows, proving \eqref{long word model block}.
		
		For (iv), set
		\[
		U_{\mathbf i}=m^{q_k-k}u
		+\sum_{d=k}^{q_k-1}i_dm^{d-k}.
		\]
		Since every $i_d$ belongs to $\mI$,
		$\La_{q_k}(z)\cap X_{U_{\mathbf i}}^{(k)}$ is nonempty.
		Applying (ii) with $\ell=q_k$ and $U=U_{\mathbf i}$
		gives \eqref{model orbit embedding}.
		
		Finally, $C(\omega,m^k)$ intersects at most two vertical strips
		$X_U^{(k)}$. Since
		$n^{q_k}\le m^k<n^{q_k+1}$, it intersects at most
		$n+1$ horizontal grid strips of height $n^{q_k}$,
		and hence at most $2(n+1)$ cells $G_{U,V}^{(k)}$.
		The asserted model inclusions follow from (ii) and
		(iii). Each tail alphabet is either $\{0\}$ or a
		column fiber $\col_i$, so its cardinality is at most
		$\Nm$.
		
		For $Q(w,m^k)$, whose side length is $2m^k$, the
		corresponding numbers of vertical and horizontal
		strips are at most $3$ and $2n+1$, respectively.
		This proves the final assertion uniformly in $w$.
	\end{proof}

	\subsection[Centered windows and model transfer]{Centered windows and model transfer}
	
	\begin{lemma}\label{centered window}
		Let $z=(u,v)\in\Z^2$ and $h\ge2$.
		Every point of $\fola(z)\cap [-h,h]^2$ has an admissible
		word representation of length at most $q(h)+1$.
		
		Moreover, if an admissible word of length $k>p(h)+1$
		represents a point of $\fola(z)\cap [-h,h]^2$, then
		\[
		u\in\{0,-1\},\qquad
		i_\ell=-(m-1)u
		\quad\text{for\ }p(h)+1\le \ell<k.
		\]
	\end{lemma}
	
	\begin{proof}
		Consider any admissible representation
		\[
		x=m^ku+\sum_{\ell=0}^{k-1}i_\ell m^\ell,
		\qquad
		y=n^kv+\sum_{\ell=0}^{k-1}j_\ell n^\ell
		\]
		of a point $(x,y)\in \fola(z)\cap[-h,h]^2$.
		
		We first determine the restrictions on the horizontal
		digits. Since $0\le i_\ell\le m-1,\
		m^ku\le x\le m^k(u+1)-1.$
		
		If $u\notin\{0,-1\}$, this gives $|x|\ge m^k$.
		Hence $m^k\le h$, and therefore $k\le p(h)$.
		
		If $u\in\{0,-1\}$, the horizontal coordinate satisfies
		\[
		|x|=
		\begin{cases}
			\displaystyle\sum_{\ell=0}^{k-1}i_\ell m^\ell,
			&u=0,\\[6pt]
			\displaystyle
			1+\sum_{\ell=0}^{k-1}(m-1-i_\ell)m^\ell,\quad
			&u=-1.
		\end{cases}
		\]
		Since $|x|\le h<m^{p(h)+1}$,
		these identities imply
		\[
		i_\ell=
		\begin{cases}
			0,&u=0,\\
			m-1,&u=-1,
		\end{cases}
		\qquad p(h)+1\le \ell<k.
		\]
		Together with the bound $k\le p(h)$ when
		$u\notin\{0,-1\}$, this proves the horizontal assertion.
		
		The same calculation for the vertical coordinate gives
		\[
		\begin{cases}
			k\le q(h),&v\notin\{0,-1\},\\
			j_\ell=-(n-1)v\quad(q(h)+1\le\ell<k),\quad
			&v\in\{0,-1\}.
		\end{cases}
		\]
		
		It remains to obtain a representation of the required
		length. If $k\le q(h)+1$, there is nothing to prove.
		Suppose that $k>q(h)+1$. The preceding bounds force
		$u,v\in\{0,-1\}$. Since $p(h)\le q(h)$, we have
		\[
		(i_\ell,j_\ell)=
		\bigl(-(m-1)u,-(n-1)v\bigr)
		\qquad(q(h)+1\le \ell<k).
		\]
		Consequently,
		\begin{gather*}
			m^ku+\sum_{\ell=q(h)+1}^{k-1}i_\ell m^\ell=
			m^ku-(m-1)u\sum_{\ell=q(h)+1}^{k-1}m^\ell
			=m^{q(h)+1}u,\\
			n^kv+\sum_{\ell=q(h)+1}^{k-1}j_\ell n^\ell=
			n^kv-(n-1)v\sum_{\ell=q(h)+1}^{k-1}n^\ell=n^{q(h)+1}v.
		\end{gather*}
		Thus the same point can be written as
		\[
		(x,y)=\left(m^{q(h)+1}u+\sum_{\ell=0}^{q(h)}i_\ell m^\ell,\,
		n^{q(h)+1}v+\sum_{\ell=0}^{q(h)}j_\ell n^\ell\right).
		\]
		The proof is complete.
	\end{proof}
	
	\begin{lemma}\label{Hausdorff orbit transfer}
		Let $z=(u,v)\in \Z^2$. There is an integer $C=C(m,n,\La,z)\ge1$ such that
		for all sufficiently large $k$ there are integer vectors $a_k,(b_{k,\ell})_{\ell=1}^{q_k+C}$
		with
		\begin{equation}\label{Hausdorff model inclusions}
			a_k+\mE_{k-C,q_{k-C}}(u)
			\subseteq \fola(z)\cap V_k^{(m)}
			\subseteq\bigcup_{\ell=1}^{q_k+C}
			\bigl(b_{k,\ell}+\mE_{k+C,q_{k+C}}(u)\bigr).
		\end{equation}
		In particular, uniformly for $1\le r\le m^k$,
		\begin{equation}\label{packing estimate1}
			\Phi(u,r,m^k)\yle_{m,n,\La,z}P_r(\fola(z)\cap V_k^{(m)})
			\yle_{m,n,\La,z}(1+k)\Phi(u,r,m^k),
		\end{equation}
		and for every $\alpha>0$,
		\begin{equation}\label{content model transfer}
			\begin{split}
				m^{-\alpha C}\nu_\alpha(\mE_{k-C,q_{k-C}}(u),C(0,m^{k-C}))
				&\le\nu_\alpha(\fola(z),V_k^{(m)})\\
				&\le(q_k+C)m^{\alpha C}
				\nu_\alpha(\mE_{k+C,q_{k+C}}(u),C(0,m^{k+C})).
			\end{split}
		\end{equation}
	\end{lemma}
	
	\begin{proof}
		Choose an integer $C$
		such that $m^C>2\max\{|u|+1,|v|+1\},$
		and take $k>C$.
		
		We first prove the left-hand inclusion in
		\eqref{Hausdorff model inclusions}. 
		If $\eta(u)=1$, then $\mB_u=\{0\}$, and hence
		\[
		(m^{k-C}u,n^{k-C}v)+\mE_{k-C,q_{k-C}}(u)=\Lambda_{k-C}(z)\subseteq \fola(z).
		\]
		In this case, set $a_k=(m^{k-C}u,n^{k-C}v)$.
		
		Suppose instead that $\eta(u)>1$. Then $u\in\{0,-1\}$ and
		$(-(m-1)u,j)\in\Lambda$ for every $j\in\mB_u$.
		Thus every choice of the
		digits defining $\mE_{k-C,q_{k-C}}(u)$ gives an admissible word of
		length $q_{k-C}$ by taking
		\[
		(i_\ell,j_\ell)\in\Lambda\quad(0\le \ell<k-C),
		\qquad
		(i_\ell,j_\ell)=(-(m-1)u,j_\ell)\quad(k-C\le \ell<q_{k-C}).
		\]
		The horizontal contribution of the latter digits satisfies
		\[
		m^{q_{k-C}}u-(m-1)u\sum_{\ell=k-C}^{q_{k-C}-1}m^\ell
		=m^{q_{k-C}}u-u(m^{q_{k-C}}-m^{k-C})=m^{k-C}u.
		\]
		Consequently,
		\[
		(m^{k-C}u,n^{q_{k-C}}v)+\mE_{k-C,q_{k-C}}(u)
		\subseteq\Lambda_{q_{k-C}}(z)\subseteq \fola(z).
		\]
		Set $a_k=(m^{k-C}u,n^{q_{k-C}}v)$ in this case.
		
		In either case, every $(x,y)\in a_k+\mE_{k-C,q_{k-C}}(u)$ satisfies
		\[
		|x|\le (|u|+1)m^{k-C},
		\qquad
		|y|\le(|v|+1)n^{q_{k-C}}\le(|v|+1)m^{k-C}.
		\]
		Our choice of $C$ gives
		\[
		\max\{|x|,|y|\}\le\max\{|u|+1,|v|+1\}m^{k-C}<\frac{m^k}{2}.
		\]
		Therefore $a_k+\mE_{k-C,q_{k-C}}(u)
		\subseteq \fola(z)\cap V_k^{(m)}.$
		
		We next prove the right-hand inclusion. 
		By Lemma \ref{centered window}, every point of
		$\fola(z)\cap V_k^{(m)}$ has a representation of length at most
		$q_k+1\le q_k+C$. It therefore suffices to show that,
		for each $1\le\ell\le q_k+C$, there is an integer vector
		$b_{k,\ell}$ such that
		\[
		\Lambda_\ell(z)\cap V_k^{(m)}
		\subseteq b_{k,\ell}+\mE_{k+C,q_{k+C}}(u).
		\]
		
		Fix a digit $\beta\in\mB_u$. First suppose that $\ell\le k+C$.
		By Lemma \ref{orbit model blocks}(i), there is an integer
		vector $T_{\ell,k+C}$ such that $
		\Lambda_\ell(z)
		\subseteq T_{\ell,k+C}+\mE_{k+C}[\mathbf 0].$
		Define
		\[
		\gamma_{k+C}=\left(0,\beta\sum_{d=k+C}^{q_{k+C}-1}n^d\right).
		\]
		Fixing every tail digit to be $\beta$ gives $
		\gamma_{k+C}+\mE_{k+C}[\mathbf 0]\subseteq\mE_{k+C,q_{k+C}}(u).$
		Hence, with $b_{k,\ell}=T_{\ell,k+C}-\gamma_{k+C}$,
		\[
		\Lambda_\ell(z)
		\subseteq
		T_{\ell,k+C}+\mE_{k+C}[\mathbf 0]
		\subseteq
		b_{k,\ell}+\mE_{k+C,q_{k+C}}(u).
		\]
		
		Suppose that $\ell>k+C$ and
		$\Lambda_\ell(z)\cap V_k^{(m)}\ne\emptyset$.
		Consider any representation
		\[
		x=m^\ell u+\sum_{d=0}^{\ell-1}i_dm^d,
		\qquad
		y=n^\ell v+\sum_{d=0}^{\ell-1}j_dn^d
		\]
		of a point in this intersection. Since $k+C\ge k+1$,
		Lemma \ref{centered window} implies that
		\[
		u\in\{0,-1\},\qquad i_d=-(m-1)u\quad(k+C\le d<\ell).
		\]
		In particular,
		\[
		x=m^\ell u-(m-1)u\sum_{d=k+C}^{\ell-1}m^d
		+\sum_{d=0}^{k+C-1}i_dm^d=
		m^{k+C}u+\sum_{d=0}^{k+C-1}i_dm^d.
		\]
		If $\eta(u)>1$, then $j_d\in\mB_u$ for
		$k+C\le d<\ell$. Set
		\[
		b_{k,\ell}=\left(
		m^{k+C}u,\,
		n^\ell v-\beta\sum_{d=\ell}^{q_{k+C}-1}n^d\right).
		\]
		We then have
		\[
		(x,y)-b_{k,\ell}=\left(
		\sum_{d=0}^{k+C-1}i_dm^d,\,
		\sum_{d=0}^{\ell-1}j_dn^d
		+\beta\sum_{d=\ell}^{q_{k+C}-1}n^d
		\right)
		\in\mE_{k+C,q_{k+C}}(u).
		\]
		Indeed, the mixed symbols below $k+C$ are admissible, and
		every tail digit from level $k+C$ to level $q_{k+C}-1$ belongs to
		$\mB_u$.
		If $\eta(u)=1$, the nonempty intersection under consideration
		forces the endpoint column $\Theta_{-(m-1)u}$ to be nonempty.
		By the definition of $\eta(u)$, it must therefore consist
		of a single digit, say $\Theta_{-(m-1)u}=\{j_*\}$. Thus
		$j_d=j_*$ for $k+C\le d<\ell$. Set
		\[
		b_{k,\ell}=\left(
		m^{k+C}u,\,
		n^\ell v+j_*\sum_{d=k+C}^{\ell-1}n^d
		\right).
		\]
		Since $\mB_u=\{0\}$ in this case,
		\[
		(x,y)-b_{k,\ell}=\left(
		\sum_{d=0}^{k+C-1}i_dm^d,\,
		\sum_{d=0}^{k+C-1}j_dn^d
		\right)
		\in\mE_{k+C,q_{k+C}}(u).
		\]

		All the vectors constructed above depend on $k$ and $\ell$,
		but not on the individual point in
		$\Lambda_\ell(z)\cap V_k^{(m)}$. Taking the union over
		$1\le\ell\le q_k+C$ proves the right-hand inclusion in
		\eqref{Hausdorff model inclusions}.
		
		We now deduce the packing estimates. \eqref{Hausdorff model inclusions} yields
		\[
		P_r\bigl(\mE_{k-C,q_{k-C}}(u)\bigr)
		\le P_r(\fola(z)\cap V_k^{(m)})
		\le
		(q_k+C)P_r\bigl(\mE_{k+C,q_{k+C}}(u)\bigr).
		\]
		The model estimate \eqref{model all scales} and the
		fixed-scale comparison \eqref{Phi comparable scales} yield
		\[
		P_r\bigl(\mE_{k+C,q_{k+C}}(u)\bigr)
		\asy_{m,n,\Lambda,z}\Phi(u,r,m^k)
		\qquad(1\le r\le m^k),
		\]
		and
		\[
		P_r\bigl(\mE_{k-C,q_{k-C}}(u)\bigr)
		\asy_{m,n,\Lambda,z}\Phi(u,r,m^k)
		\qquad(1\le r\le m^{k-C}).
		\]
		Thus the upper bound in \eqref{packing estimate1} holds
		throughout the required range, since $q_k+C=O_{m,n,z}(1+k)$,
		and the lower bound holds when $r\le m^{k-C}$. For the remaining range $m^{k-C}<r\le m^k$, we have
		\[
		0\le k-p(r)\le C,
		\qquad
		0\le q_k-q(r)\le\lceil\tau C\rceil.
		\]
		Using $N_\Lambda\le\#\Lambda$ in the definition of $\Phi$
		therefore gives
		\[
		1\le\Phi(u,r,m^k)
		\le
		(\#\Lambda)^C\eta(u)^{\lceil\tau C\rceil}.
		\]
		Since $\fola(z)\cap V_k^{(m)}$
		is nonempty for all sufficiently large $k$, so $P_r(\fola(z)\cap V_k^{(m)})\ge1$.
		This proves the remaining lower bound and completes
		\eqref{packing estimate1}.
		
		Finally, fix $\alpha>0$. Applying translation invariance,
		monotonicity, and subadditivity of $\mathscr H_\alpha$
		to \eqref{Hausdorff model inclusions}, we obtain
		\[
		\mathscr H_\alpha\bigl(\mE_{k-C,q_{k-C}}(u)\bigr)\le
		\mathscr H_\alpha(\fola(z)\cap V_k^{(m)})\le(q_k+C)
		\mathscr H_\alpha\bigl(\mE_{k+C,q_{k+C}}(u)\bigr).
		\]
		Recall that $d(V_k^{(m)})=m^k,\ d(C(0,m^{k\pm C}))=m^{k\pm C}$
		and that
		$\mE_{k\pm C,q_{k\pm C}}(u)\subseteq C(0,m^{k\pm C})$.
		Consequently,
		\begin{gather*}
			m^{-\alpha k}
			\mathscr H_\alpha\bigl(\mE_{k\pm C,q_{k\pm C}}(u)\bigr)=
			m^{\pm\alpha C}
			\nu_\alpha\bigl(
			\mE_{k\pm C,q_{k\pm C}}(u),
			C(0,m^{k\pm C})
			\bigr),\\
			m^{-\alpha k}\mathscr H_\alpha(\fola(z)\cap V_k^{(m)})=
			\nu_\alpha(\fola(z),V_k^{(m)}).
		\end{gather*}
		Dividing the preceding inequalities by $m^{\alpha k}$
		therefore gives \eqref{content model transfer}.
	\end{proof}

	\section[Counting estimates and dimension formulae]{Counting estimates and dimension formulae}\label{section counts}
	
	Throughout this section, fix $m>n$ and $z=(u,v)\in\Z^2$, and write
	$A=\fola(z)$. Set
	\[
	D_0(u)=\log_m\Big(\frac{\#\La}{\eta(u)}\Big)+\log_n\eta(u),\qquad
	D_1(u)=\log_m N_\La+\log_n\eta(u),
	\]
	and write $\ud(u)=\min\{D_0(u),D_1(u)\}$ and
	$\od(u)=\max\{D_0(u),D_1(u)\}$.
	
	\begin{proposition}\label{mass}
		For every $z=(u,v)\in\Z^2$,
		$\dimLM\fola(z)=\dimUM\fola(z)=D_0(u)$.
	\end{proposition}
	\begin{proof}
		At the last depth the model consists of
		$(\#\La)^k\eta(u)^{q_k-k}$ singletons. The two inclusions in
		\eqref{Hausdorff model inclusions} therefore give
		\[
		(\#\La)^k\eta(u)^{q_k-k}
		\yle\#(\fola(z)\cap V_k^{(m)})
		\yle(1+k)(\#\La)^k\eta(u)^{q_k-k}.
		\]
		Taking logarithms and using $q_k=\tau k+O(1)$ yields
		\[
		\log\#(\fola(z)\cap V_k^{(m)})=kD_0(u)\log m+O_{m,n,\La,z}(\log(k+1)).
		\]
		Therefore
		\[
		\lim_{h\to+\infty}\f{\log\#(\fola(z)\cap[-h,h]^2)}{\log h}=\lim_{k\to\infty}\f{\log\#(\fola(z)\cap V_k^{(m)})}{k\log m}=D_0(u).
		\]
	\end{proof}
	
	\begin{proposition}\label{Beurling estimate}
		Uniformly in $k\ge1$,
		\[
		(\#\La)^k\Nm^{q_k-k}
		\yle_{m,n,\La,z}\sup_{w\in\R^2}\#(\fola(z)\cap Q(w,m^k))
		\yle_{m,n,\La,z}(1+k)(\#\La)^k\Nm^{q_k-k}.
		\]
		Consequently,
		\begin{equation}\label{Beurling}
			\dimbe A=\log_m\Big(\frac{\#\La}{\Nm}\Big)+\log_n\Nm.
		\end{equation}
	\end{proposition}
	\begin{proof}
		Fix a window $Q(w,m^k)$. By
		Lemma \ref{orbit model blocks}(i)--(ii), for each
		$1\le\ell\le q_k$ the set $\La_\ell(z)\cap Q(w,m^k)$ is
		contained in $O_{m,n}(1)$ translates of models whose tail
		alphabets have cardinality at most $\Nm$.
		The same statement holds for the combined set
		$\big(\bigcup_{\ell>q_k}\La_\ell(z)\big)\cap Q(w,m^k)$ by part (iii), independently of the
		number of terminal lengths represented there.
		Equation \eqref{varying model cardinality} bounds the cardinality
		of each containing model by $(\#\La)^k\Nm^{q_k-k}$. Hence
		\[
		\#(\fola(z)\cap Q(w,m^k))
		\yle_{m,n}(1+q_k)(\#\La)^k\Nm^{q_k-k}.
		\]
		This proves the upper bound
		uniformly in $w$. The factor $1+q_k$ counts the $q_k$ finite
		lengths and one combined contribution from all longer words.
		
		For the lower bound, choose $i_*\in\mI$ with $N_{i_*}=\Nm$.
		Lemma \ref{orbit model blocks}(iv) gives
		\[
		T_k+\mE_k\big[(\col_{i_*})_{\ell=k}^{q_k-1}\big]\subseteq \fola(z),\qquad
		T_k=\left(m^{q_k}u+i_*\sum_{\ell=k}^{q_k-1}m^\ell,\ n^{q_k}v\right).
		\]
		The translated model lies in $Q(w_k,m^k)$ for any
		$w_k\in T_k+\mE_k\big[(\col_{i_*})_{\ell=k}^{q_k-1}\big]$. This proves the lower bound.
		Taking logarithms, using $q_k=\tau k+O(1)$ proves \eqref{Beurling}.
	\end{proof}
	
	\begin{lemma}\label{count interpolation}
		Uniformly for $1\le r\le h$, $h\ge2$,
		\begin{equation}\label{uniform count exponents}
			\Big(\f hr\Big)^{\ud(u)}\yle_{m,n,\La}\Phi(u,r,h)
			\yle_{m,n,\La}\Big(\f hr\Big)^{\od(u)}.
		\end{equation}
		The two endpoint exponents occur at
		\begin{equation}\label{endpoint count exponents}
			\Phi(u,1,h)\asy h^{D_0(u)},\qquad
			\Phi(u,n^{p(h)},h)\asy
			\Big(\f h{n^{p(h)}}\Big)^{D_1(u)}.
		\end{equation}
	\end{lemma}
	\begin{proof}
		We first rewrite the definition of $\Phi$ as
		\[
		\Phi(u,r,h)
		=
		N_\Lambda^{p(h)-p(r)}
		\eta(u)^{q(h)-q(r)}
		\left(\frac{\#\Lambda}{N_\Lambda\eta(u)}\right)^{
			\max\{p(h)-q(r),0\}}.
		\]
		Indeed, this follows from the identities
		\[
		\begin{aligned}
			\min\{q(r),p(h)\}-p(r)
			&=p(h)-p(r)-\max\{p(h)-q(r),0\},\\
			q(h)-\max\{p(h),q(r)\}
			&=q(h)-q(r)-\max\{p(h)-q(r),0\}.
		\end{aligned}
		\]
		
		Put
		\[
		L=\log_m(h/r),
		\qquad
		T=(\log_m h-\log_n r)_+.
		\]
		Since $m>n$ and $1\le r\le h$, we have
		$\log_n r\ge\log_m r$ and therefore
		\[
		0\le T\le \log_m h-\log_m r=L.
		\]
		The definitions of $p$ and $q$ give
		\begin{equation}\label{L}
			\bigl|p(h)-p(r)-L\bigr|<1,\qquad
			\bigl|q(h)-q(r)-\tau L\bigr|<1.
		\end{equation}
		Furthermore, the fact that $x\mapsto x_+$ is $1$-Lipschitz and $\bigl|p(h)-q(r)-(\log_m h-\log_n r)\bigr|<1$ give that
		\begin{equation}\label{T}
			\bigl|(p(h)-q(r))_+-T\bigr|<1.
		\end{equation}

		Taking logarithms in the preceding expression for $\Phi$, together with \eqref{L} and \eqref{T}
		now yield
		\[
		\begin{aligned}
			\log\Phi(u,r,h)&=L\log N_\Lambda+\tau L\log\eta(u)
			+T\log\Big(\frac{\#\Lambda}{N_\Lambda\eta(u)}\Big)
			+O_{m,n,\Lambda}(1)\\
			&=\bigl[LD_1(u)+T(D_0(u)-D_1(u))\bigr]\log m
			+O_{m,n,\Lambda}(1)\\
			&=\big[(L-T)D_1(u)+TD_0(u)\big]\log m
			+O_{m,n,\Lambda}(1).
		\end{aligned}
		\]
		The error is uniform in $r$ and $h$.
		Since $L\ge T\ge 0$ ,
		\[
		L\underline D(u)
		\le
		(L-T)D_1(u)+TD_0(u)
		\le
		L\overline D(u).
		\]
		Exponentiating and using $m^L=h/r$ gives
		\[
		\left(\frac hr\right)^{\underline D(u)}
		\yle_{m,n,\Lambda}
		\Phi(u,r,h)
		\yle_{m,n,\Lambda}
		\left(\frac hr\right)^{\overline D(u)}.
		\]
		
		It remains to verify the two endpoint estimates.
		At $r=1$, we have $L=T=\log_m h$, so the logarithmic
		identity above becomes
		\[
		\log\Phi(u,1,h)
		=
		D_0(u)\log h+O_{m,n,\Lambda}(1).
		\]
		Hence $\Phi(u,1,h)\asy_{m,n,\Lambda}h^{D_0(u)}$.
		
		For the second endpoint, set $r_*=n^{p(h)}$.
		Then
		\[
		1\le r_*\le m^{p(h)}\le h,
		\qquad
		q(r_*)=p(h).
		\]
		Substituting into the definition of $\Phi$ gives
		\[
		\Phi(u,r_*,h)
		=
		N_\Lambda^{p(h)-p(r_*)}
		\eta(u)^{q(h)-q(r_*)}.
		\]
		Applying the same floor estimates, we obtain
		\[
		\begin{aligned}
			\log\Phi(u,r_*,h)
			&=
			\log_m(h/r_*)\log N_\Lambda
			+\log_n(h/r_*)\log\eta(u)
			+O_{m,n,\Lambda}(1)\\
			&=
			D_1(u)\log(h/r_*)+O_{m,n,\Lambda}(1).
		\end{aligned}
		\]
		Therefore
		\[
		\Phi(u,n^{p(h)},h)
		\asy_{m,n,\Lambda}
		\left(\frac{h}{n^{p(h)}}\right)^{D_1(u)},
		\]
		as required.
	\end{proof}

	\begin{proposition}\label{packing}
		For every $z=(u,v)\in\Z^2$, $\dimdP\fola(z)=\od(u)$.
	\end{proposition}
	\begin{proof}
		By Proposition \ref{base independence}(ii), we use the window
		definition with $t=m$.
		If $\alpha>\od(u)$ and $\varepsilon\in(0,1)$, the centered-window
		bound \eqref{packing estimate1} and
		\eqref{uniform count exponents} imply
		\[
		\max_{1\le r\le m^{k(1-\varepsilon)}}
		\left(\frac r{m^k}\right)^\alpha P_r(\fola(z)\cap V_k^{(m)})
		\yle(1+k)m^{-\varepsilon k(\alpha-\od(u))}\longrightarrow0.
		\]
		This proves $\dimdP\fola(z)\le\od(u)$.
		
		If $\od(u)=0$, nonnegativity completes the proof. Otherwise fix
		$0\le\alpha<\od$ and choose $0<\varepsilon_*<1-1/\tau$.
		The lower bound in \eqref{packing estimate1} and the endpoint
		counts \eqref{endpoint count exponents} hold in every sufficiently
		large window. If $\od=D_0$, take $r=1$ to obtain
		\[
		\left(\frac1{m^k}\right)^\alpha P_1(\fola(z)\cap V_k^{(m)})
		\yge m^{k(D_0-\alpha)}\longrightarrow\infty.
		\]
		If $\od=D_1$, take $r=n^k$, which is admissible since
		$n^k<m^{k(1-\varepsilon_*)}$. Then
		\[
		\left(\frac{n^k}{m^k}\right)^\alpha P_{n^k}(\fola(z)\cap V_k^{(m)})
		\yge(m^k/n^k)^{D_1-\alpha}
		=m^{k(1-1/\tau)(D_1-\alpha)}\longrightarrow\infty.
		\]
		This proves the reverse inequality.
	\end{proof}
	
	\begin{proposition}\label{assouad estimate}
		For $1\le r\le R$,
		\begin{equation}\label{moving window bound}
			\sup_{w\in\mathbb R^2}N_r\Big(\fola(z)\cap Q(w,R)\Big)
			\yle_{m,n,\Lambda,z}
			\bigl(1+p(R)-p(r)+q(R)-q(r)\bigr)
			N_\Lambda^{p(R)-p(r)}
			N_{\max}^{q(R)-q(r)}.
		\end{equation}
	\end{proposition}
	
	\begin{proof}
		Fix $w\in\mathbb R^2$.  Choose $	k=\max\{1,\lceil\log_m R\rceil\}$ and write
		\[
		M=N_\Lambda^{k-p(r)}N_{\max}^{q_k-q(r)}.
		\]
		Then $r\le R\le m^k\le mR$.
		By \eqref{model maximal fiber bound}, every model appearing
		in Lemma~\ref{orbit model blocks} has covering number at most
		a constant multiple of $M$ at scale $r$.
		Since $Q(w,R)\subseteq Q(w,m^k)$, parts (i)--(iii) of that
		lemma give
		\begin{equation}\label{under q_k and above q_k}
			N_r(\Lambda_\ell(z)\cap Q(w,R))\yle_{m,n}M\
			(1\le\ell\le q_k),\quad N_r\bigg(\Big(\bigcup_{\ell>q_k}\Lambda_\ell(z)\Big)\bigcap Q(w,R)\bigg)\yle_{m,n}M.
		\end{equation}
		We estimate two ranges of the remaining lengths collectively.
		
		First, if $\ell\le p(r)$, every $(x,y)\in\Lambda_\ell(z)$
		satisfies
		\[
		|x|\le(|u|+1)m^\ell\le(|u|+1)r,\qquad|y|\le(|v|+1)n^\ell\le(|v|+1)r.
		\]
		Thus
		\[
		N_r\bigg(\bigcup_{1\le\ell\le p(r)}\Lambda_\ell(z)\bigg)
		\yle_z 1\le M.
		\]
		
		Next suppose that $q(r)>k$, and consider all lengths
		$k<\ell\le q(r)$. Their vertical coordinates satisfy
		\[
		|y|\le(|v|+1)n^\ell\le(|v|+1)r.
		\]
		Within $Q(w,R)$, the integer
		$U=\lfloor x/m^k\rfloor$ takes at most three values.
		For each such $U$, the horizontal expansion gives
		\[
		\left\lfloor\frac{x}{m^{p(r)}}\right\rfloor=m^{k-p(r)}U+\sum_{d=p(r)}^{k-1}i_d m^{d-p(r)}.
		\]
		Since $i_d\in\mathcal I$, the sum takes at most
		$N_\Lambda^{k-p(r)}$ values, independently of $\ell$.
		Each value of the quotient determines a horizontal interval
		of length $m^{p(r)}\le r$. Its product with
		$[-(|v|+1)r,(|v|+1)r]$ can be covered by $O_z(1)$ squares
		of side $r$. Consequently,
		\[
		N_r\bigg(\Big(\bigcup_{k<\ell\le q(r)}\Lambda_\ell(z)\Big)\bigcap Q(w,R)\bigg)\yle_z N_\Lambda^{k-p(r)}\le M.
		\]
		
		The lengths not yet included lie in
		\[
		p(r)<\ell\le k,\qquad\max\{k,q(r)\}<\ell\le q_k.
		\]
		There are at most $k-p(r)+q_k-q(r)$ such lengths, and
		each has covering number $O_{m,n}(M)$ by \eqref{under q_k and above q_k}. Subadditivity therefore yields
		\[
		N_r(\fola(z)\cap Q(w,R))
		\yle_{m,n,z}
		\bigl(1+k-p(r)+q_k-q(r)\bigr)
		N_\Lambda^{k-p(r)}N_{\max}^{q_k-q(r)}.
		\]
		
		Finally, the choice of $k$ gives
		\[
		0\le k-p(R)\le1,
		\qquad
		0\le q_k-q(R)\le\lceil\tau\rceil.
		\]
		Replacing $k$ and $q_k$ in the preceding upper bound by
		$p(R)$ and $q(R)$ therefore changes it by at most a fixed
		multiplicative constant. All estimates are uniform in $w$,
		so taking the supremum proves \eqref{moving window bound}.
	\end{proof}
	
	\begin{corollary}\label{assouad}
		Let $m>n$. For every $z\in\Z^2$,
		$\dima\fola(z)=\log_m N_\La+\log_n\Nm$.
	\end{corollary}
	\begin{proof}
		Put $D_A=\log_m N_\La+\log_n\Nm$. Since
		\[
		p(R)-p(r)\le\log_m\Big(\f Rr\Big)+1,\qquad q(R)-q(r)\le\log_n\Big(\f Rr\Big)+1,
		\]
		Proposition \ref{assouad estimate} gives, for every $\alpha>0$,
		\[
		\sup_w N_r(\fola(z)\cap Q(w,R))
		\yle\Big(1+\log\Big(\f Rr\Big)\Big)\Big(\f Rr\Big)^{D_A}
		\yle_\alpha\Big(\f Rr\Big)^{D_A+\alpha}.
		\]
		Euclidean balls are contained in these squares, so $\dima A\le D_A$.
		
		For the lower bound, use the model
		$\mE_k[(\col_{i_*})_{a=k}^{q_k-1}]$ with
		$N_{i_*}=\Nm$. By Lemma \ref{orbit model blocks}(iv), a
		translate of this model lies in $\fola(z)\cap Q(w_k,m^k)$ for
		some $w_k\in \fola(z)$, and hence in $\fola(z)\cap B(w_k,2m^k)$.
		Since $q(n^k)=k$,
		\eqref{varying model all scales} gives
		\[
		N_{n^k}\bigl(\fola(z)\cap B(w_k,2m^k)\bigr)
		\ge N_{n^k}\big(\mE_k[(\col_{i_*})_{a=k}^{q_k-1}]\big)
		\asy_{m,n}N_\La^{k-p(n^k)}\Nm^{q_k-k}\asy\Big(\f mn\Big)^{kD_A}.
		\]
		Since $m>n$, this establishes the lower bound.
	\end{proof}

	\section[Discrete Hausdorff dimensions and pressure]{Discrete Hausdorff dimensions via entropy and pressure}\label{section hausdorff}
	
	Throughout this section, we assume that $m>n$. We retain the finite model and its depth-$s$ partitions from Section \ref{section model}. Put $c_u=\log\eta(u)$ and define
	\[
	\Delta(u)=\frac{1}{\log m}\min_{0\le\lambda\le1}
	\left\{\log\sum_{i\in\mI}N_i^\lambda+(\tau-\lambda)c_u\right\}.
	\]

	\subsection[Pressure and entropy variational formulae]{Pressure and entropy variational formulae}
	We first express $\Delta(u)$ in terms of topological pressure. Let
	$\N_0=\{0,1,2,\ldots\}$ and equip the full shift
	$\Sigma=\mI^{\N_0}$ with the product topology. Let
	$T\colon\Sigma\to\Sigma$ be the left shift,
	\[
	T(i_0i_1i_2\ldots)=i_1i_2i_3\ldots.
	\]
	Denote by $\mathcal M(\Sigma,T)$ the set of $T$-invariant Borel
	probability measures on $\Sigma$, and by $h_\mu(T)$ the
	measure-theoretic entropy of $\mu\in\mathcal M(\Sigma,T)$.
	
	For a continuous potential $f\colon\Sigma\to\R$, recall the
	variational principle for topological pressure:
	\[
	P(T,f)
	=
	\sup_{\mu\in\mathcal M(\Sigma,T)}
	\left\{
	h_\mu(T)+\int_\Sigma f\,d\mu
	\right\}.
	\]
	
	Let $\mathscr P(\mI)$ denote the set of probability vectors on
	$\mI$. For
	$\bs=(\sigma_i)_{i\in\mI}\in\mathscr P(\mI)$, write
	\[
	h(\bs)=-\sum_{i\in\mI}\sigma_i\log\sigma_i,
	\qquad
	N(\bs)=\sum_{i\in\mI}\sigma_i\log N_i,
	\]
	with the convention $0\log0=0$.
	
	We first relate $h_\mu(T)$ to the entropy of its one-symbol
	distribution. Let
	\[
	\mathcal C=\{[i]:i\in\mI\},
	\qquad
	[i]=\{\bi=i_0i_1\ldots\in\Sigma:i_0=i\}.
	\]
	For $\mu\in\mathcal M(\Sigma,T)$, set
	$\sigma_i=\mu([i])$ and write
	$\bs=(\sigma_i)_{i\in\mI}$. Since $\mathcal C$ is a generating partition,
	\[
	h_\mu(T)=\lim_{n\to\infty}\frac1nH_\mu\left(
	\bigvee_{\ell=0}^{n-1}T^{-\ell}\mathcal C
	\right)\le h(\bs).
	\]
	 Equality holds for the
	Bernoulli measure $\bs^{\otimes\N_0}$.
	
	Define
	\[
	g(\bi)=\log N_{i_0},
	\qquad
	\bi=i_0i_1\ldots\in\Sigma,
	\]
	and, for $0\le\lambda\le1$, set
	\[
	f_{u,\lambda}
	=
	\lambda g+(\tau-\lambda)c_u.
	\]
	Since $f_{u,\lambda}$ depends only on the first symbol, it is
	continuous on $\Sigma$. Moreover,
	\[
	\int_\Sigma f_{u,\lambda}\,d\mu=
	\lambda N(\bs)+(\tau-\lambda)c_u.
	\]
	It follows from the variational principle and the preceding entropy
	estimate that
	\begin{equation}\label{variational principle}
		\begin{aligned}
			P(T,f_{u,\lambda})
			&=\sup_{\mu\in\mathcal M(\Sigma,T)}
			\left\{
			h_\mu(T)+\int_\Sigma f_{u,\lambda}\,d\mu
			\right\}\\
			&=\max_{\bs\in\mathscr P(\mI)}
			\big\{h(\bs)+\lambda N(\bs)\big\}+(\tau-\lambda)c_u.
		\end{aligned}
	\end{equation}
	Indeed, the upper bound follows from $h_\mu(T)\le h(\bs)$,
	while every probability vector $\bs$ is realized by the Bernoulli
	measure $\mu=\bs^{\otimes\N_0}$, for which equality holds.
	
	To evaluate this maximum, set
	\begin{equation}\label{equality}
		Z_\lambda=\sum_{i\in\mI}N_i^\lambda,
		\qquad
		w_i^{(\lambda)}
		=
		\frac{N_i^\lambda}{Z_\lambda},
		\qquad
		\mathbf w^{(\lambda)}
		=
		(w_i^{(\lambda)})_{i\in\mI}.
	\end{equation}
	For every $\bs\in\mathscr P(\mI)$,
	\begin{equation}\label{entropy inequality}
		h(\bs)+\lambda N(\bs)=
		\sum_{i\in\mI}
		\sigma_i\log\frac{N_i^\lambda}{\sigma_i}=
		\log Z_\lambda-
		\sum_{i\in\mI}
		\sigma_i
		\log\frac{\sigma_i}{w_i^{(\lambda)}}
		\le
		\log Z_\lambda,
	\end{equation}
	where the last inequality follows from the non-negativity of
	relative entropy. Equality holds precisely when
	$\bs=\mathbf w^{(\lambda)}$.
	We therefore obtain the explicit pressure formula
	\[
	P(T,f_{u,\lambda})=\log Z_\lambda+(\tau-\lambda)c_u.
	\]
	In particular,
	\begin{equation}\label{minmax}
		\Delta(u)=\frac1{\log m}\min_{0\le\lambda\le1}P(T,f_{u,\lambda}).
	\end{equation}
	Since $\lambda\mapsto P(T,f_{u,\lambda})$ is continuous on
	$[0,1]$, there exists a minimizing parameter
	$\lambda_u\in[0,1]$. Fix any such parameter and set
	\[
	f_u=f_{u,\lambda_u},
	\qquad
	\mu_u=
	\bigl(\mathbf w^{(\lambda_u)}\bigr)^{\otimes\N_0}.
	\]
	Then
	\[
	\Delta(u)=
	\frac{P(T,f_u)}{\log m}
	=
	\frac1{\log m}
	\sup_{\mu\in\mathcal M(\Sigma,T)}
	\left\{
	h_\mu(T)+\int_\Sigma f_u\,d\mu
	\right\}.
	\]
	Thus $\mu_u$ is an equilibrium measure for $f_u$, and the quantity $\Delta(u)$ is represented by the pressure of this single potential.
	
	By \eqref{variational principle} and \eqref{minmax}, we have
	\[
	\Delta(u)=\frac1{\log m}
	\min_{0\le\lambda\le1}
	\Big\{\max_{\bs\in\mathscr P(\mI)}
	\big\{h(\bs)+\lambda N(\bs)\big\}+(\tau-\lambda)c_u\Big\}.
	\]
	The following lemma shows that the order of minimization and
	maximization can be interchanged.

	\begin{lemma}\label{entropy pressure formula}
		We have
		\begin{align*}
			\Delta(u)&=\f1{\log m}\max_{\bs\in\mathscr P(\mI)}\Big\{h(\bs)+\min_{0\le \lambda\le 1}\big\{\lambda N(\bs)+(\tau-\lambda)c_u\big\}\Big\}\\
			&=\frac{1}{\log m}\Big(\max_{\bs\in\mathscr P(\mI)}
			\big\{h(\bs)+\min\{N(\bs),c_u\}\big\}+(\tau-1)c_u\Big).
		\end{align*}
	\end{lemma}
	
	\begin{proof}
		
		For each fixed $\bs\in\mathscr P(\mI)$, the expression
		$h(\bs)+\lambda N(\bs)+(\tau-\lambda)c_u$
		is affine in $\lambda$. Hence
		\begin{align*}
			\min_{0\le \lambda\le 1}
			\big\{\lambda N(\bs)+(\tau-\lambda)c_u
			\big\}&=\min\big\{\tau c_u,N(\bs)+(\tau-1)c_u\big\}\\
			&=(\tau-1)c_u+\min\{N(\bs),c_u\}.
		\end{align*}
		
		Since 
		\[
		\max_{\bs\in\mathscr P(\mI)}\Big\{h(\bs)+\min_{0\le \lambda\le 1}\big\{\lambda N(\bs)+(\tau-\lambda)c_u\big\}\Big\}\le \max_{\bs\in\mathscr P(\mI)}
		\left\{
		h(\bs)+\lambda N(\bs)\right\}+(\tau-\lambda)c_u
		\]
		for every $\lambda\in [0,1]$, taking the minimum over $\lambda$ gives one inequality

		To prove the reverse inequality, let $F(\lambda)=\log Z_\lambda+(\tau-\lambda)c_u$ and  choose $\lambda_*\in[0,1]$ minimizing $F$. 
		We have $F'(\lambda)=N(\mathbf w^{(\lambda)})-c_u$, so
		\[
		\begin{cases}
			N(\mathbf w^{(\lambda_*)})\ge c_u,\quad&\lambda_*=0,\\
			N(\mathbf w^{(\lambda_*)})=c_u,&0<\lambda_*<1,\\
			N(\mathbf w^{(\lambda_*)})\le c_u,&\lambda_*=1.
		\end{cases}
		\]
		In each case,
		$\min\{N(\mathbf w^{(\lambda_*)}),c_u\}
		=\lambda_*N(\mathbf w^{(\lambda_*)})+(1-\lambda_*)c_u$.
		Consequently, \eqref{entropy inequality} yields
		\begin{align*}
			&\max_{\bs\in\mathscr P(\mI)}
			\big\{h(\bs)+\min\{N(\bs),c_u\}\big\}+(\tau-1)c_u\\
			\ge\,&
			h(\mathbf w^{(\lambda_*)})
			+\min\{N(\mathbf w^{(\lambda_*)}),c_u\}+(\tau-1)c_u\\
			=\,&h(\mathbf w^{(\lambda_*)})+\lambda_*N(\mathbf w^{(\lambda_*)})+(\tau-\lambda_*)c_u\\
			=\,&\max_{\bs\in\mathscr P(\mI)}\big\{h(\bs)+\lambda_*N(\bs)\big\}+(\tau-\lambda_*)c_u\\
			=\,& F(\lambda_*)\\
			=\,& 	\min_{0\le\lambda\le1}
			\Big\{\max_{\bs\in\mathscr P(\mI)}
			\big\{h(\bs)+\lambda N(\bs)\big\}+(\tau-\lambda)c_u\Big\}.\qedhere
		\end{align*}
	\end{proof}

	\subsection{Covers and measures on the finite model}
	Let $\mu$ be a probability measure on
	$\Omega_k[\mbB]$ with an arbitrary joint distribution of the
	horizontal digits. Suppose that, conditional on these digits,
	the vertical digits are independent, with $j_\ell$ uniform in
	$\col_{i_\ell}$ for $0\le\ell<k$ and in $\mB_\ell$ for
	$k\le\ell<q_k$. Then, for every $\omega\in\Omega_k[\mbB]$
	and $0\le s\le k$,
	\begin{equation}\label{approximate square mass}
		\begin{aligned}
			\bigl((\Psi_k)_*\mu\bigr)(Q_s(\omega))\prod_{\ell=k-t_s}^{k-1}N_{i_\ell}
			\prod_{\ell=q_k-e_s}^{q_k-1}\#\mB_\ell=\mu\!\left(\left\{\omega'\in\Omega_k[\mbB]:
			i'_\ell=i_\ell\text{ for }k-s\le\ell<k\right\}\right)
			.
		\end{aligned}
	\end{equation}
	Indeed, $t_s\le s$, so the horizontal event on the right fixes
	every column involved in the prescribed mixed vertical range.
	Multiplying the conditional probabilities of the prescribed
	vertical digits gives the identity.
	
	\begin{lemma}\label{Hausdorff stopping lemma}
		Let $0\le\lambda<1$, $\kappa>0$, and $C\ge0$. For each $k$, suppose
		that $x_0,\ldots,x_{k-1}\in\R$ satisfy $|x_\ell-c_u|\le C$, and set
		$T_s=\sum_{\ell=k-s}^{k-1}(x_\ell-c_u)$ for $0\le s\le k$, with
		$T_0=0$. There are $\delta>0$ and $k_0$, depending only on
		$\lambda,\kappa,\tau,C$, such that for $k\ge k_0$,
		\[
		\max_{0\le s\le k}\{\kappa s+\lambda T_s-T_{t_s}\}\ge\delta k.
		\]
	\end{lemma}
	
	\begin{proof}
		The assumption on $T_s$ gives
		\begin{equation}\label{TT}
		|T_s-T_t|\le C|s-t|,\qquad0\le s,t\le k.
		\end{equation}
		Put $
		\vartheta=1-\tau^{-1}>0,\,
		\delta=\kappa\vartheta/2.$
		Choose an integer $L\ge1$ such that $
		C\lambda^L\le\kappa\vartheta/4.$

		For $0\le j\le L$, define $
		s_j=\lfloor k(1-\tau^{-j})\rfloor.$
		Then all $s_j$ belong to $\{0,\ldots,k\}$ and $
		s_j\ge\vartheta k-1$ for $1\le j\le L.$
		The definition of $s_j$ gives
		\[
		k(1-\tau^{-j})-1<s_j\le k(1-\tau^{-j}),\qquad s_{j-1}\le k(1-\tau^{1-j})<s_{j-1}+1.
		\]
		Therefore
		\begin{equation}\label{s_j}
			\begin{aligned}
				s_{j-1}-\tau\le k(1-\tau^{1-j})-\tau&=\tau\big[k(1-\tau^{-j})-1\big]-(\tau-1)k\\
				&<\tau s_j-(\tau-1)k\le k(1-\tau^{1-j})<s_{j-1}+1.
			\end{aligned}
		\end{equation}
		Recall that
		\[
		t_s=\max\{0,k-q(r_s)\}=
		\max\{0,k-\lfloor\tau(k-s)\rfloor\}=
		\max\{0,\lceil\tau s-(\tau-1)k\rceil\}.
		\]
		\eqref{s_j} then yields $s_{j-1}-\tau<t_{s_j}\le s_{j-1}+1.$
		Consequently,
		\[
		|T_{t_{s_j}}-T_{s_{j-1}}|
		\le C|t_{s_j}-s_{j-1}|
		\le C(\tau+1).
		\]
		
		Write $
		W=\sum_{j=1}^{L}\lambda^{j-1}\ge1.$
		By \eqref{TT}, we obtain
		\[
		\sum_{j=1}^{L}\lambda^{j-1}
		\bigl(\lambda T_{s_j}-T_{s_{j-1}}\bigr)=
		\lambda^LT_{s_L}-T_{s_0}=\lambda^LT_{s_L}\ge -Cs_L\lambda^L\ge -Ck\lambda^L.
		\]
		Using this identity and the preceding estimates gives
		\[
		\begin{aligned}
			\sum_{j=1}^{L}\lambda^{j-1}(\kappa s_j+\lambda T_{s_j}-T_{t_{s_j}})
			&=
			\kappa\sum_{j=1}^{L}\lambda^{j-1}s_j
			+\lambda^LT_{s_L}
			+\sum_{j=1}^{L}\lambda^{j-1}
			\bigl(T_{s_{j-1}}-T_{t_{s_j}}\bigr)\\
			&\ge
			\kappa(\vartheta k-1)W
			-C\lambda^Lk-C(\tau+1)W.
		\end{aligned}
		\]
		Since the weights are nonnegative,
		\[
		\begin{aligned}
			\max_{0\le s\le k}\{\kappa s+\lambda T_s-T_{t_s}\}
			&\ge
			\frac1W\sum_{j=1}^{L}\lambda^{j-1}(\kappa s_j+\lambda T_{s_j}-T_{t_{s_j}})\\
			&\ge
			\kappa\vartheta k-\kappa
			-\frac{C\lambda^L}{W}k-C(\tau+1)\\
			&\ge
			\frac{3\kappa\vartheta}{4}k-\kappa-C(\tau+1),
		\end{aligned}
		\]
		where the last inequality follows from $W\ge1$ and the
		choice of $L$.
		
		Finally, choose $
		k_0=
		\Big\lceil
		\frac{4(\kappa+C(\tau+1))}
		{\kappa\vartheta}
		\Big\rceil.
		$
		For every $k\ge k_0$, the preceding bound yields
		\[
		\max_{0\le s\le k}
		\{\kappa s+\lambda T_s-T_{t_s}\}
		\ge
		\frac{\kappa\vartheta}{2}k
		=
		\delta k.
		\]
		This proves the lemma.
	\end{proof}

	\begin{lemma}
		\label{finite upper Hausdorff estimate}
		If $\alpha>\Delta(u)$, there is $C>0$ such that
		\[
		\nu_\alpha(\mE_{k,q_k}(u),C(0,m^k))
		\le ne^{-C k}
		\]
		for all sufficiently large $k$.
	\end{lemma}
	
	\begin{proof}
		By continuity of the pressure, choose $0\le\lambda<1$ such that
		\[
		\kappa=\alpha\log m-\log Z_\lambda-(\tau-\lambda)c_u>0.
		\]
		Define a probability measure $\mu_{k,\lambda}$ on $\Omega_k(u)$ by
		choosing the horizontal digits independently with probabilities
		$w_i^{(\lambda)}=N_i^\lambda/Z_\lambda$, then choosing each mixed
		vertical digit independently and uniformly in its column, conditional
		on the horizontal digits. Choose the tail digits independently and
		uniformly from $\mB_u$.
		
		Fix $\omega\in\Omega_k(u)$ and set $x_\ell=\log N_{i_\ell}$ and
		$T_s=\sum_{\ell=k-s}^{k-1}(x_\ell-c_u)$.
		Applying \eqref{approximate square mass} to
		$\mu_{k,\lambda}$ gives
		\[
		\bigl((\Psi_k)_*\mu_{k,\lambda}\bigr)(Q_s(\omega))
		=Z_\lambda^{-s}
		\prod_{\ell=k-s}^{k-1}N_{i_\ell}^{\lambda}
		\prod_{\ell=k-t_s}^{k-1}N_{i_\ell}^{-1}\eta(u)^{-e_s}.
		\]
		Using  \eqref{levelsforapproximatesquare}, we obtain
		\[
		\begin{aligned}
			\log\frac{\bigl((\Psi_k)_*\mu_{k,\lambda}\bigr)(Q_s(\omega))}
			{m^{-s\alpha}}&=s(\alpha\log m-\log Z_\lambda)
			+\lambda\sum_{\ell=k-s}^{k-1}x_\ell
			-\sum_{\ell=k-t_s}^{k-1}x_\ell-e_sc_u\\
			&=\kappa s+\lambda T_s-T_{t_s}
			+(\tau s-t_s-e_s)c_u\\
			&\ge \kappa s+\lambda T_s-T_{t_s}-\log n.
		\end{aligned}
		\]
		By Lemma \ref{Hausdorff stopping lemma}, there is $C>0$ such that,
		for all sufficiently large $k$, every $\omega$ has a depth $s$ with
		\[
		\bigl((\Psi_k)_*\mu_{k,\lambda}\bigr)(Q_s(\omega))
		\ge \f{e^{C k}}nm^{-s\alpha}.
		\]
		
		For each $\omega\in\Omega_k(u)$, define its first stopping
		time by
		\[
		s(\omega)=\min\left\{
		s\in\{0,\ldots,k\}:
		\bigl((\Psi_k)_*\mu_{k,\lambda}\bigr)(Q_s(\omega))
		\ge\f{e^{C k}}nm^{-s\alpha}
		\right\}.
		\]
		The preceding estimate ensures that this minimum exists.
		Let
		\[
		\mathscr Q_k^*=
		\{Q_{s(\omega)}(\omega):\omega\in\Omega_k(u)\}.
		\]
		Since $\Psi_k(\omega)\in Q_{s(\omega)}(\omega)$ for every
		$\omega$, we have
		\[
		\mE_{k,q_k}(u)=\bigcup_{Q\in\mathscr Q_k^*}Q.
		\]
		We claim that the distinct members of $\mathscr Q_k^*$
		are pairwise disjoint. Suppose that
		$Q_{s(\omega)}(\omega)$ and $Q_{s(\omega')}(\omega')$ intersect. Assume that $s(\omega)\le s(\omega')$.
		By nestedness, $
		Q_{s(\omega')}(\omega')\subseteq Q_{s(\omega)}(\omega'),$
		so $Q_{s(\omega)}(\omega)$ and $Q_{s(\omega)}(\omega')$ also intersect.
		Both belong to the partition $\mathscr Q_{s(\omega)}(k,u)$;
		therefore $
		Q_{s(\omega)}(\omega)=Q_{s(\omega)}(\omega').$
		Since $\omega$ stops at depth $s(\omega)$, this equality implies that $\omega'$ stops at depth $s(\omega)\le s(\omega')$. The definition of its first stopping time implies
		$s(\omega)=s(\omega')$, and the two selected
		squares coincide. This proves the claim and also shows
		that the stopping time $s(Q)$ of each
		$Q\in\mathscr Q_k^*$ is well defined.
		
		Consequently, $\mathscr Q_k^*$ is a partition of
		$\mE_{k,q_k}(u)$, and
		\[
		\sum_{Q\in\mathscr Q_k^*}
		\bigl((\Psi_k)_*\mu_{k,\lambda}\bigr)(Q)=
		\bigl((\Psi_k)_*\mu_{k,\lambda}\bigr)
		\bigl(\mE_{k,q_k}(u)\bigr)=
		\mu_{k,\lambda}(\Omega_k(u))
		=1.
		\]
		Cover each 
		square $Q\in\mathscr Q_k^*$ by the lattice cube of side length $r_{s(Q)}$ supplied by
		Lemma \ref{approximate square geometry}. We obtain
		\begin{align*}
			\nu_\alpha(\mE_{k,q_k}(u),C(0,m^k))
			&\le\sum_{Q\in\mathscr Q_k^*}\left(\frac{r_{s(Q)}}{m^k}\right)^\alpha\\
			&=\sum_{Q\in\mathscr Q_k^*}m^{-s(Q)\alpha}
			\le ne^{-C k}\sum_{Q\in\mathscr Q_k^*}
			\bigl((\Psi_k)_*\mu_{k,\lambda}\bigr)(Q)
			=ne^{-C k}.\qedhere
		\end{align*}
	\end{proof}

	\begin{lemma}
		\label{finite lower Hausdorff estimate}
		If $0<\alpha<\Delta(u)$, there is $c_\alpha>0$ such that
		\[
		\nu_\alpha(\mE_{k,q_k}(u),C(0,m^k))\ge c_\alpha
		\]
		for all sufficiently large $k$.
	\end{lemma}
	
	\begin{proof}
		We construct a probability measure $\mu_k$ on $\Omega_k(u)$
		such that
		\[
		\bigl((\Psi_k)_*\mu_k\bigr)(U)
		\le C_\alpha\left(\frac{d(U)}{m^k}\right)^\alpha
		\qquad(U\in\mathscr C),
		\]
		where $C_\alpha$ is independent of $k$. Applying this estimate
		to a cover of the model will give the required lower bound.
		
		By Lemma \ref{entropy pressure formula} and continuity,
		choose a rational probability vector
		$\bs\in\mathscr P(\mI)$ such that
		\[
		h(\bs)+(\tau-1)c_u+\min\{N(\bs),c_u\}
		>\alpha\log m.
		\]
		For an integer $L$ satisfying $L\sigma_i\in\mathbb Z$
		for every $i\in\mI$, let $\mathcal T_L$ be the set of
		words in $\mI^L$ in which each digit $i$ occurs exactly
		$L\sigma_i$ times. Put
		\[
		h_L=\frac1L\log\#\mathcal T_L.
		\]
		The multinomial formula and Stirling's estimate give
		\[
		\#\mathcal T_L=
		\frac{L!}{\prod_{i\in\mI}(L\sigma_i)!},
		\qquad
		h_L=h(\bs)+O_{\mI}\left(\frac{\log(L+1)}{L}\right).
		\]
		We may therefore fix $L$ sufficiently large 
		such that
		\begin{equation}\label{type block exponent}
			h_L+(\tau-1)c_u+\min\{N(\bs),c_u\}>\alpha\log m.
		\end{equation}

		For $k\ge L$, write $k=bL+a$, where $0\le a<L$,
		and fix $i_*\in\mI$. Define the blocks
		\[
		I_\nu=
		\{k-\nu L,k-\nu L+1,\ldots,k-(\nu-1)L-1\},
		\qquad 1\le\nu\le b,
		\]
		and the subset
		\[
		\Omega_k^*=\left\{\omega\in\Omega_k(u):
		\begin{array}{ll}
			i_\ell=i_*,
			&0\le\ell<a,\\
			(i_\ell)_{\ell\in I_\nu}\in\mathcal T_L,\quad
			&1\le\nu\le b
		\end{array}
		\right\}.
		\]
		There are $(\#\mathcal T_L)^b$ allowed horizontal words.
		In each such word, the digit $i$ occurs
		$bL\sigma_i+a\mathbf1_{\{i_*\}}(i)$ times. Hence the number
		of compatible vertical words is always
		\[
		N_{i_*}^{a}
		\prod_{i\in\mI}N_i^{bL\sigma_i}
		\eta(u)^{q_k-k},
		\]
		independently of the chosen horizontal word. Consequently,
		\[
		\#\Omega_k^*=(\#\mathcal T_L)^b
		N_{i_*}^{a}\prod_{i\in\mI}N_i^{bL\sigma_i}\eta(u)^{q_k-k}.
		\]
		Define $\mu_k$ to be the uniform probability measure on
		$\Omega_k^*$, viewed as a measure on $\Omega_k(u)$:
		\[
		\mu_k(F)
		=
		\frac{\#(F\cap\Omega_k^*)}{\#\Omega_k^*},
		\qquad F\subseteq\Omega_k(u).
		\]
		In particular, $\mu_k(\Omega_k(u))=1$.
		Since every allowed horizontal word has the same number
		of vertical completions, its $\mu_k$-probability is
		$(\#\mathcal T_L)^{-b}$. Conditional on a horizontal word,
		the vertical word is uniformly distributed on
		\[
		\prod_{\ell=0}^{k-1}\Theta_{i_\ell}
		\times\mB_u^{q_k-k}.
		\]
		Thus the vertical digits are conditionally independent and
		uniform in their respective alphabets, and
		\eqref{approximate square mass} applies to $\mu_k$.
		
		Fix $0\le s\le k$ and $Q\in\mQ_s(k,u)$ with positive
		$(\Psi_k)_*\mu_k$-mass. Choose
		$\omega\in\Omega_k^*\cap\Psi_k^{-1}(Q)$, so that
		$Q=Q_s(\omega)$.
		The highest $s$ mixed levels contain
		$\lfloor s/L\rfloor$ complete blocks. Prescribing their
		horizontal digits fixes one word of $\mathcal T_L$ in
		each of these blocks, leaving at most
		$(\#\mathcal T_L)^{b-\lfloor s/L\rfloor}$
		allowed horizontal words. Therefore
		\begin{align*}
			\mu_k\left(
			\left\{\omega'\in\Omega_k(u):
			i'_\ell=i_\ell\text{ for }k-s\le\ell<k\right\}
			\right)&\le
			\frac{(\#\mathcal T_L)^{b-\lfloor s/L\rfloor}}
			{(\#\mathcal T_L)^b}\\
			&=e^{-\lfloor s/L\rfloor Lh_L}
			\le e^{(L-s)h_L},
		\end{align*}
		where $i_\ell$ and $i'_\ell$ denote the digits
		of $\omega$ and $\omega'$, respectively. Since $(i_\ell)_{\ell\in I_\nu}\in \mathcal T_L$,
		every complete block satisfies
		\[
		\sum_{\ell\in I_\nu}\log N_{i_\ell}=
		\sum_{i\in\mI}L\sigma_i\log N_i=LN(\bs).
		\]
		The highest $t_s$ mixed levels contain
		$\lfloor t_s/L\rfloor$ such blocks. Since
		$\log N_i\ge0$, the remaining digits give nonnegative
		contributions, and hence
		\[
		\sum_{\ell=k-t_s}^{k-1}\log N_{i_\ell}\ge
		\Big\lfloor \f{t_s}L\Big\rfloor LN(\bs)\ge
		(t_s-L)N(\bs).
		\]
		Combining these estimates with
		\eqref{approximate square mass} yields
		\begin{align*}
			\bigl((\Psi_k)_*\mu_k\bigr)(Q)&=\prod_{\ell=k-t_s}^{k-1}N_{i_\ell}^{-1}\eta(u)^{-e_s}\mu_k\left(
			\left\{\omega'\in\Omega_k(u):
			i'_\ell=i_\ell\text{ for }k-s\le\ell<k\right\}
			\right)\\
			&\le 	\exp\big(
			L(h_L+N(\bs))-sh_L-t_sN(\bs)-e_sc_u
			\big).
		\end{align*}

		Using $|t_s+e_s-\tau s|<1$, $c_u\ge0$, and
		\[
		t_s(N(\bs)-c_u)
		\ge s\min\{N(\bs)-c_u,0\},
		\]
		we obtain
		\[
		\begin{aligned}
			sh_L+t_sN(\bs)+e_sc_u
			&=
			s(h_L+\tau c_u)+t_s(N(\bs)-c_u)
			+(t_s+e_s-\tau s)c_u\\
			&\ge
			s\bigl(h_L+(\tau-1)c_u+\min\{N(\bs),c_u\}\bigr)-c_u.
		\end{aligned}
		\]
		By \eqref{type block exponent}, this gives
		\[
		\bigl((\Psi_k)_*\mu_k\bigr)(Q)
		\le
		e^{C_L}m^{-\alpha s}=e^{C_L}\left(\frac{r_s}{m^k}\right)^\alpha,
		\]
		where $C_L=L(h_L+N(\bs))+c_u$.
		The constant $C_L$ is independent of $k$, $s$, and $Q$.
		The same bound is immediate for squares of zero mass,
		so it holds for every $Q\in\mQ_s(k,u)$.
		
		Let $U\in\mathscr C$. If $d(U)\ge m^k$, then
		\[
		\bigl((\Psi_k)_*\mu_k\bigr)(U)
		\le1\le\left(\frac{d(U)}{m^k}\right)^\alpha.
		\]
		Otherwise choose $s\in\{0,\ldots,k\}$ such that $d(U)\le r_s<md(U).$
		Lemma \ref{approximate square geometry} shows that $U$
		intersects at most $2(n+2)$ members of $\mQ_s(k,u)$.
		Consequently,
		\[
		\bigl((\Psi_k)_*\mu_k\bigr)(U)\le
		\sum_{\substack{Q\in\mQ_s(k,u)\\Q\cap U\ne\emptyset}}
		\bigl((\Psi_k)_*\mu_k\bigr)(Q)\le
		2(n+2)e^{C_L}\left(\frac{r_s}{m^k}\right)^\alpha\le
		C_\alpha\left(\frac{d(U)}{m^k}\right)^\alpha,
		\]
		where $C_\alpha=2(n+2)m^\alpha e^{C_L}$ is independent
		of $k$. This includes $d(U)=1$, for which $s=k$.
		
		Finally, for any lattice-cube cover $\{U_j\}$ of
		$\mE_{k,q_k}(u)$, the probability measure
		$(\Psi_k)_*\mu_k$ satisfies
		\[
		1
		\le
		\sum_j\bigl((\Psi_k)_*\mu_k\bigr)(U_j)
		\le
		C_\alpha\sum_j
		\left(\frac{d(U_j)}{m^k}\right)^\alpha.
		\]
		Taking the infimum over all such covers proves
		\[
		\nu_\alpha(\mE_{k,q_k}(u),C(0,m^k))
		\ge C_\alpha^{-1}.
		\]
		Thus we may take $c_\alpha=C_\alpha^{-1}$.
	\end{proof}

	\subsection{Transfer to the orbit}
	
	\begin{proposition}\label{discrete Hausdorff dimension formula}
		Let $z=(u,v)\in\Z^2$. Then
		\[
		\dimdH\fola(z)=\dimL\fola(z)=\Delta(u).
		\]
	\end{proposition}
	
	\begin{proof}
		By Proposition \ref{base independence} (i), both dimensions may be computed using $V_k^{(m)}$.
		
		Let $\alpha>\Delta(u)$. Lemmas \ref{finite upper Hausdorff estimate}
		and \ref{Hausdorff orbit transfer} give constants $C>0$
		such that
		\[
		\nu_\alpha(\fola(z),V_k^{(m)})
		\le n(1+k)e^{-C k}
		\]
		for all sufficiently large $k$. These contents are summable and tend
		to zero, proving both upper bounds. If $0<\alpha<\Delta(u)$,
		Lemmas~\ref{finite lower Hausdorff estimate}
		and~\ref{Hausdorff orbit transfer} give
		\[
		\liminf_{k\to\infty}\nu_\alpha(\fola(z),V_k^{(m)})>0.
		\]
		Thus the contents do not tend to zero and their series diverges, proving
		both lower bounds. When $\Delta(u)=0$, the lower bounds follow from
		nonnegativity of the dimensions.
	\end{proof}

	Recall that $D_0(u)=\log_m(\#\La/\eta(u))+\log_n\eta(u)$ and
	$D_1(u)=\log_mN_\La+\log_n\eta(u)$.
	
	\begin{corollary}\label{Hausdorff endpoint estimates}
		For every $u\in\Z$, $\Delta(u)\le\min\{D_0(u),D_1(u)\}$. Moreover,
		if $\eta(u)=1$, then $\Delta(u)=\log_mN_\La$; if $\eta(u)=\Nm$,
		then $\Delta(u)=D_0(u)$.
	\end{corollary}
	
	\begin{proof}
		Evaluation of the pressure at $\lambda=1$ and $\lambda=0$ gives
		$\Delta(u)\le D_0(u)$ and $\Delta(u)\le D_1(u)$, respectively.
		If $\eta(u)=1$, then $c_u=0$ and
		$\lambda\mapsto\sum_{i\in\mI}N_i^\lambda$ is nondecreasing, so its
		minimum is $\#\mI=N_\La$. If $\eta(u)=\Nm$, the pressure derivative
		is
		\[
		\frac{\sum_{i\in\mI}N_i^\lambda\log N_i}
		{\sum_{i\in\mI}N_i^\lambda}
		-\log\eta(u)\le0.
		\]
		The minimum is therefore attained at $\lambda=1$, giving
		$\Delta(u)=D_0(u)$.
	\end{proof}
	
	\begin{remark}
		If the numbers $N_i$, $i\in\mI$, are not all equal, the pressure is
		strictly convex, since its second derivative is the variance of
		$\log N_i$ with respect to the positive weights $w_i^{(\lambda)}$.
		Its minimizer lies in $(0,1)$ precisely when
		\[
		\frac{1}{N_\La}\sum_{i\in\mI}\log N_i
		<\log\eta(u)
		<\frac{1}{\#\La}\sum_{i\in\mI}N_i\log N_i.
		\]
		In this case $\Delta(u)<\min\{D_0(u),D_1(u)\}$.
	\end{remark}

	\section{Equal expansion factors }\label{section equal}
	
	We now treat the case $m=n$, in which the two coordinate
	scales coincide. The finite model has no vertical tail, and
	uniform weights on words suffice for the lower bounds. Together
	with a direct Assouad estimate, these bounds yield the common
	dimension asserted in Theorem~\ref{main theorem m=n}.
	
	\begin{proof}[Proof of Theorem \ref{main theorem m=n}]
		Put $d=\log_m\#\La$. Since $m=n$, the finite model has
		$q_k=k$ and no tail. A depth-$s$ approximate square therefore
		fixes precisely the highest $s$ full symbols. In particular,
		\[
		\Lambda_\ell(z)=m^\ell z+\mE_{\ell,\ell}(u),
		\qquad
		\#\Lambda_\ell(z)=(\#\Lambda)^\ell.
		\]

		We shall prove
		\[
		\dim_{\mathrm A}\fola(z)\le d
		\qquad\text{and}\qquad
		\min\big\{\dimL \fola(z),\dimLM \fola(z)\big\}\ge d.
		\]
		The dimension inequalities in the introduction will then give
		all the asserted equalities.
		
		We first establish the Assouad upper bound. Fix $1\le r<R$
		and $w\in\mathbb R^2$. We use the half-open square grid of
		side length $m^{p(r)}$. Since $m^{p(r)}\le r$, each grid cell
		can be covered by one square of side $r$. For terminal levels $\ell\le p(r)$, every point
		$\omega\in\Lambda_\ell(z)$ satisfies
		\[
		\|\omega\|_\infty
		\le(\|z\|_\infty+1)m^\ell
		\le(\|z\|_\infty+1)r.
		\]
		Thus the union of all these layers can be covered by
		$O_z(1)$ squares of side $r$. For a fixed length $p(r)<\ell\le p(R)$, prescribing the
		symbols at levels $p(r),\ldots,\ell-1$ places all remaining
		points in one grid cell of side $m^{p(r)}$. Consequently, $
		N_r(\Lambda_\ell(z))\le(\#\Lambda)^{\ell-p(r)}.$ We count all lengths $\ell>p(R)$ together. For a point
		$(x,y)\in \fola(z)\cap Q(w,R)$, set
		\[
		U=\left\lfloor\frac{x}{m^{p(R)}}\right\rfloor,
		\qquad
		V=\left\lfloor\frac{y}{m^{p(R)}}\right\rfloor.
		\]
		Since $R/m<m^{p(R)}\le R$, only $O_m(1)$ pairs $(U,V)$
		occur in $Q(w,R)$. If $(x,y)$ has a representation of length $\ell>p(R)$,
		its fine-scale quotients satisfy
		\[
		\begin{aligned}
			\left\lfloor\frac{x}{m^{p(r)}}\right\rfloor
			&=
			m^{p(R)-p(r)}U
			+\sum_{\ell=p(r)}^{p(R)-1}i_\ell m^{\ell-p(r)},\\
			\left\lfloor\frac{y}{m^{p(r)}}\right\rfloor
			&=
			m^{p(R)-p(r)}V
			+\sum_{\ell=p(r)}^{p(R)-1}j_\ell m^{\ell-p(r)}.
		\end{aligned}
		\]
		For each fixed $(U,V)$, the symbols
		$(i_\ell,j_\ell)\in\Lambda$ in these sums have at most
		$(\#\Lambda)^{p(R)-p(r)}$ possible choices. These identities
		hold independently of the terminal length $\ell$.
		Therefore,
		\[
		N_r\Big(
		Q(w,R)\cap\bigcup_{\ell>p(R)}\Lambda_\ell(z)
		\Big)
		\yle_m(\#\Lambda)^{p(R)-p(r)}.
		\]
		Combining the three ranges gives
		\[
		\begin{aligned}
			N_r(\fola(z)\cap Q(w,R))
			\yle_{m,z}\;&
			1+
			\sum_{\ell=p(r)+1}^{p(R)}
			(\#\Lambda)^{\ell-p(r)}
			+(\#\Lambda)^{p(R)-p(r)}.
		\end{aligned}
		\]
		If $\#\Lambda\ge2$, geometric summation yields
		\[
		N_r(\fola(z)\cap Q(w,R))
		\yle_{m,z}
		(\#\Lambda)^{p(R)-p(r)}
		\yle_{m,\Lambda,z}
		\left(\frac Rr\right)^d.
		\]
		If $\#\Lambda=1$, the same estimate before summation gives
		\[
		N_r(\fola(z)\cap Q(w,R))
		\yle_{m,z}
		1+p(R)-p(r)
		\yle_{m,z}
		1+\log\f Rr
		\yle_{m,z,\gamma}
		\left(\frac Rr\right)^\gamma
		\]
		for every $\gamma>0$. Hence
		$\dim_{\mathrm A}\fola(z)\le d$ in both cases.
		
		If $d=0$, nonnegativity and the dimension inequalities
		already prove the proposition. We henceforth assume $d>0$.
		
		Choose an integer $C\ge1$ such that $m^C>2(\|z\|_\infty+1).$
		For every $k>C$ and every $\omega\in\Lambda_{k-C}(z)$,
		\[
		\|\omega\|_\infty
		\le(\|z\|_\infty+1)m^{k-C}
		<\frac{m^k}{2}.
		\]
		Thus $
		\Lambda_{k-C}(z)\subset \fola(z)\cap V_k^{(m)}.$ Consequently,
		\[
		\#\bigl(\mathcal O_\Lambda^+(z)\cap V_k^{(m)}\bigr)
		\ge(\#\Lambda)^{k-C}=m^{(k-C)d}.
		\]
		Since consecutive windows have a fixed ratio of side lengths,
		monotonicity gives
		$\dimLM\mathcal O_\Lambda^+(z)\ge d$.

		Let $\mu_k$ be the uniform probability measure on this
		finite planar set; explicitly,
		\[
		\mu_k(E)
		=
		\frac{\#(E\cap\Lambda_{k-C}(z))}
		{(\#\Lambda)^{k-C}},
		\qquad E\subseteq\mathbb Z^2.
		\]
		Let $U$ be a lattice cube of side $r=d(U)$.
		Suppose first that $1\le r\le m^{k-C}$.
		Since $r/m<m^{p(r)}\le r$, the cube $U$ intersects at most
		$O_m(1)$ grid cells of side $m^{p(r)}$.
		The translation $m^{k-C}z$ is a multiple of $m^{p(r)}$
		in both coordinates. Each grid cell meeting
		$\Lambda_{k-C}(z)$ therefore fixes the highest
		$k-C-p(r)$ full symbols and leaves precisely the lowest
		$p(r)$ symbols free. It consequently contains exactly
		$(\#\Lambda)^{p(r)}$ points of $\Lambda_{k-C}(z)$.
		Hence
		\[
		\begin{aligned}
			\mu_k(U)\yle_m
			(\#\Lambda)^{p(r)-(k-C)}=
			m^{Cd}
			\Big(\frac{m^{p(r)}}{m^k}\Big)^d\le
			m^{Cd}\left(\frac r{m^k}\right)^d
		\end{aligned}
		\]
		up to the same constant depending only on $m$.
		For $r>m^{k-C}$, the corresponding bound follows from
		\[
		\mu_k(U)\le1
		\le m^{Cd}\left(\frac r{m^k}\right)^d.
		\]
		We have therefore obtained a constant
		$K=K(m,\Lambda,z)\ge1$, independent of $k$ and $U$, such that
		\[
		\mu_k(U)\le K\left(\frac{d(U)}{m^k}\right)^d
		\]
		for every lattice cube $U$. Fix $0<\alpha\le d$. If $d(U)\le m^k$, then
		\[
		\left(\frac{d(U)}{m^k}\right)^d
		\le
		\left(\frac{d(U)}{m^k}\right)^\alpha.
		\]
		If $d(U)>m^k$, we instead use
		\[
		\mu_k(U)\le1
		\le\left(\frac{d(U)}{m^k}\right)^\alpha.
		\]
		Thus, in either case,
		\[
		\mu_k(U)
		\le K\left(\frac{d(U)}{m^k}\right)^\alpha.
		\]
		Let $\{U_i\}$ be any lattice-cube cover of
		$\fola(z)\cap V_k^{(m)}$. Since $\mu_k$ is supported on this set,
		\[
		1
		\le\sum_i\mu_k(U_i)
		\le
		K\sum_i\left(\frac{d(U_i)}{m^k}\right)^\alpha.
		\]
		Taking the infimum over all such covers and using
		$d(V_k^{(m)})=m^k$, we obtain
		\[
		\nu_\alpha(\fola(z),V_k^{(m)})\ge K^{-1},
		\qquad k>C.
		\]
		Proposition \ref{base independence}(i) therefore implies
		$\dimL \fola(z)\ge d$. 	This proves the theorem.

	\end{proof}

	\section{The lower entropy index}\label{section lower entropy}
	
	Changing the base may alter the value of this index, as
	Example~\ref{lower entropy radix counterexample} shows. Moreover, the
	annuli in its definition cannot in general be replaced by windows; see
	Example~\ref{lower entropy window counterexample}. We therefore keep
	both the annuli and the base $t$ explicit throughout this section.
	Nevertheless, we prove that the index is independent of the base for
	the class of sets considered here. The conclusion for the index defined
	in the introduction follows by taking $t=2$.
	
	\subsection[Annular localization of the model]{Annular localization of the model}
	
	Unlike the packing calculation, the lower bound here requires a
	model block contained in one annulus. The following lemma chooses
	that block before any small scale is selected, so its estimates hold
	simultaneously over the entire admissible range of $r$.
	
	\begin{lemma}\label{annular model localization}
		Suppose $m>n$, $z=(u,v)\in\Z^2$, and $\#\La\ge2$. For every $t>1$ and every
		$\varepsilon\in(0,1)$, there are infinitely many integers $j$ for which
		\begin{equation}\label{annular lower bound}
			P_r(\fola(z)\cap S_j^{(t)})\yge_{m,n,\La,z,t,\varepsilon}
			\Phi(u,r,t^j)
			\quad\text{simultaneously for all }1\le r\le t^{j(1-\varepsilon)}.
		\end{equation}
	\end{lemma}
	
	\begin{proof}
		Fix $t>1$ and $\varepsilon\in(0,1)$. The construction in
		Lemma \ref{Hausdorff orbit transfer} gives
		\[
		T_k+\mE_{k,q_k}(u)\subset A,
		\qquad
		T_k=
		\begin{cases}
			(m^ku,n^{q_k}v),&\eta(u)>1,\\
			(m^ku,n^kv),&\eta(u)=1.
		\end{cases}
		\]
		Put $j_k=\lfloor\log_t(m^k)\rfloor$. Since $
		1\le m^kt^{-j_k}<t,1/n<n^{q_k}t^{-j_k}<t,$
		we may pass to an infinite subsequence along which
		\[
		\frac{m^k}{t^{j_k}}\longrightarrow\lambda_m\in[1,t],
		\qquad
		\frac{n^{q_k}}{t^{j_k}}\longrightarrow\lambda_n\in\Big[\f1n,t\Big].
		\]
		All limits below are taken along this subsequence. 
		
		We first show that there are fixed integers $S\ge0$ and
		$a\in\mathbb Z$ such that, for infinitely many $k$, one can choose
		$Q_k\in\mQ_S(k,u)$ satisfying
		\[
		T_k+Q_k\subset S_{j_k+a}^{(t)}.
		\]
		The depth $S$ and the sets $Q_k$ will be chosen independently of $r$.
		
		Suppose first that $N_\Lambda\ge2$ or $\eta(u)>1$. Consider
		\[
		K_x=
		\left\{\sum_{d=1}^{\infty}\iota_dm^{-d}:
		\iota_d\in\mI\right\},
		\qquad
		K_y=
		\left\{\sum_{d=1}^{\infty}\zeta_dn^{-d}:
		\zeta_d\in\mB_u\right\}.
		\]
		Each of these sets is uncountable whenever its defining alphabet
		contains at least two digits.
		If $N_\Lambda\ge2$, choose $X\in K_x$ such that
		\[
		\lambda_m(u+X)\notin
		\{0\}\cup\Big\{\pm \f {t^i}2:i\in\mathbb Z\Big\}.
		\]
		If $\eta(u)>1$, choose $Y\in K_y$ such that
		\[
		\lambda_n(v+Y)\notin
		\{0\}\cup\Big\{\pm \f {t^i}2:i\in\mathbb Z\Big\}.
		\]
		These choices are possible since the excluded sets are countable.
		When $\eta(u)=1$, put $Y=0$. The only case in which $X$ has not yet been specified is
		$N_\Lambda=1$ and $\eta(u)>1$. Writing $\mI=\{i_0\}$, the definition
		of $\eta(u)$ gives $i_0=-(m-1)u.$
		Thus $K_x$ consists of the single point
		\[
		X=\frac{i_0}{m-1}=-u.
		\]
		Consequently, in all cases currently under consideration, the vector
		\[
		W=
		\begin{cases}
			\bigl(\lambda_m(u+X),\lambda_n(v+Y)\bigr),
			&\eta(u)>1,\\
			\bigl(\lambda_m(u+X),0\bigr),
			&\eta(u)=1
		\end{cases}
		\]
		satisfies
		\[
		c:=\|W\|_\infty>0,
		\qquad
		c\notin\Big\{\pm \f {t^i}2:i\in\mathbb Z\Big\}.
		\]
		Choose $a\in\mathbb Z$ such that $t^{a-1}<2c<t^a$, and set
		\[
		\delta=\frac12\min\left\{c-\frac{t^{a-1}}2,\,\frac{t^a}2-c\right\}>0.
		\]
		Fix an integer $S\ge1$ large enough that $ntm^{-S}<\delta.$ 
		For this fixed $S$ and all sufficiently large $k>S$,
		\[
		q(m^{k-S})\ge k,
		\qquad
		t_S=0,
		\qquad
		e_S=q_k-q(m^{k-S})>\tau S-1.
		\]
		Write the above choices of $X$ and $Y$ in the form
		\[
		X=\sum_{d=1}^{\infty}\iota_dm^{-d},
		\qquad
		Y=\sum_{d=1}^{\infty}\zeta_dn^{-d}.
		\]
		By our choice above, when $\eta(u)=1$, we have $\zeta_d=0$ for every $d\ge1$.
		Define
		\[
		Q_k=\Psi_k\left(\left\{
		\omega\in\Omega_k(u):
		\begin{array}{ll}
			i_{k-d}=\iota_d,&1\le d\le S,\\
			j_{q_k-d}=\zeta_d,&1\le d\le e_S
		\end{array}
		\right\}\right)\in\mQ_S(k,u).
		\]
		For every $(x,y)\in Q_k$, the prescribed digits give
		\begin{gather*}
			\left|\frac{x}{m^k}-X\right|=\left|\sum_{d=0}^{k-S-1}i_dm^{d-k}-\sum_{d>S}\iota_dm^{-d}\right|\le m^{-S},\\
			\left|\frac{y}{n^{q_k}}-Y\right|=\left|\sum_{d=0}^{q_k-e_S-1}j_dn^{d-q_k}-\sum_{d>e_S}\zeta_dn^{-d}\right|\le n^{-e_S}\le nm^{-S}.
		\end{gather*}
		When $\eta(u)=1$, all tail digits are zero, so we also have
		$0\le y<n^k$. Note that for every $(x,y)\in Q_k$, we have
		\begin{align*}
			&\frac{T_k+(x,y)}{t^{j_k}}-W\\
			&=\begin{cases}
				\big((m^kt^{-j_k}-\lam_m)u+xt^{-j_k}-\lam_m X,\,(n^{q_k}t^{-j_k}-\lam_n)v+yt^{-j_k}-\lam_n Y\big),\quad  &\eta(u)>1,\\
				\big((m^kt^{-j_k}-\lam_m)u+xt^{-j_k}-\lam_m X,\,(n^kv+y)t^{-j_k}\big), & \eta(u)=1.
			\end{cases}
		\end{align*}
		Using these estimates and the convergence of the normalization
		factors, we obtain
		\[
		\limsup_{k\to\infty}
		\sup_{w\in T_k+Q_k}
		\left\|\frac{w}{t^{j_k}}-W\right\|_\infty
		\le ntm^{-S}<\delta.
		\]
		Here, when $\eta(u)=1$, the vertical coordinate tends uniformly
		to zero since
		\[
		\sup_{(x,y)\in Q_k}
		\frac{|n^kv+y|}{t^{j_k}}
		\le (|v|+1)\frac{n^k}{t^{j_k}}\le (|v|+1)t\frac{n^k}{m^k}
		\longrightarrow0.
		\]
		Therefore, for all sufficiently large $k$ and every
		$w\in T_k+Q_k$,
		\[
		\frac{t^{a-1}}2
		<
		\left\|\frac{w}{t^{j_k}}\right\|_\infty
		<
		\frac{t^a}2.
		\]
		Both inequalities are strict, and hence $T_k+Q_k\subset S_{j_k+a}^{(t)}.$
		
		It remains to consider $N_\Lambda=\eta(u)=1$.
		Write $\mI=\{i_0\}$. Every point of
		$T_k+\mE_{k,q_k}(u)$ has the same horizontal coordinate
		\[
		x_k
		=
		m^ku+i_0\frac{m^k-1}{m-1}
		=
		m^k\left(u+\frac{i_0}{m-1}\right)
		-\frac{i_0}{m-1},
		\]
		and its vertical coordinate satisfies $|y|\le(|v|+1)n^k.$
		We claim that
		\[
		u+\frac{i_0}{m-1}\ne0.
		\]
		Indeed, equality would imply
		$(u,i_0)=(0,0)$ or $(-1,m-1)$. The unique column would then be the
		corresponding endpoint column, giving
		$\eta(u)=\#\Lambda\ge2$, a contradiction. It follows that
		\[
		\frac{|x_k|}{t^{j_k}}
		\longrightarrow
		\lambda_m\left|u+\frac{i_0}{m-1}\right|>0,
		\qquad
		\frac{(|v|+1)n^k}{|x_k|}\longrightarrow0.
		\]
		For large $k$, let $b_k\ge2$ be the unique index such that
		$(x_k,0)\in S_{b_k}^{(t)}$. Then $t^{b_k-1}\le 2|x_k|\le t^{b_k}.$
		Thus, for all sufficiently large $k$, every vertical coordinate
		under consideration satisfies
		\[
		|y|<\frac{|x_k|}{t}\le\frac{t^{b_k-1}}2.
		\]
		Consequently, $T_k+\mE_{k,q_k}(u)\subset S_{b_k}^{(t)}.$
		The convergence of $|x_k|t^{-j_k}$ and the preceding bounds on
		$|x_k|$ show that $b_k-j_k$ is bounded. Passing to a further
		infinite subsequence, we may assume that $b_k-j_k=a$ for some
		fixed integer $a$. Taking $S=0$ and
		$Q_k=\mE_{k,q_k}(u)$ gives the required inclusion in this case
		as well.
		
		We now apply Lemma~\ref{model descendants}. Since
		$t^{j_k}\le m^k$,
		\[
		\frac{t^{(j_k+a)(1-\varepsilon)}}{m^{k-S}}
		\le
		t^{a(1-\varepsilon)}m^{S-\varepsilon k}
		\longrightarrow0.
		\]
		Hence, for all sufficiently large $k$ in the chosen subsequence,
		every $1\le r\le t^{(j_k+a)(1-\varepsilon)}$ also satisfies
		$r\le m^{k-S}$. For all these scales, translation invariance
		of packing numbers gives
		\[
		P_r(\fola(z)\cap S_{j_k+a}^{(t)})\ge P_r(Q_k)\ge c_S\Phi(u,r,m^k).
		\]
		Finally, $t^{a-1}<t^{j_k+a}/m^k\le t^a,$
		so
		\[
		|p(t^{j_k+a})-k|+|q(t^{j_k+a})-q_k|=O_{m,n,t,a}(1).
		\]
		All exponents in the definition of $\Phi$ therefore change
		by a bounded amount, uniformly in $r$. It follows that
		\[
		\Phi(u,r,m^k)
		\asy_{m,n,\Lambda,t,a}
		\Phi(u,r,t^{j_k+a}).
		\]
		The resulting lower bound holds simultaneously for every
		$1\le r\le t^{(j_k+a)(1-\varepsilon)}$, since $Q_k$ was chosen
		independently of $r$. Since $j_k\to\infty$, the indices
		$j=j_k+a$ include infinitely many distinct integers.
		This proves the lemma.
	\end{proof}
	
	\subsection[The lower entropy formula]{The lower entropy formula}
	
	\begin{proposition}\label{lower entropy formula}
		For every $t>1$ and every $z=(u,v)\in\Z^2$,
		\[
		\delta_t(\fola(z))=
		\begin{cases}
			\displaystyle\min\left\{\log_m\frac{\#\La}{\eta(u)},
			\log_mN_\La\right\}+\log_n\eta(u),\quad&m>n,\\
			\log_m\#\La,&m=n.
		\end{cases}
		\]
	\end{proposition}
	
	\begin{proof}
		First suppose $m>n$. If
		$k=\lceil\log_m(t^j)\rceil$, then
		$V_j^{(t)}\subseteq V_k^{(m)}$ and $t^j\le m^k<m t^j$.
		By \eqref{packing estimate1} and \eqref{Phi comparable scales},
		\begin{equation}\label{general base window count}
			P_r(\fola(z)\cap S_j^{(t)})\yle_{m,n,\La,z,t}
			(1+j)\Phi(u,r,t^j),\qquad1\le r\le t^j.
		\end{equation}
		Fix $\alpha>\ud(u)$ and choose once and for all
		$0<\varepsilon_*<1-1/\tau$. If $\ud=D_0$, the admissible choice
		$r=1$ and \eqref{endpoint count exponents} give
		\[
		\inf_{1\le r\le t^{j(1-\varepsilon_*)}}
		\left(\frac r{t^j}\right)^\alpha P_r(\fola(z)\cap S_j^{(t)})
		\yle(1+j)t^{-j(\alpha-D_0)}\longrightarrow0.
		\]
		If $\ud=D_1$, take $r_j=n^{p(t^j)}$. Since
		$r_j\asy_{m,n}t^{j/\tau}$, it is admissible for all sufficiently
		large $j$. The second endpoint estimate gives
		\[
		\inf_{1\le r\le t^{j(1-\varepsilon_*)}}
		\left(\frac r{t^j}\right)^\alpha P_r(\fola(z)\cap S_j^{(t)})\yle(1+j)\Big(\f {t^j}{r_j}\Big)^{D_1-\alpha}\yle_{\alpha}(1+j)t^{-j(1-1/\tau)(\alpha-D_1)}
		\longrightarrow0.
		\]
		Thus the convergence required by the definition holds for one
		fixed $\varepsilon_*$, proving $\delta_t(A)\le \ud$.
		
		If $\ud=0$, nonnegativity gives the reverse inequality. Otherwise
		$\#\La\ge2$. Fix $0\le\alpha<\ud$ and now let
		$\varepsilon\in(0,1)$ be arbitrary. Apply
		Lemma \ref{annular model localization} with this $t,\varepsilon$.
		On its infinite sequence of annuli, every admissible scale satisfies
		\[
		\left(\frac r{t^j}\right)^\alpha P_r(\fola(z)\cap S_j^{(t)})
		\yge\left(\frac r{t^j}\right)^\alpha\Phi(u,r,t^j)
		\yge\Big(\f {t^j}r\Big)^{\ud-\alpha}
		\ge t^{\varepsilon j(\ud-\alpha)}\longrightarrow\infty.
		\]
		The second inequality is \eqref{uniform count exponents}, and all
		constants are uniform in $r,j$ on this sequence.
		Taking the infimum over $r$ preserves the lower bound. Since this
		argument applies separately to every $\varepsilon$, no
		$\varepsilon$ can give the convergence in the definition. Hence
		$\delta_t(A)\ge \ud$.
		
		Finally, suppose $m=n$. Apply Proposition~\ref{base independence}(iii)
		to $A=\fola(z)$. By Theorem~\ref{main theorem m=n}, the two endpoints
		of~\eqref{general lower index comparison} both equal
		$\log_m\#\La$, so $\delta_t(\fola(z))=\log_m\#\La$. This gives the remaining
		case of the formula. This also completes the proof of
		Theorem \ref{main theorem m>n}.
	\end{proof}
	
	\section[Invariant lattice sets]{Invariant lattice sets of $\{S_{(i,j)}\}_{(i,j)\in\Lambda}$}\label{section invariant}
	
	We now apply the orbit dimension formulae to invariant subsets
	of $\Z^2$. The reduction to finitely many orbits uses the decomposition
	theorem of~\cite{MX26}. We recall its setting and statement before
	specializing to the maps in~\eqref{map}.
	
	Let $(X,d)$ be an unbounded, locally compact, complete metric space. A finite family $\mathcal F=\{f_i\}_{i=1}^p$ of distinct self-maps of $X$ is called a \emph{reverse iterated function system} (RIFS) if there exists $r>1$ such that
	\begin{equation}\label{def_r}
		d\bigl(f_i(x),f_i(y)\bigr)
		\ge r\,d(x,y)
		\qquad
		(x,y\in X,\ 1\le i\le p).
	\end{equation}
	A nonempty set $K\subseteq X$ is said to be \emph{invariant} under $\mathcal F$ or an \emph{invariant set} of $\F$ if
	\begin{equation}\label{eq: FUTF}
		K=\bigcup_{i=1}^p f_i(K).
	\end{equation}
	An invariant set $K$ is called \emph{non-overlapping} if the union in \eqref{eq: FUTF} is disjoint, or equivalently, if $f_i(K)\cap f_j(K)=\varnothing$ whenever $i\ne j$.
	
	The forward orbit of $a\in X$ is defined by
	\[
	\fo_{\F}(a)=\{g(a):g\in G^*(\F)\},
	\]
	where $G^*(\F)$ denotes the semigroup of distinct maps generated by $\F$. 	Let $P(\F)$ be the set of fixed points of maps in $G^*(\F)$; namely,
	\[
	P(\F)=\bigl\{x\in X:g(x)=x\text{ for some }g\in G^*(\F)\bigr\}.
	\]

	\begin{theorem}\cite{MX26}\label{discreteinvariant}
		Let $\F=\{f_i\}_{i=1}^m$ be an RIFS on $(X,d)$. If $K$ is a locally finite (i.e., its intersection with every bounded set is finite) invariant set of $\F$, then $K$ is a finite union of forward orbits. More precisely, we have
		\[
		K=\bigcup_{a\in K\cap P(\F)}\fo_\F(a),
		\]
		where $K\cap P(\F)$ is finite.
	\end{theorem} 
	
	Since every subset of $\mathbb Z^2$ is locally finite, Theorem \ref{discreteinvariant} implies that every invariant set of $\{S_{(i,j)}\}_{(i,j)\in\Lambda}$ is a finite union of forward orbits based at points of $P(\Lambda)$, where
	\[
	P(\Lambda)=\{z\in \Z^2:S_{(\bi,\bj)}(z)=z\text{\ for some\ } (\bi,\bj)\in \La^k \text{\ and\ } k\in \N^+\}.
	\]
	For the affine maps considered here, the set $P(\La)$ admits
	a particularly simple description.
	\begin{lemma}\label{fixedpoint}
		We have 
		\begin{align*}
			P(\Lambda)&=\big\{(u,v)\in\{0,-1\}^2:
			(-(m-1)u,-(n-1)v)\in\La\big\}\\
			&=\big\{z\in\Z^2:S_{(i,j)}(z)=z\text{\ for some\ } (i,j)\in \La\}.
		\end{align*}
	\end{lemma}
	\begin{proof}
		Let $(u,v)\in P(\La)$. Then there exist $k\ge1$ and
		$(\bi,\bj)=(i_0\ldots i_{k-1},j_0\ldots j_{k-1})\in\La^k$
		such that $S_{(\bi,\bj)}(u,v)=(u,v)$. Hence
		\begin{equation}\label{iterate}
			(u,v) =\Big(m^ku+\sum_{\ell=0}^{k-1}i_{\ell}m^\ell,n^kv+\sum_{\ell=0}^{k-1}j_{\ell}n^\ell\Big).
		\end{equation}
		Consequently,	
		\[
		u=\dfrac{\sum_{\ell=0}^{k-1}i_{\ell}m^\ell}{1-m^k}\in \Z,\qquad v=\dfrac{\sum_{\ell=0}^{k-1}j_{\ell}n^\ell}{1-n^k}\in \Z.
		\]
		Since $\La\subseteq\{0,\ldots,m-1\}\times \{0,\ldots,n-1\}$, we obtain $(u,v)\in\{0,-1\}^2$. Moreover, $u=0$ forces $i_\ell=0$ for every $\ell$,
		whereas $u=-1$ forces $i_\ell=m-1$ for every $\ell$.
		Similarly, $v=0$ forces $j_\ell=0$ for every $\ell$,
		whereas $v=-1$ forces $j_\ell=n-1$ for every $\ell$. Thus
		\[
		(i_\ell,j_\ell)=(-(m-1)u,-(n-1)v)\in\La
		\]
		for every $0\le\ell<k$.
		
		Conversely, if $(u,v)\in\{0,-1\}^2$ and
		$(i,j)=(-(m-1)u,-(n-1)v)\in\La$, then
		$S_{(i,j)}(u,v)=(u,v)$, so $(u,v)\in P(\La)$.
		This completes the proof.
	\end{proof}
	
	The following lemma shows that, whenever $P(\La)$ is non-empty,
	the forward orbit of each point in $P(\La)$ is an affine image of a discrete Bedford--McMullen carpet.
	\begin{lemma}\label{allsame}
		Let $(u,v)\in\{0,-1\}^2$. Define $T_{u,v}:\Z^2\longrightarrow \Z^2$ by
		\[
		T_{u,v}(z)=\begin{pmatrix}
			2u+1&0\\
			0 &2v+1
		\end{pmatrix}z-\begin{pmatrix}
			(m-1)u\\
			(n-1)v
		\end{pmatrix},\qquad z\in\Z^2.
		\] 
		Then
		\[
		\fola (u,v)=T_{u,v}\big(\fo_{T_{u,v}(\La)}(0,0)\big)+\begin{pmatrix}
			mu\\
			nv
		\end{pmatrix}.
		\]
	\end{lemma}
	\begin{proof}
		Put
		\[
		A=
		\begin{pmatrix}
			2u+1&0\\
			0&2v+1
		\end{pmatrix},
		\qquad
		D=
		\begin{pmatrix}
			m&0\\
			0&n
		\end{pmatrix},
		\qquad
		z_0=(u,v).
		\]
		Since $u,v\in\{0,-1\}$, we have
		$A^2=I$ and $Az_0=-z_0$.
		
		Define $G:\Z^2\longrightarrow\Z^2$ by
		\[
		G(z)=Az+z_0=T_{u,v}(z)+Dz_0.
		\]
		For every $(i,j)\in\La$, we have
		\[
		\begin{aligned}
			G\bigl(S_{T_{u,v}(i,j)}(z)\bigr)
			&=A\bigl(Dz+T_{u,v}(i,j)\bigr)+z_0\\
			&=ADz+(i,j)+(D-I)z_0+z_0\\
			&=D(Az+z_0)+(i,j)\\
			&=S_{(i,j)}(G(z)).
		\end{aligned}
		\]
		Thus $
		S_{(i,j)}\circ G=G\circ S_{T_{u,v}(i,j)}.$
		Iterating this identity and using $G(0,0)=z_0$, we obtain
		\[
		\fola(z_0)=G\bigl(\fo_{T_{u,v}(\La)}(0,0)\bigr).
		\]
		Substituting the definition of $G$ gives the desired identity.
	\end{proof}
	\begin{remark}
		In particular, if $(u,v)\in P(\La)$, then by Lemma \ref{fixedpoint},
		\[
		(0,0)=T_{u,v}(-(m-1)u,-(n-1)v)\in T_{u,v}(\La).
		\]
		Thus the forward orbit of every point in $P(\La)$ is an affine image
		of a discrete Bedford--McMullen carpet whose digit set contains
		the origin.
	\end{remark}
	
	\begin{proposition}\label{invariant dimensions}
		Let $K\subseteq\Z^2$ be an invariant set of
		$\{S_{(i,j)}\}_{(i,j)\in\La}$. Then $K$ is non-overlapping and
		\[
		K=\bigcup_{z\in P'}\fola(z),
		\qquad
		\emptyset\ne P'\subseteq P(\La).
		\]
		Moreover, we have
		\[
		\textup{Dim}K=\max_{z\in P'}\,\textup{Dim}\fola(z)
		\]
		for $\textup{Dim}\in\{\dimL,\dimdH,\de(\cdot),\dimM,\dimdP,\dimbe,\dima\} $.
	\end{proposition}
	
	\begin{proof}
		We first prove the non-overlapping property. Suppose that
		$(i,j),(i',j')\in\La$ and
		\[
		S_{(i,j)}(K)\cap S_{(i',j')}(K)\ne\emptyset.
		\]
		Then there exist $(x,y),(x',y')\in K$ such that
		\[
		(mx+i,ny+j)=(mx'+i',ny'+j').
		\]
		Hence
		\[
		i-i'\in m\Z,\qquad j-j'\in n\Z.
		\]
		Since $0\le i,i'\le m-1, 0\le j,j'\le n-1,$
		it follows that $i=i'$ and $j=j'$. Thus $K$ is non-overlapping.
		
		By Theorem \ref{discreteinvariant}, there exists a
		nonempty finite set $P'\subseteq P(\La)$ such that
		\[
		K=\bigcup_{z\in P'}\fola(z).
		\]
		By Remark \ref{finite stability}, the lower discrete Hausdorff,
		discrete Hausdorff, discrete packing, Beurling, and Assouad dimensions
		are finitely stable. Moreover, every forward orbit has a mass
		dimension by Theorem \ref{main theorem m=n} and Theorem \ref{main theorem m>n}, so the same remark
		also applies to the mass dimension. Consequently,
		\[
		\textup{Dim}K=\max_{z\in P'}\textup{Dim}\fola(z)
		\]
		for $\textup{Dim}\in
		\{\dimL,\dimdH,\dimM,\dimdP,\dimbe,\dima\}.$
		
		It remains to consider the lower entropy index, which is not finitely
		stable in general. Suppose first that $m>n$, and set
		\[
		\eta_{P'}=\max\{\eta(u):(u,v)\in P'\}.
		\]
		Choose $z_*=(u_*,v_*)\in P'$ such that
		$\eta(u_*)=\eta_{P'}$. Since $\fola(z_*)\subseteq K$, monotonicity gives $\de(K)\ge\de(\fola(z_*)).$
		On the other hand, by subadditivity of packing numbers and the uniform
		packing estimate for forward orbits,
		\[
		P_r(K\cap S_k^{(2)})\le\sum_{(u,v)\in P'}P_r\bigl(\fola(u,v)\cap S_k^{(2)}\bigr)\yle_{m,n,\La,P'}
		(1+k)\Phi(u_*,r,2^k).
		\]
		Here we have used that
		$\Phi(u,r,h)$ is nondecreasing in $\eta(u)$. Using the same choices of the intermediate scale $r$ as in the proof
		of the lower entropy index formula for a single forward orbit, we
		obtain
		\[
		\de(K)\le\min\left\{
		\log_m\frac{\#\La}{\eta_{P'}},
		\log_m N_\La\right\}+\log_n\eta_{P'}
		=\de(\fola(z_*)).
		\]
		Hence
		\[
		\de(K)=\de(\fola(z_*))
		=\max_{z\in P'}\de(\fola(z)),
		\]
		where the last equality follows from the fact that the orbit formula
		for $\de(\fola(u,v))$ is nondecreasing in $\eta(u)$.
		
		If $m=n$, all forward orbits have lower entropy index
		$\log_m\#\La$, and the same finite-union packing estimate gives
		\[
		\de(K)=\log_m\#\La
		=\max_{z\in P'}\de(\fola(z)).
		\]
		This completes the proof.
	\end{proof}

	\section{Semigroup growth, orbit dimensions, and the dual attractor}
	\label{semigroup growth section}
	
	We turn to the comparison outlined in the introduction between
	orbit dimensions, semigroup growth, and the dual attractor. For this
	purpose, we recall the precise one-dimensional result of~\cite{MX26}.
	Let $\F=\{f_i\}_{i=1}^p$ be an RIFS on $\R^d$. When the maps $f_i$ are surjective, their inverses form a contractive IFS $\F^{-1}=\{f_i^{-1}\}_{i=1}^p$, called the \emph{dual system}, whose attractor is denoted by $K(\F^{-1})$.
	
	\begin{theorem}\cite[Theorem 1.1]{MX26}\label{thm:intro-main}
		Let $\F=\{f_i(x)=r_i x+b_i\}_{i=1}^p$ with $b_i\in\R,|r_i|>1$ and define $G_{\F}(R):=\#\{g\in G^*(\F):|g'|\le R\}.$ Then the limit
		\begin{equation*}
			d_{\F}\coloneq\lim_{R\to\infty}\frac{\log G_{\F}(R)}{\log R}
		\end{equation*}
		exists and satisfies
		\begin{equation}\label{df exists}
			\dimB K(\F^{-1})\le d_{\F}\le s,
		\end{equation}
		where $s$ is the similarity dimension of $K(\F^{-1})$, determined by
		\[
		\sum_{i=1}^p |r_i|^{-s}=1.
		\]
		
		Furthermore,
		for every $a\in\R$,
		\[
		\dimM\fo_{\F}(a)=\begin{cases}
			d_{\F},&\text{if }\fo_{\F}(a)\text{ is locally finite},\\
			\infty,&\text{otherwise},
		\end{cases}
		\]
		and
		\[
		\dimbe\fo_{\F}(a)=\begin{cases}
			d_{\F}\le1,&\text{if }\fo_{\F}(a)\text{ is uniformly locally finite},\\
			\infty,&\text{otherwise}.
		\end{cases}
		\]
	\end{theorem}
	For the planar systems considered here, we shall show that the
	mass and Beurling dimensions lie between the semigroup-growth exponent
	and the affinity dimension. Moreover, the ordering between the dual
	box dimension and semigroup growth in~\eqref{df exists} is reversed
	when $m>n$. These differences reflect the unequal coordinate
	expansion rates.

	\subsection{Planar semigroup growth and affinity dimension}
	
	Write $D=\operatorname{diag}(m,n)$ and regard each $S_{(i,j)}$ as an
	affine map of $\R^2$. Let $G^*(\La)$ be the semigroup generated by
	these maps, with elements counted as distinct maps rather than as
	words. For $g\in G^*(\La)$, let $J_g$ be its linear part, and set
	\[
	G_\La(R)=\#\{g\in G^*(\La):\|J_g\|\le R\},
	\]
	where $\|\cdot\|$ is the Euclidean operator norm.
	
	\begin{proposition}\label{semigroup growth exponent}
		The semigroup $G^*(\La)$ is free, and
		\[
		d_\La:=\lim_{R\to\infty}\frac{\log G_\La(R)}{\log R}
		=\log_m\#\La.
		\]
	\end{proposition}
	
	\begin{proof}
		For $(\bi,\bj)=(i_0\ldots i_{k-1},j_0\ldots j_{k-1})\in\La^k$, the
		corresponding map is
		\[
		S_{(\bi,\bj)}(z)
		=D^kz+\left(\sum_{\ell=0}^{k-1}i_\ell m^\ell,
		\sum_{\ell=0}^{k-1}j_\ell n^\ell\right).
		\]
		If two words of length $k$ define the same map, their translation
		vectors coincide. Uniqueness of finite base-$m$ and base-$n$
		expansions then gives equality of every digit. Words of different
		lengths have different linear parts and therefore define different
		maps. Thus there are exactly $(\#\La)^k$ elements of length $k$.
		
		Since $\|D^k\|=m^k$, for $R\ge m$ we obtain
		\[
		G_\La(R)=\sum_{k=1}^{\lfloor\log_m R\rfloor}(\#\La)^k.
		\]
		If $\#\La\ge2$, the logarithm of this sum equals
		$\lfloor\log_m R\rfloor\log\#\La+O_\La(1)$.
		If $\#\La=1$, the sum is $\lfloor\log_m R\rfloor$, whose
		logarithm is $o(\log R)$. Both cases give the asserted limit.
	\end{proof}
	
	The dual maps are
	$S_{(i,j)}^{-1}(x,y)=((x-i)/m,(y-j)/n)$.
	Iterating them shows that their attractor satisfies
	\[
	K(\La^{-1})\coloneq \left\{\left(-\sum_{\ell=1}^\infty\frac{i_\ell}{m^\ell},
	-\sum_{\ell=1}^\infty\frac{j_\ell}{n^\ell}\right):
	(i_\ell,j_\ell)\in\La\right\}.
	\]
	It is therefore an isometric image of the classical
	Bedford--McMullen carpet $F_\La$.
	For its row structure, set
	\[
	M_j=\#\{i:(i,j)\in\La\},\qquad
	M_\La=\#\{j:M_j>0\},\qquad
	M_{\max}=\max_{0\le j<n}M_j.
	\]
	The classical dimension formula \cite{McMullen1984} gives
	\begin{equation}\label{dual box dimension}
		\dimB K(\La^{-1})
		=\log_m\left(\frac{\#\La}{M_\La}\right)+\log_n M_\La.
	\end{equation}
	When $m=n$, the right-hand side reduces to $\log_m\#\La$.
	
	The affinity dimension $s_{\mathrm{aff}}$ of the dual system is
	the planar counterpart of the similarity dimension in
	\eqref{df exists}. To compute it, note that the singular values of
	$D^{-1}$ are $n^{-1}$ and $m^{-1}$. Its singular value function is
	\[
	\varphi^s(D^{-1})=
	\begin{cases}
		n^{-s},&0\le s\le1,\\
		n^{-1}m^{1-s},&1\le s\le2.
	\end{cases}
	\]
	All words of length $k$ have linear part $D^{-k}$, and
	$\varphi^s(D^{-k})=\varphi^s(D^{-1})^k$. Hence the singular value
	pressure equals
	\[
	\lim_{k\to\infty}\frac1k
	\log\sum_{\omega\in\La^k}\varphi^s(D^{-k})
	=\log\#\La+\log\varphi^s(D^{-1}).
	\]
	Since $1\le\#\La\le mn$, this pressure has a unique zero in
	$[0,2]$, which is $s_{\mathrm{aff}}$. Thus
	\begin{equation}\label{dual affinity dimension}
		s_{\mathrm{aff}}=
		\begin{cases}
			\log_n\#\La,&\#\La\le n,\\[1mm]
			1+\log_m\Big(\dfrac{\#\La}{n}\Big),&\#\La\ge n.
		\end{cases}
	\end{equation}
	The two expressions agree when $\#\La=n$.
	In particular, for fixed $m,n$, both $d_\La$ and
	$s_{\mathrm{aff}}$ depend only on $\#\La$, not on the distribution
	of the digits among rows and columns.
	
	\subsection{Dimension bounds and the effect of anisotropy}

	\begin{proposition}\label{full semigroup dimension chain}
		For every $z=(u,v)\in\Z^2$,
		\begin{equation}\label{semigroup dual dimension chain}
			d_\La\le\dimM\fola(z)\le\dimbe\fola(z)
			\le\dimB K(\La^{-1})\le s_{\mathrm{aff}}.
		\end{equation}
		If $m=n$, all these quantities equal $\log_m\#\La$.
	\end{proposition}
	
	\begin{proof}
		For $1\le t\le\min\{\#\La,n\}$, set
		\[
		\chi(t)=\log_m\Big(\frac{\#\La}{t}\Big)+\log_n t
		=\log_m\#\La+
		\Big(\frac1{\log n}-\frac1{\log m}\Big)\log t.
		\]
		This function is nondecreasing.
		Every occupied column contains at most one digit in each occupied
		row. Together with the definition of $\eta(u)$, this gives
		\[
		1\le\eta(u)\le\Nm\le M_\La\le\min\{\#\La,n\}.
		\]
		If $m>n$, Proposition \ref{semigroup growth exponent},
		Theorem~\ref{main theorem m>n}, and \eqref{dual box dimension}
		identify the first four quantities in
		\eqref{semigroup dual dimension chain} as
		$\chi(1)$, $\chi(\eta(u))$, $\chi(\Nm)$, and $\chi(M_\La)$,
		respectively. Moreover, \eqref{dual affinity dimension} gives
		$\chi(\min\{\#\La,n\})=s_{\mathrm{aff}}$.
		Monotonicity of $\chi$ proves the chain.
		
		If $m=n$, the function $\chi$ is constant, with value
		$\log_m\#\La$. Theorem \ref{main theorem m=n},
		Proposition \ref{semigroup growth exponent}, and
		\eqref{dual box dimension}--\eqref{dual affinity dimension}
		give all the stated equalities.
	\end{proof}
	
	For $m>n$, the departures from the one-dimensional identities can
	be read directly from the formulae:
	\begin{equation}\label{anisotropic dimension corrections}
		\begin{aligned}
			\dimM\fola(z)-d_\La
			&=\left(\frac1{\log n}-\frac1{\log m}\right)\log\eta(u),\\
			\dimbe\fola(z)-d_\La
			&=\left(\frac1{\log n}-\frac1{\log m}\right)\log\Nm,\\
			\dimB K(\La^{-1})-d_\La
			&=\left(\frac1{\log n}-\frac1{\log m}\right)\log M_\La.
		\end{aligned}
	\end{equation}
	Thus the mass and Beurling dimensions
	equal $d_\La$ precisely when $\eta(u)=1$ and $\Nm=1$,
	respectively. More importantly, the one-dimensional inequality
	$\dimB K(\F^{-1})\le d_{\F}$ is replaced here by
	$\dimB K(\La^{-1})\ge d_\La$, with strict inequality whenever
	$M_\La>1$. The special cases of equality are retained in these
	statements; the distinction is that the one-dimensional identities
	no longer hold in general.
	
	The following example shows that all four inequalities in
	\eqref{semigroup dual dimension chain} can be strict.
	
	\begin{example}\label{strict anisotropic dimension chain}
		Take $m=6$, $n=5$, $z=(0,0)$, and
		\[
		\La=(\{0\}\times\{0,1\})\cup(\{1\}\times\{1,2,3\}).
		\]
		Here $\#\La=5$, $\eta(0)=2$, $\Nm=3$,
		and $M_\La=4$. Thus $d_\La=\log_6 5$ and $s_{\mathrm{aff}}=1$,
		while Theorem~\ref{main theorem m>n} and \eqref{dual box dimension} give
		\[
		\begin{aligned}
			\dimM\fola(z)&=\log_6\frac52+\log_5 2,\\
			\dimbe\fola(z)&=\log_6\frac53+\log_5 3,\\
			\dimB K(\La^{-1})&=\log_6\frac54+\log_5 4.
		\end{aligned}
		\]
		Since $t\mapsto\log_6(5/t)+\log_5 t$ is strictly increasing, we have
		\[
		\log_6 5=d_\La<\dimM\fola(z)<\dimbe\fola(z)
		<\dimB K(\La^{-1})<s_{\mathrm{aff}}=1.
		\]
		Figure~\ref{fig:orbit-m6-n5} shows the orbit in the window $[0,300)^2$.
		\begin{figure}[htbp]
			\centering
			\includegraphics[width=0.66\linewidth]{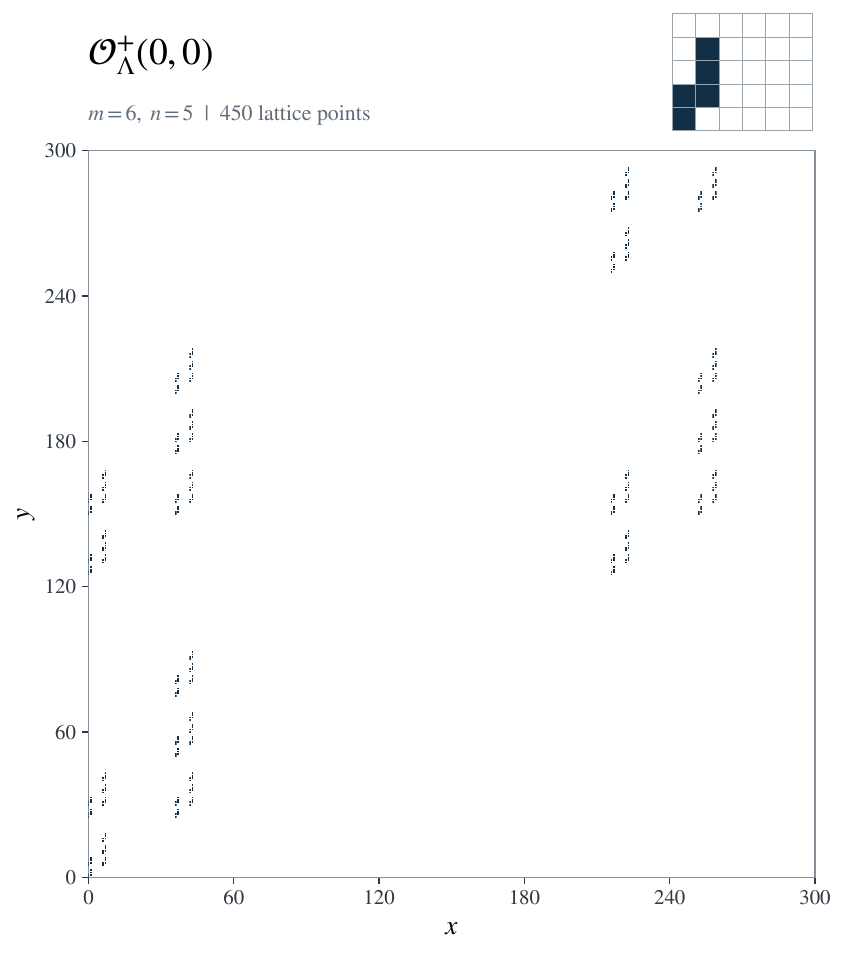}
			\caption{The orbit $\fola(0,0)$ for the digit set in
				Example~\ref{strict anisotropic dimension chain}.
				Both axes use the same scale. }
			\label{fig:orbit-m6-n5}
		\end{figure}
	\end{example}
	
	These discrepancies stem from the unequal expansion rates in the two
	coordinate directions, rather than from identifications among semigroup
	elements. At the large scale $m^k$, the finite
	model has $k$ mixed levels but $q_k=\lfloor \tau k\rfloor$
	vertical levels. The additional $q_k-k$ levels contribute the factors
	$\eta(u)^{q_k-k}$ in centered-window counts and
	$\Nm^{q_k-k}$ in maximal moving-window counts. Taking logarithms
	and dividing by $k\log m$ gives exactly the first two correction
	terms in \eqref{anisotropic dimension corrections}. The dual system has the opposite scale ordering: the horizontal
	direction contracts faster. At scale $n^{-k}$, the vertical
	coordinate requires $k$ digits, while the horizontal coordinate
	requires only $\lceil k/\tau\rceil$ digits. After the initial paired digits, it is the
	choice of rows that remains visible. This accounts for the factor
	$M_\La$ in \eqref{dual box dimension}. In contrast, the norm
	$\|D^k\|=m^k$ used in $d_\La$ records only the faster expansion
	rate, and the affinity pressure records the two contraction rates
	and the total number of maps but not their row or column placement.
	Neither scalar quantity retains the directional branching data
	needed for the orbit dimensions. When $m=n$, the two coordinate
	scales coincide and all the correction terms vanish.
	
	The column pressure from Section~\ref{section hausdorff} provides
	one way to retain these data. Namely, put
	$\mathcal P(\lambda)=\log\sum_{i\in\mI}N_i^\lambda$ for
	$\lambda\ge0$. This is $P(T,\lambda g)$ for
	$g(\bi)=\log N_{i_0}$, since its length-$k$ partition sum is
	$(\sum_{i\in\mI}N_i^\lambda)^k$. It satisfies
	\[
	\mathcal P(0)=\log N_\La,\qquad
	\mathcal P(1)=\log\#\La,\qquad
	\lim_{\lambda\to\infty}\frac{\mathcal P(\lambda)}{\lambda}
	=\log\Nm.
	\]
	Thus, together with $m,n$ and $\eta(u)$, the column pressure
	determines all the orbit dimensions in Theorem~\ref{main theorem m>n}.
	
	\subsection{Comparison with classical Bedford--McMullen carpets}
	
	For $m>n$, Table~\ref{tab:inverse-forward-dimensions} lists the
	explicit formulae for the classical dimensions of $K(\La^{-1})$
	and the corresponding discrete dimensions of $\fola(z)$.
	The classical Hausdorff and box formulae are due to Bedford and
	McMullen; see \cite{McMullen1984}. The Assouad formula is proved in
	\cite[Theorem~1.1]{Mackay2011}.

	The table makes the directional distinction explicit: the classical
	formulae involve the row counts $M_j$, whereas the discrete formulae
	involve the column counts $N_i$ and the endpoint parameter $\eta(u)$.
	In particular, the dimensions of the dual attractor do not depend on
	the starting point, while several forward-orbit dimensions do.
	
	\begin{table}[htbp]
		\centering
		\small
		\renewcommand{\arraystretch}{1.5}
		\begin{tabularx}{\textwidth}{@{}
				>{\hsize=0.9\hsize\centering\arraybackslash}X|
				>{\hsize=1.1\hsize\centering\arraybackslash}X@{}}
			\hline
			Classical carpet $K(\La^{-1})$ & Discrete orbit $\fola(z)$\\
			\hline
			\paddedmath{\begin{gathered}
					\dimH K(\La^{-1})
					=\log_n\bigg(\sum_{j:M_j>0}M_j^{1/\tau}\bigg)
			\end{gathered}}
			& \paddedmath{\begin{gathered}
					\dimdH\fola(z)\\
					=\min_{0\le\lambda\le1}\log_m\Big(
					\sum_{i:N_i>0}\Big(\frac{N_i}{\eta(u)}\Big)^\lambda\Big)+\log_n\eta(u)
			\end{gathered}}\\
			\hline
			\paddedmath{\begin{gathered}
					\dimB K(\La^{-1})
					=\log_m\Big(\frac{\#\La}{M_\La}\Big)+\log_n M_\La
			\end{gathered}}
			& \paddedmath{\begin{gathered}
					\dimM\fola(z)
					=\log_m\Big(\frac{\#\La}{\eta(u)}\Big)+\log_n\eta(u)
			\end{gathered}}\\
			\hline
			\paddedmath{\begin{gathered}
					\dimP K(\La^{-1})
					=\log_m\Big(\frac{\#\La}{M_\La}\Big)+\log_n M_\La
			\end{gathered}}
			& \paddedmath{\begin{gathered}
					\dimdP\fola(z)\\
					=\max\Big\{\log_m\Big(\frac{\#\La}{\eta(u)}\Big),
					\log_m N_\La\Big\}+\log_n\eta(u)
			\end{gathered}}\\
			\hline
			\paddedmath{\begin{gathered}
					\dima K(\La^{-1})=\log_n M_\La+\log_m M_{\max}
			\end{gathered}}
			& \paddedmath{\begin{gathered}
					\dima\fola(z)=\log_m N_\La+\log_n\Nm
			\end{gathered}}\\
			\hline
		\end{tabularx}
		\caption{Classical and discrete dimension formulae when $m>n$,
			with $\tau=\log m/\log n$.}
		\label{tab:inverse-forward-dimensions}
	\end{table}
	
	\begin{remark}
		For the compact dual attractor, the classical Assouad dimension
		describes covering behaviour at scales $0<r<R\le\operatorname{diam}
		K(\La^{-1})$ when the attractor is not a singleton; a singleton
		has dimension zero. The discrete Assouad dimension of the forward
		orbit uses $1\le r<R$ and arbitrarily large $R$. The two formulae
		therefore concern small-scale and large-scale geometry, respectively.
	\end{remark}

\end{document}